\PassOptionsToPackage{unicode}{hyperref}
\PassOptionsToPackage{hyphens}{url}
\PassOptionsToPackage{dvipsnames,svgnames,x11names}{xcolor}
\documentclass[
  12pt]{article}

\usepackage{graphicx}
\usepackage{float}
\restylefloat{table}
\usepackage{subcaption}
\usepackage{array, booktabs, makecell, multirow, longtable}
\usepackage{amsmath, amssymb, amsthm}

\usepackage{mathtools}                  
\usepackage{mathrsfs}                  
\usepackage{bm}              
\usepackage{makecell}
\usepackage{esvect}                     
\usepackage{cancel}                     
\usepackage{latexsym, amsfonts, amstext}
\usepackage{enumitem}
\newtheorem{theorem}{Theorem}
\newtheorem{lemma}{Lemma}

\newtheorem{assumption}{Assumption}
\newtheorem{proposition}{Proposition}
\newtheorem{remark}{Remark}

\DeclareMathOperator*{\argmin}{arg\,min}

\usepackage{iftex}
\ifPDFTeX
  \usepackage[T1]{fontenc}
  \usepackage[utf8]{inputenc}
  \usepackage{textcomp} % provide euro and other symbols
\else % if luatex or xetex
  \usepackage{unicode-math}
  \defaultfontfeatures{Scale=MatchLowercase}
  \defaultfontfeatures[\rmfamily]{Ligatures=TeX,Scale=1}
\fi
\usepackage{lmodern}
\ifPDFTeX\else  
\fi

\IfFileExists{upquote.sty}{\usepackage{upquote}}{}
\IfFileExists{microtype.sty}{% use microtype if available
  \usepackage[]{microtype}
  \UseMicrotypeSet[protrusion]{basicmath} % disable protrusion for tt fonts
}{}
\makeatletter
\@ifundefined{KOMAClassName}{% if non-KOMA class
  \IfFileExists{parskip.sty}{%
    \usepackage{parskip}
  }{% else
    \setlength{\parindent}{0pt}
    \setlength{\parskip}{6pt plus 2pt minus 1pt}}
}{% if KOMA class
  \KOMAoptions{parskip=half}}
\makeatother
\usepackage{xcolor}
\makeatletter
\ifx\paragraph\undefined\else
  \let\oldparagraph\paragraph
  \renewcommand{\paragraph}{
    \@ifstar
      \xxxParagraphStar
      \xxxParagraphNoStar
  }
  \newcommand{\xxxParagraphStar}[1]{\oldparagraph*{#1}\mbox{}}
  \newcommand{\xxxParagraphNoStar}[1]{\oldparagraph{#1}\mbox{}}
\fi
\ifx\subparagraph\undefined\else
  \let\oldsubparagraph\subparagraph
  \renewcommand{\subparagraph}{
    \@ifstar
      \xxxSubParagraphStar
      \xxxSubParagraphNoStar
  }
  \newcommand{\xxxSubParagraphStar}[1]{\oldsubparagraph*{#1}\mbox{}}
  \newcommand{\xxxSubParagraphNoStar}[1]{\oldsubparagraph{#1}\mbox{}}
\fi
\makeatother

\usepackage{longtable,booktabs,array}
\usepackage{calc} 
\usepackage{etoolbox}
\makeatletter
\patchcmd\longtable{\par}{\if@noskipsec\mbox{}\fi\par}{}{}
\makeatother

\IfFileExists{footnotehyper.sty}{\usepackage{footnotehyper}}{\usepackage{footnote}}
\makesavenoteenv{longtable}
\usepackage{graphicx}
\makeatletter
\def\maxwidth{\ifdim\Gin@nat@width>\linewidth\linewidth\else\Gin@nat@width\fi}
\def\maxheight{\ifdim\Gin@nat@height>\textheight\textheight\else\Gin@nat@height\fi}
\makeatother

\setkeys{Gin}{width=\maxwidth,height=\maxheight,keepaspectratio}

\makeatletter
\def\fps@figure{htbp}
\makeatother

\makeatletter
\@ifpackageloaded{caption}{}{\usepackage{caption}}
\AtBeginDocument{%
\ifdefined\contentsname
  \renewcommand*\contentsname{Table of contents}
\else
  \newcommand\contentsname{Table of contents}
\fi
\ifdefined\listfigurename
  \renewcommand*\listfigurename{List of Figures}
\else
  \newcommand\listfigurename{List of Figures}
\fi
\ifdefined\listtablename
  \renewcommand*\listtablename{List of Tables}
\else
  \newcommand\listtablename{List of Tables}
\fi
\ifdefined\figurename
  \renewcommand*\figurename{Figure}
\else
  \newcommand\figurename{Figure}
\fi
\ifdefined\tablename
  \renewcommand*\tablename{Table}
\else
  \newcommand\tablename{Table}
\fi
}
\@ifpackageloaded{float}{}{\usepackage{float}}
\floatstyle{ruled}
\@ifundefined{c@chapter}{\newfloat{codelisting}{h}{lop}}{\newfloat{codelisting}{h}{lop}[chapter]}
\floatname{codelisting}{Listing}

\makeatother
\makeatletter
\@ifpackageloaded{caption}{}{\usepackage{caption}}
\@ifpackageloaded{subcaption}{}{\usepackage{subcaption}}
\makeatother

\ifLuaTeX
  \usepackage{selnolig}  
\fi
\usepackage[]{natbib}
\usepackage{bookmark}

\IfFileExists{xurl.sty}{\usepackage{xurl}}{}
\hypersetup{
  pdftitle={Distributionally Robust PCA with Data-Adaptive Wasserstein Geometry},
  pdfauthor={Chuang Xu; Andrew T. A. Wood; Yanrong Yang},
  pdfkeywords={principal component analysis; distributionally robust optimization; Wasserstein ambiguity set; adaptive transport geometry; robust Wasserstein profile inference},
  colorlinks=true,
  linkcolor={blue},
  citecolor={blue},
  urlcolor={blue}
}

\newcommand{\anon}{1}

\begin{document}

\def\spacingset#1{\renewcommand{\baselinestretch}%
{#1}\small\normalsize} \spacingset{1}

\if1\anon
{
  \title{\bf Distributionally Robust PCA with Data-Adaptive Wasserstein Geometry}
  \author{Chuang Xu\thanks{The authors are grateful to the Australian Research Council for supporting this research through grant DP220102232\hspace{.2cm}},\, \, Andrew T. A. Wood and  Yanrong Yang\\
    Research School of Finance, Actuarial Studies and Statistics, \\
    Australian National University }
  \maketitle
} \fi

\if0\anon
{
  \bigskip
  \bigskip
  \bigskip
  \begin{center}
    {\LARGE\bf Title}
\end{center}
  \medskip
} \fi

\bigskip
\begin{abstract}
We develop a distributionally robust formulation of principal component analysis that minimizes worst-case reconstruction risk over distributions lying within a Wasserstein neighborhood of the empirical measure. 
The Wasserstein neighborhood, viewed as an ambiguity set of distributions, is adaptively calibrated through a transport matrix $G$ to capture heterogeneous uncertainty across dimensions. The homogeneous case, in which \(G\) is a scalar multiple of the identity matrix, recovers classical PCA.
Under a general transport matrix \(G\), we derive a dual characterization of the associated minimax optimization problem and introduce a tractable surrogate objective function consisting of the square-root empirical reconstruction error plus a geometry-dependent residual exposure penalty. 
The exact and surrogate estimators are shown to be consistent for the population PCA subspace and asymptotically equivalent at the projector level.
The transport geometry is allowed to be data adaptive, while the Wasserstein radius is calibrated via robust Wasserstein profile inference, yielding a data-driven radius of order \(n^{-1/2}\). 
Comprehensive theoretical guarantees are established, including consistency and local Grassmannian asymptotics exhibiting an explicit Wasserstein-induced drift determined by the limiting transport geometry and calibration level. 
Numerical experiments and a real-data application demonstrate that the proposed method can substantially improve finite-sample out-of-sample performance under structured covariance shifts, moderate contamination, and certain same-distribution regimes.
\end{abstract}

\noindent%
\vfill

\newpage
\spacingset{1.8} % DON'T change the spacing!

\section{Introduction}
\label{sec:introduction}

Principal component analysis (PCA) is one of the most widely used techniques for dimension reduction. It seeks a low-dimensional linear subspace that optimally represents the variation in original data, typically by maximizing explained variance or, equivalently, minimizing reconstruction error. Let \(X\in\mathbb R^p\) be centered and the
covariance matrix be \(\Sigma\), and if \(U\in\mathbb V_{p,r}:=\{A\in\mathbb R^{p\times r}:A^\top A=I_r\}\) is an
orthonormal basis for an \(r\)-dimensional candidate principal subspace, then the population residual risk is
\[
    \mathbb E\{\|(I_p-UU^\top)X\|_2^2\}
    =
    \operatorname{tr}\{\Sigma(I_p-UU^\top)\}.
\]
In practice, PCA replaces the population covariance matrix \(\Sigma\) with its sample counterpart \(\widehat{\Sigma}_n\). The asymptotic properties of the resulting estimator are well understood in the fixed-dimensional regime; see \citet{anderson1963}, and extensive work has also studied in high-dimensional and functional settings  
\citep{johnstone2009,koltchinskii2017}. Despite its optimality under the empirical distribution, ordinary PCA does not explicitly model uncertainty in the underlying data-generating mechanism. As a plug-in estimator based on the sample covariance matrix, it is inherently sensitive to sampling variability, contamination, and distributional shifts between the training and deployment environments. These issues are especially pronounced when eigengaps are small, when the covariance structure contains moderate spurious directions, or when future perturbations occur primarily in low-variance directions that are disregarded by classical PCA. 

A substantial literature has developed robust alternatives to principal component analysis. Classical robust PCA methods either replace the empirical covariance matrix with a robust scatter estimator or define principal directions through projection pursuit. Early projection-pursuit approaches to robust dispersion matrices and principal components were introduced by \citet{li1985projection}, while \citet{croux2000principal} studied PCA based on robust covariance and correlation estimators, including analyses of influence functions and efficiency. The ROBPCA procedure of \citet{hubert2005new}, which combines projection pursuit with robust scatter estimation, has become a widely used method for outlier-resistant PCA. Another robustification strategy is based on Cauchy likelihood; see \citet{FPTW24}.
A different influential line of work uses the term robust PCA for low-rank-plus-sparse matrix decomposition. The seminal work of \citet{candes2011} shows that a low-rank matrix can be exactly recovered from observations contaminated by sparse but potentially large corruptions. Robust PCA methods have also been developed for functional data; see,
for example, the projection-pursuit approach of \citet{bali2011robust}. 
These contributions address important forms of robustness, including outliers, heavy tails, sparse corruptions and functional trajectory contamination. However, their primary objective is
usually robust estimation of a covariance structure, a principal subspace, or a low-rank signal. In contrast, we aim to a general formulation which treats robustness as protection against a class of nearby distributions around the empirical distribution.

Distributionally robust optimization (DRO) provides a complementary robustness principle. Rather than replacing a sample quantity by a robust analogue, DRO optimizes the worst-case expected loss over an ambiguity set of probability distributions. Moment-based ambiguity sets were studied by \citet{delage2010distributionally}, while Wasserstein ambiguity sets have become central in data-driven DRO because they provide a geometric framework for modelling distributional perturbations around the empirical distribution \citep{esfahani2018,blanchet2019modelrisk,gao2023}. A key modelling choice is the transport cost, which determines the geometry of the ambiguity set. Most Wasserstein DRO formulations fix this geometry a priori and focus on tractability, guarantees, or radius calibration, whereas \citet{blanchet2019cost} show that learning transport costs can induce adaptive regularisation. This suggests that the ambiguity-set geometry may play an important statistical role: the radius controls the magnitude of distributional perturbations, while the transport cost determines which perturbation directions are considered plausible.
Wasserstein DRO is closely connected to regularisation, generalisation and out-of-sample stability. For example, \citet{wu2025generalization} establish broad regularisation and generalisation results, \citet{gao2024variation} connect Wasserstein DRO with variation regularisation, and \citet{gao2023finite} study finite-sample guarantees. In portfolio optimisation, \citet{blanchet2022distributionally} develop a Wasserstein-robust mean--variance framework with data-driven uncertainty calibration, while related extensions consider high-dimensional factor portfolios \citep{wu2025highdimportfolio} and downside-risk optimisation under flexible risk measures and transport costs \citep{liu2026downside}. More generally, dimensionality reduction has been proposed as a tool for improving tractability in high-dimensional DRO problems \citep{jiang2025optimized}.
From a statistical perspective, robust Wasserstein profile inference (RWPI) calibrates Wasserstein ambiguity sets through profile statistics \citep{blanchet2019rwpi}, while related work develops confidence regions and asymptotic inference for Wasserstein DRO estimators \citep{blanchet2022confidence}. More broadly, DRO has been linked to generalized empirical likelihood and uniform performance guarantees under distributional perturbations \citep{duchi2021statistics,duchi2021learning}. 

A closely related contribution is \citet{wang2025enhancing}, who study PCA under a DRO formulation based on the standard Euclidean transport cost and show that it is equivalent to ordinary PCA. Motivated by this observation, we investigate the statistical role of anisotropic and data-adaptive transport geometry in Wasserstein DRO for PCA, together with RWPI-based radius calibration and local asymptotic analysis. Despite the growing DRO literature, statistically calibrated Wasserstein robustness for unsupervised PCA remains comparatively underexplored, particularly when the transport geometry is itself estimated from the data. 
The notion of robustness considered here differs from classical outlier resistance. Traditional robust PCA methods reduce the influence of atypical observations through robust scatter estimators, loss functions, or decomposition models. By contrast, distributional robustness treats the empirical distribution as a reference measure and seeks a subspace whose reconstruction risk remains controlled over a neighbourhood of nearby distributions. This reflects the broader distinction between classical robust statistics, which addresses data contamination, and DRO, which protects against distributional shifts between training and deployment environments \citep{blanchet2025drostatistics}. 
Ordinary PCA can also be sensitive to heterogeneous noise variances. For example, heteroskedastic PCA corrects coordinate-dependent noise in spiked covariance models \citep{zhang2022heteroskedastic}. By contrast, the transport geometry in our Wasserstein DRO formulation encodes heterogeneous distributional uncertainty rather than variance heterogeneity, providing robustness to distributional perturbations beyond classical heteroscedastic modelling.

This gap is not merely a matter of extending an existing Wasserstein DRO construction to an unsupervised setting. Under the standard Euclidean transport cost, the Wasserstein ambiguity set represents homogeneous uncertainty across directions, and the resulting robust optimization problem preserves the classical PCA ordering of subspaces.
This observation shifts the focus from the size of the Wasserstein ambiguity set to its geometry. While the Euclidean transport cost provides a natural homogeneous-uncertainty benchmark, it is incapable of producing a non-trivial robustification of PCA. The relevant statistical object is therefore the transport geometry itself. We propose to model this geometry through a positive definite transport matrix $G$, which specifies the relative cost of perturbations in different directions and thereby encodes heterogeneous distributional uncertainty. The induced Wasserstein DRO criterion regularizes PCA by controlling residual exposure along uncertainty-sensitive directions, while a locally calibrated Wasserstein radius captures the overall scale of uncertainty.

This paper develops a weighted Wasserstein DRO formulation of PCA. Let
\(X_1,\ldots,X_n\in\mathbb R^p\) be centred observations and write \(    \widehat{\mathbb P}_n
    =
    \frac1n\sum_{i=1}^n\delta_{X_i},
    \widehat\Sigma_n
    =
    \frac1n\sum_{i=1}^n X_iX_i^\top .\)
For a candidate rank-\(r\) principal basis \(U\in\mathbb V_{p,r}\), let \(    \Pi_U:=UU^\top,
    Q_U:=I_p-\Pi_U,\)
where \(\Pi_U\) is the principal projector and \(Q_U\) is the residual projector of rank
\(s=p-r\). For a positive definite matrix \(G\), define the weighted quadratic transport
cost \(    c_G(x,y)=(x-y)^\top G(x-y),\)
and let \(W_{2,G}\) be the corresponding Wasserstein distance, defined below in \eqref{eq:Wasserstein_distance_eq}. For \(Q\in\mathcal G_s\),
where \(\mathcal G_s\) denotes the Grassmannian represented as rank-\(s\) orthogonal
projectors, we define the exact robust residual risk
\[
        \Phi_{n,\delta}^G(Q)
    :=
    \sup_{\mathbb P:\,W_{2,G}(\mathbb P,\widehat{\mathbb P}_n)\le \delta}
    \mathbb E_{Z\sim\mathbb P}(Z^\top QZ).
\]
The weighted Wasserstein DRO-PCA problem is then to minimise
\(\Phi_{n,\delta}^G(Q)\) over \(Q\in\mathcal G_s\).

The Euclidean transport cost provides a natural benchmark corresponding to homogeneous uncertainty across directions. In this case, Wasserstein robustification is degenerate for PCA: the exact worst-case reconstruction risk is a monotone transformation of the empirical reconstruction error and therefore yields exactly the same principal subspace. A related equivalence result for Euclidean Wasserstein DRO-PCA was obtained by \citet{wang2025enhancing}. Consequently, the key modelling ingredient is not the Wasserstein radius but the transport geometry. The positive definite matrix $G$ governs this geometry by determining which perturbation directions are relatively inexpensive for the adversary, thereby controlling how robustness reshapes the fitted subspace. The relevant directional quantity is the residual
exposure \(    \rho_G(Q)
    :=
    \lambda_{\max}\{G^{-1/2}QG^{-1/2}\}.\)
Directions that are inexpensive under the transport cost are amplified by \(G^{-1/2}\).
Residual subspaces aligned with such directions therefore have larger worst-case exposure.
Wasserstein duality yields an exact dual representation of \(\Phi_{n,\delta}^G\) and
motivates the tractable surrogate criterion
\begin{align}
        \mathcal J_{n,\delta}^G(Q)
    :=
    \sqrt{\operatorname{tr}(\widehat\Sigma_nQ)}
    +
    \delta\sqrt{\rho_G(Q)}. \label{surrogate in intro}
\end{align}
The first term is the square root empirical reconstruction error, while the second term is
a geometry-dependent exposure penalty. A key feature of this formulation is that the transport matrix \(G\) is not merely a
technical device. It determines which perturbation directions are inexpensive for the
adversary and therefore which residual directions are penalised by the DRO criterion.
While much of the Wasserstein DRO literature focuses on the ambiguity radius, the choice
of the transport cost is itself a modelling problem. We therefore allow the geometry to be
estimated from the data, writing \(\widehat G_n\), and develop the theory under general
operator-norm consistency and spectral stability conditions.

The second ingredient is data-driven calibration of the Wasserstein radius. The radius
\(\delta\) should be large enough to reflect sampling uncertainty, but not so large that
the adversarial term dominates the empirical reconstruction term. We calibrate the radius through an RWPI statistic associated with the PCA first-order
condition.
For a data-adaptive transport matrix \(\widehat G_n\), the RWPI statistic is defined as
the smallest squared \(W_{2,\widehat G_n}\)-perturbation of the empirical distribution
under which the defined estimating equation holds at the population subspace. Its limiting law
determines a radius \(\widehat\delta_{n,\alpha}\) of order \(n^{-1/2}\). This calibration
places the robustification on the local statistical scale: the Wasserstein penalty is
asymptotically small enough to preserve the population PCA target, but large enough to
affect the root-\(n\) local behaviour of the fitted subspace.

Note that the proposed method is designed to achieve distributional robustness in local neighborhood of the empirical distribution, defined in terms of the Wasserstein distance.  Distributional robustness does not include robustness against distributions with arbitrarily heavy tails , because the Wasserstein distance (\ref{eq:Wasserstein_distance_eq}) does not remain bounded in this setting. Nevertheless, the proposed framework achieves distributional robustness over a broad class of local perturbations characterized by bounded Wasserstein distance from the reference distribution, thereby providing protection against a wide range of realistic model mis-specifications.

The main contributions of the paper are as follows.

\begin{enumerate}[label=(\roman*)]

\item \textit{Methodology}. We formulate PCA as a weighted Wasserstein DRO problem with transport matrix $G$. The transport geometry encodes heterogeneous distributional uncertainty, and Euclidean geometry is shown to be degenerate for PCA. For general $G$, we derive a dual representation of the robust risk and an interpretable upper bound, leading to the surrogate objective \eqref{surrogate in intro}, which provides a tractable directional regularization of PCA.

\item \textit{Data-adaptivity}. We introduce a data-adaptive transport geometry $\widehat G_n$ together with an RWPI-calibrated radius $\widehat\delta_{n,\alpha}$. Under operator-norm consistency and spectral stability of $\widehat G_n$, the adaptive profile statistic converges to a weighted chi-square distribution. This yields a fully data-driven radius calibration that avoids cross-validation and externally specified regularization parameters.

\item \textit{Theory}. Fundamental asymptotic theory is established on the Grassmannian. The exact adaptive DRO estimator is consistent, while the adaptive surrogate estimator satisfies a root-$n$ projector central limit theorem. Relative to classical PCA, the limiting distribution contains an explicit Wasserstein-induced drift determined by the transport geometry and RWPI calibration. 

\item \textit{Finite-sample performance}. We evaluate the proposed method through simulations and real-data analysis. Across settings involving out-of-sample reconstruction, covariance shift, training contamination, and image-region transfer, the results show that adaptive Wasserstein geometry can improve reconstruction when aligned with the underlying perturbation structure. The experiments identify the settings in which calibrated Wasserstein regularization is most beneficial.

\end{enumerate}

In the numerical studies, we consider two diagonal data-adaptive transport geometries. The first, a residual-variance geometry, assigns lower transport cost to coordinates with large variation unexplained by an initial PCA fit. The second, a PCA-block variance geometry, assigns lower transport cost to coordinates with large variation captured by the leading empirical principal components. These choices reflect different assumptions about whether target-relevant directions lie in the residual or PCA subspace.
More generally, the theory applies to any adaptive transport geometry $\widehat G_n$ that converges in operator norm to a positive definite limit. The transport geometry can therefore be chosen to incorporate domain knowledge about plausible distributional perturbations.

For a matrix \(A\), \(\|A\|_{\operatorname{op}}\) and \(\|A\|_F\) denote the operator and
Frobenius norms. If \(A\) is a symmetric matrix, \(\lambda_{\min}(A)\) and
\(\lambda_{\max}(A)\) denote its smallest and largest eigenvalues;
\(A\succeq0\) means that \(A\) is positive semidefinite; and \(A\preceq B\) means
\(B-A\succeq0\). Convergence in
probability and convergence in distribution are denoted by \(\xrightarrow{p}\) and
\(\Rightarrow\), respectively, and equality in distribution by \(\overset{d}{=}\). We use
\(O_p(\cdot)\) and \(o_p(\cdot)\) in their usual stochastic-order senses.  We write $A=O_p(\cdot)$ or $A=o_p(\cdot)$ for a matrix $A$, if $\vert \vert A \vert \vert_F=O_p(\cdot)$ or $\vert \vert A \vert \vert_F=o_p(\cdot)$, respectively.

The rest of the paper is organised as follows. Section~\ref{sec:dro_pca} introduces the
weighted Wasserstein DRO-PCA formulation, records the Euclidean degeneracy result in our formulation, and
derives the surrogate criterion. Section~\ref{sec:radius_calibration} develops the RWPI
calibration of the Wasserstein radius under an adaptive transport metric.
Section~\ref{subsec:main_adaptive_clt} gives the consistency, local quadratic expansion,
and Grassmannian central limit theorem for the adaptive surrogate estimator.
Section~\ref{sec:adaptive_transport_geometry} discusses practical choices of the
transport geometry and summarises the adaptive DRO-PCA implementation. Section~\ref{sec:numerical} reports simulation and Section~\ref{subsec:image_segmentation_data} reports real-data evidence.
The supplementary material contains proofs and additional numerical details.
\section{A weighted Wasserstein formulation of DRO-PCA}\label{sec:dro_pca}
Let \(X_1,\ldots,X_n\in\mathbb R^p\) be centred observations, and write \(    \widehat{\mathbb P}_n:=\frac1n\sum_{i=1}^n\delta_{X_i},
    \widehat\Sigma_n:=\frac1n\sum_{i=1}^n X_iX_i^\top .\)
\begin{remark}
\label{rem:centring_convention}
Throughout the paper we assume, for notational convenience, that the observations have
been centred. If the raw observations are not centred, the procedure may be applied after
empirical centring. This
empirical centring has  asymptotically negligible effect on all of our results under suitable moment conditions.
\end{remark}
For \(m\in\{1,\ldots,p-1\}\), let \(    \mathcal G_m
    :=
    \{Q\in\mathbb R^{p\times p}: Q^\top=Q,\ Q^2=Q,\ \operatorname{tr}(Q)=m\}\)
denote the Grassmannian represented as rank-\(m\) orthogonal projectors. We fix a target
principal dimension \(r<p\), set \(s:=p-r\), and write
\[
    \mathbb V_{p,r}:=\{U\in\mathbb R^{p\times r}:U^\top U=I_r\},
    \qquad
    \Pi_U:=UU^\top,
    \qquad
    Q_U:=I_p-\Pi_U .
\]
Thus \(\Pi_U\in\mathcal G_r\) is the candidate principal projector and
\(Q_U\in\mathcal G_s\) is its residual projector. Classical PCA solves \(\min_{U\in\mathbb V_{p,r}}\operatorname{tr}(\widehat\Sigma_nQ_U),\)
or equivalently maximises \(\operatorname{tr}(\widehat\Sigma_n\Pi_U)\).

The distributionally robust version replaces the empirical reconstruction error by a
worst-case reconstruction error over a neighbourhood of the empirical law. For a symmetric positive definite matrix \(G\), define \(    \|x\|_G^2:=x^\top Gx,
    c_G(x,y):=\|x-y\|_G^2,\)
and let \(W_{2,G}\) be the corresponding Wasserstein distance,
\begin{equation}
\label{eq:Wasserstein_distance_eq}
    W_{2,G}(\mu,\nu)
\coloneqq
\left(
  \inf_{\pi \in \Pi(\mu,\nu)}
  \int
    \|x-y\|_G^2\, \mathrm{d}\pi(x,y)
\right)^{1/2},
\end{equation}

where \(\Pi(\mu,\nu)\) denotes the set of couplings of \(\mu\) and \(\nu\).
The weighted
Wasserstein ambiguity set is \(    \mathcal U_\delta^G(\widehat{\mathbb P}_n)
    :=
    \{\mathbb P:W_{2,G}(\mathbb P,\widehat{\mathbb P}_n)\le\delta\}.\)
    Throughout, when writing \(\mathbb E_{\mathbb P} f(Z)\), we understand \(Z\) to be distributed
according to \(\mathbb P\).
For \(Q\in\mathcal G_s\), define the exact robust residual risk
\begin{equation}
\label{eq:main_exact_weighted_risk}
    \Phi_{n,\delta}^G(Q)
    :=
    \sup_{\mathbb P\in\mathcal U_\delta^G(\widehat{\mathbb P}_n)}
    \mathbb E_{\mathbb P}(Z^\top QZ).
\end{equation}
The weighted Wasserstein DRO-PCA problem is then
\begin{equation}
\label{eq:main_dro_pca_weighted}
    \inf_{Q\in\mathcal G_s}\Phi_{n,\delta}^G(Q)
    =\inf_{U\in\mathbb V_{p,r}}
    \sup_{\mathbb P\in\mathcal U_\delta^G(\widehat{\mathbb P}_n)}
    \mathbb E_{\mathbb P}\{\|Q_UZ\|_2^2\}.
\end{equation}
This formulation separates two choices: the radius \(\delta\), which determines the size of
the ambiguity set, and the matrix \(G\), which determines the geometry of admissible
distributional perturbations. Figure~\ref{fig:intuitive_transport_geometry} gives a two-dimensional schematic of the
role of \(G\): changing \(G\) changes the directions in which Wasserstein perturbations are
cheap or expensive.
\begin{figure}[t]
\centering
\includegraphics[width=0.95\textwidth]{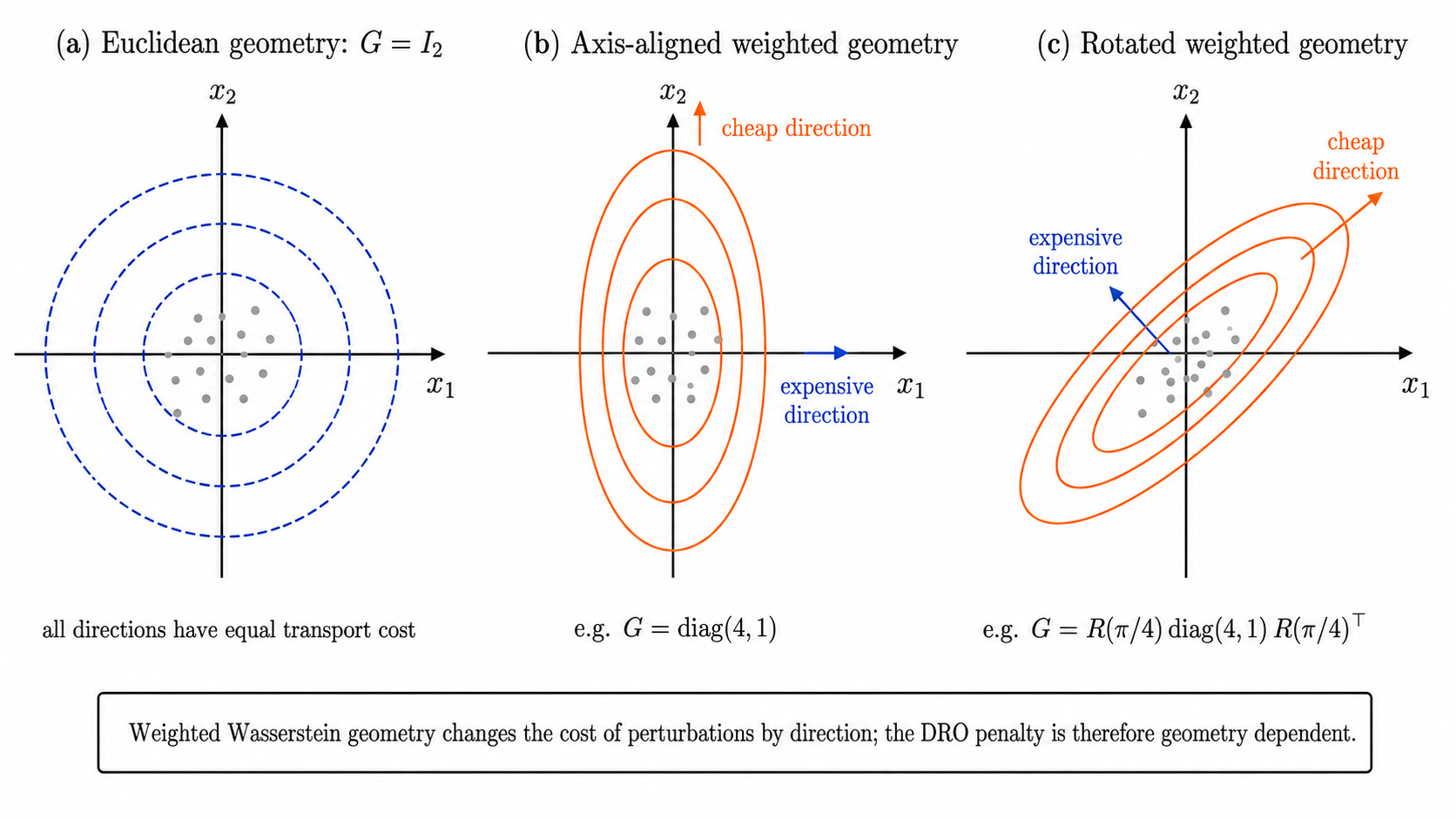}
\caption{Schematic illustration of the weighted Wasserstein transport geometry. Panel (a) shows
the Euclidean case \(G=I_2\), where all directions have the same transport cost. Panel
(b) illustrates an axis-aligned weighted geometry, for example
\(G=\operatorname{diag}(4,1)\), where perturbations in the second coordinate direction are
cheaper. Panel (c) illustrates a rotated weighted geometry, for example
\(G=R(\pi/4)\operatorname{diag}(4,1)R(\pi/4)^\top\), where the cheap and expensive
perturbation directions are determined by the eigenvectors of \(G\). The figure
illustrates that the matrix \(G\) changes the cost of perturbations by direction.}
\label{fig:intuitive_transport_geometry}
\end{figure}
The following theorem characterizes the homogeneous-geometry benchmark in our distributionally robust framework. Under the standard Euclidean transport geometry, the Wasserstein robust objective is equivalent to classical PCA and therefore induces the same ordering of subspaces. Related equivalence results for Euclidean Wasserstein DRO-PCA were also established by \citet{wang2025enhancing}. In the
Euclidean case \(G=I_p\), we write \(W_2:=W_{2,I_p}\).

\begin{theorem}
\label{thm:main_euclidean_exact}
Let \(X_1,\ldots,X_n\in\mathbb R^p\) be centred observations, define \(    \widehat{\mathbb P}_n=n^{-1}\sum_{i=1}^n\delta_{X_i},
    \widehat\Sigma_n=n^{-1}\sum_{i=1}^n X_iX_i^\top ,\)
and let \(1\le r<p\). Let \(\mathcal P_2(\mathbb R^p)\) denote the set of probability measures on
\(\mathbb R^p\) with finite second moment. For every \(\delta\ge0\) and every \(U\in\mathbb V_{p,r}\),
\[
    \sup_{\mathbb P\in\mathcal P_2(\mathbb R^p):
    W_2(\mathbb P,\widehat{\mathbb P}_n)\le\delta}
    \mathbb E_{\mathbb P}\{\|Q_UZ\|_2^2\}
    =
    \left\{\sqrt{\operatorname{tr}(\widehat\Sigma_nQ_U)}+\delta\right\}^2 .
\]
Consequently,
\[
    \argmin_{U\in\mathbb V_{p,r}}
    \sup_{\mathbb P\in\mathcal P_2(\mathbb R^p):
    W_2(\mathbb P,\widehat{\mathbb P}_n)\le\delta}
    \mathbb E_{\mathbb P}\{\|Q_UZ\|_2^2\}
    =
    \argmin_{U\in\mathbb V_{p,r}}\, 
    \operatorname{tr}(\widehat\Sigma_nQ_U).
\]
\end{theorem}
\begin{remark}
\label{rem:main_euclidean_interpretation}
Theorem \ref{thm:main_euclidean_exact} shows that Euclidean
Wasserstein robustification applies a monotone transformation to the empirical reconstruction
error and therefore preserves the ordinary PCA ordering of subspaces. Thus, classical PCA arises as a special case of Wasserstein DRO corresponding to a homogeneous transport geometry, where all perturbation directions are treated symmetrically. The weighted Wasserstein formulation considered here generalizes this benchmark by introducing a positive definite transport matrix $G$ that specifies an anisotropic geometry of distributional uncertainty and permits distributional robustness to act differently across directions.
\end{remark}
We now return to a general positive definite matrix \(G\). The role of \(G\) is most clearly
seen through the residual exposure operator
\begin{align}
        \mathsf S_G(Q):=G^{-1/2}QG^{-1/2},
    \qquad
    \rho_G(Q):=\lambda_{\max}\{\mathsf S_G(Q)\}, \label{def:exposure operator}
    \qquad Q\in\mathcal G_s.
\end{align}
We also write $\widetilde\Sigma_{n,G}:=G^{1/2}\widehat\Sigma_nG^{1/2}$.
The exposure term can be understood through the Rayleigh--Ritz characterisation
\(
    \rho_G(Q)
    =
    \sup_{\|v\|_2=1}v^\top G^{-1/2}QG^{-1/2}v .
\)
For example, if \(G=\operatorname{diag}(g_1,\ldots,g_p)\), then for the \(j\)th standard
coordinate vector \(e_j\),
\[
    \rho_G(Q)\ge g_j^{-1}e_j^\top Qe_j
    =
    g_j^{-1}\|Qe_j\|_2^2 .
\]
Thus a coordinate with small transport cost \(g_j\) receives a large penalty if it has
substantial projection onto the residual subspace. The surrogate \eqref{eq:main_surrogate_objective} defined later therefore encourages the
fitted principal subspace to capture directions that are inexpensive for the adversary to
perturb.
By standard Wasserstein DRO duality
\citep[e.g.][]{blanchet2019modelrisk,gao2023}, followed by a direct quadratic
calculation, we obtain
\[
    \Phi_{n,\delta}^G(Q)
    =
    \inf_{\lambda>\rho_G(Q)}
    \left[
    \lambda\delta^2
    +
    \lambda\operatorname{tr}
    \left\{
    \widetilde\Sigma_{n,G}
    \mathsf S_G(Q)
    \bigl(\lambda I_p-\mathsf S_G(Q)\bigr)^{-1}
    \right\}
    \right],
\]
with the detailed derivation given in Supplementary
Section~\ref{supp:sec_weighted_dual}. Although this
dual representation is exact, it is most useful for our purposes because it
makes explicit how the empirical second moment \(\widetilde\Sigma_{n,G}\) interacts with
 \(\mathsf S_G(Q)\). To obtain a tractable
criterion, we bound the resolvent term by its worst spectral value. Since the eigenvalues
of \(\mathsf S_G(Q)\) are bounded above by \(\rho_G(Q)\), for
\(\lambda>\rho_G(Q)\),
\[
    \mathsf S_G(Q)\{\lambda I_p-\mathsf S_G(Q)\}^{-1}
    \preceq
    \frac{1}{\lambda-\rho_G(Q)}\mathsf S_G(Q).
\]
Together with \(    \operatorname{tr}\{\widetilde\Sigma_{n,G}\mathsf S_G(Q)\}
    =
    \operatorname{tr}(\widehat\Sigma_nQ),\)
this reduces the exact dual objective to a one-dimensional upper envelope depending only
on the empirical reconstruction error and the residual exposure \(\rho_G(Q)\). Optimising
this envelope over \(\lambda\) gives the following bound.

\begin{theorem}
\label{thm:main_weighted_upper}
Let \(X_1,\ldots,X_n\in\mathbb R^p\) be centred observations, let
\(\widehat{\mathbb P}_n=n^{-1}\sum_{i=1}^n\delta_{X_i}\), \(\widehat\Sigma_n=n^{-1}\sum_{i=1}^n X_iX_i^\top \), let G be symmetric positive definite matrix, and let
\(\delta\ge0\). Then, for every \(Q\in\mathcal G_s\),
\[
    \Phi_{n,\delta}^G(Q)
    \le
    \left\{
    \sqrt{\operatorname{tr}(\widehat\Sigma_nQ)}
    +
    \delta\sqrt{\rho_G(Q)}
    \right\}^2 ,
\]
where $\rho_G(Q)$ is defined in \eqref{def:exposure operator}.
\end{theorem}

Theorem~\ref{thm:main_weighted_upper} shows that, unlike the Euclidean case, the transport
geometry now affects the relative ranking of residual subspaces through \(\rho_G(Q)\),
thereby inducing a genuine regularisation effect.

The bound in Theorem~\ref{thm:main_weighted_upper}
suggests the surrogate criterion
\begin{equation}
\label{eq:main_surrogate_objective}
    \mathcal J_{n,\delta}^G(Q)
    :=
    \sqrt{\operatorname{tr}(\widehat\Sigma_nQ)}
    +
    \delta\sqrt{\rho_G(Q)},
    \qquad Q\in\mathcal G_s .
\end{equation}
Since the map \(x\mapsto x^2\) is increasing on \([0,\infty)\), minimising
\(\mathcal J_{n,\delta}^G(Q)\) is equivalent to minimising the upper bound in
Theorem~\ref{thm:main_weighted_upper}. We therefore use
\(\mathcal J_{n,\delta}^G\) as the tractable criterion for computation and for the local
asymptotic analysis developed below. Numerical implementation details are given in
Supplementary Section~\ref{sec:supp_numerical_optimisation}.

When \(G\) is clear from context, we write \(\mathcal J_{n,\delta}\) for
\(\mathcal J_{n,\delta}^G\).
The two terms in \eqref{eq:main_surrogate_objective} have distinct statistical meanings. The
first is the square root empirical reconstruction error, and therefore favours the ordinary
PCA residual subspace. The second penalises residual subspaces that are aligned with
low-cost transport directions. In this sense, the surrogate does not merely inflate the PCA
criterion; it introduces a directional regularisation determined by the transport geometry.
\begin{remark}
    If \(G=I_p\), then
\(\mathsf S_G(Q)=Q\) and \(\rho_G(Q)=1\) for every \(Q\in\mathcal G_s\) so the surrogate minimisers coincide with ordinary PCA, consistent with
Theorem~\ref{thm:main_euclidean_exact}.
\end{remark}
It is important to distinguish the exact criterion \(\Phi_{n,\delta}^G\) from the surrogate
criterion \(\mathcal J_{n,\delta}^G\). The former is the genuine Wasserstein worst-case
reconstruction risk, whereas the latter is the tractable spectral criterion used for
computation and local asymptotic analysis. We do not claim that the two criteria have
identical finite-sample minimisers.
\begin{remark}
\label{rem:exact_surrogate_gap}
Let \(    Q_{\mathrm{sur}}
    \in
    \argmin_{Q\in\mathcal G_s}\mathcal J_{n,\delta}^G(Q),
    \quad
    Q_{\mathrm{ex}}
    \in
    \argmin_{Q\in\mathcal G_s}\Phi_{n,\delta}^G(Q).\)
Since \(\Phi_{n,\delta}^G(Q)\le\{\mathcal J_{n,\delta}^G(Q)\}^2\) for every
\(Q\in\mathcal G_s\), optimality gives
\[
\begin{aligned}
    0
    &\le
    \Phi_{n,\delta}^G(Q_{\mathrm{sur}})
    -
    \Phi_{n,\delta}^G(Q_{\mathrm{ex}})
\le
    \{\mathcal J_{n,\delta}^G(Q_{\mathrm{ex}})\}^2
    -
    \Phi_{n,\delta}^G(Q_{\mathrm{ex}})
\le
    \sup_{Q\in\mathcal G_s}
    \left[
        \{\mathcal J_{n,\delta}^G(Q)\}^2
        -
        \Phi_{n,\delta}^G(Q)
    \right].
\end{aligned}
\]
Thus the exact-DRO regret incurred by using the surrogate optimiser is controlled by the
relaxation gap between the spectral upper bound and the exact robust risk. A crude uniform
bound follows from
\(\Phi_{n,\delta}^G(Q)\ge\operatorname{tr}(\widehat\Sigma_nQ)\),
\(\operatorname{tr}(\widehat\Sigma_nQ)\le\operatorname{tr}(\widehat\Sigma_n)\), and
\(\rho_G(Q)\le\lambda_{\min}(G)^{-1}\):
\[
    \sup_{Q\in\mathcal G_s}
    \left[
        \{\mathcal J_{n,\delta}^G(Q)\}^2
        -
        \Phi_{n,\delta}^G(Q)
    \right]
    \le
    2\delta
    \sqrt{
        \frac{\operatorname{tr}(\widehat\Sigma_n)}
        {\lambda_{\min}(G)}
    }
    +
    \frac{\delta^2}{\lambda_{\min}(G)} .
\]
Moreover, if the exact objective satisfies the local quadratic growth condition
\[
    \Phi_{n,\delta}^G(Q)-\Phi_{n,\delta}^G(Q_{\mathrm{ex}})
    \ge
    \frac{\kappa}{2}\|Q-Q_{\mathrm{ex}}\|_F^2
\]
on a neighbourhood containing \(Q_{\mathrm{sur}}\), then
\[
    \|Q_{\mathrm{sur}}-Q_{\mathrm{ex}}\|_F
    \le
    \left\{
        \frac{2}{\kappa}
        \sup_{Q\in\mathcal G_s}
        \left[
            \{\mathcal J_{n,\delta}^G(Q)\}^2
            -
            \Phi_{n,\delta}^G(Q)
        \right]
    \right\}^{1/2}.
\]
\end{remark}

The remaining problem is to choose the transport geometry and radius in a principled way.
The next section introduces a data adaptive metric \(\widehat G_n\) and calibrates the radius
through robust Wasserstein profile inference.
\section{Data-driven calibration}
\label{sec:radius_calibration}

The previous section shows that anisotropic transport geometry is essential for obtaining a
non-trivial Wasserstein robust PCA criterion. We now turn to the second ingredient: the size
of the Wasserstein neighbourhood. The radius \(\delta\) should be large enough to account for
sampling uncertainty, but not so large that the robust criterion becomes dominated by
adversarial perturbations. Following the robust Wasserstein profile inference framework of
\citet{blanchet2019rwpi}, we calibrate the radius by the smallest
Wasserstein perturbation of the empirical law under which the relevant estimating
equation is satisfied. In our setting, the estimating equation is the first-order
condition characterising the population PCA subspace. The construction is stated for a general
data adaptive transport matrix \(\widehat G_n\).

Let \(U_\star\in\mathbb V_{p,r}\) span the leading population eigenspace, let
\(U_{\star,\perp}\in\mathbb R^{p\times s}\) be an orthonormal complement of \(U_\star\),
and write \(    \Pi_\star:=U_\star U_\star^\top,
    Q_\star:=I_p-\Pi_\star
    =
    U_{\star,\perp}U_{\star,\perp}^\top,
    \Gamma:=[U_\star,U_{\star,\perp}] .\)
The population PCA first-order condition is the vanishing of the covariance cross-block
between the leading subspace and its orthogonal complement:
$U_{\star,\perp}^\top \Sigma U_\star=0$ or, equivalently, $Q_\star\Sigma U_\star=0$.
We refer to this moment condition as the PCA estimating equation. To write it in
estimating-equation form, define
\begin{align}
    h(x;[U,U_\perp])
    :=
    \operatorname{vec}(U_\perp^\top xx^\top U),
    \label{def: h function}
\end{align}
where \(\operatorname{vec}\) denotes column-wise vectorisation. Since
\(\Sigma=\mathbb E(XX^\top)\), the population PCA subspace satisfies \(    \mathbb E\{h(X;[U_\star,U_{\star,\perp}])\}
    =
    \operatorname{vec}(U_{\star,\perp}^\top\Sigma U_\star)
    =
    0 .\)
Thus \(h\) is the estimating function whose mean encodes first-order stationarity of the
PCA subspace. Although \(h\) is written using a particular basis and orthogonal
complement, the condition
\(\mathbb E\{h(X;[U,U_\perp])\}=0\) depends only on the subspace projector, not on the
chosen orthonormal coordinates.

The adaptive-metric Robust Wasserstein profile inference(RWPI) statistic is defined by
\begin{equation}
\label{eq:main_rwp_stat_randomG}
    R_n^{(\widehat G)}([U_\star,U_{\star,\perp}])
    :=
    \inf_{\mathbb P}
    \left\{
    W_{2,\widehat G_n}^2(\mathbb P,\widehat{\mathbb P}_n):
    \mathbb E_{\mathbb P}
    \{h(Z;[U_\star,U_{\star,\perp}])\}=0
    \right\},
\end{equation}
where $W_{2,G}^2(\mu,\nu) $ is defined in \eqref{eq:Wasserstein_distance_eq}.
Thus \(R_n^{(\widehat G)}\) is the smallest squared adaptive Wasserstein perturbation of the
empirical law that makes the population PCA estimating equation hold at the true subspace.
Define the estimating equation plausible set
\begin{equation}
\label{eq:main_estimating_plausible_projectors}
    \widetilde{\mathcal P}_n^{(\widehat G)}(\delta)
    :=
    \bigcup_{\{\mathbb P:W_{2,\widehat G_n}(\mathbb P,\widehat{\mathbb P}_n)\le\delta\}}
    \left\{
    \Pi\in\mathcal G_r:
    \begin{array}{l}
    \Pi=UU^\top\text{ for some }U\in\mathbb V_{p,r},\\
    \mathbb E_{\mathbb P}\{h(Z;[U,U_\perp])\}=0
    \text{ for some }U_\perp
    \end{array}
    \right\},
\end{equation}
and the corresponding global PCA plausibility set
\begin{equation}
\label{eq:main_plausible_projectors}
    \mathcal P_n^{(\widehat G)}(\delta)
    :=
    \bigcup_{\{\mathbb P:W_{2,\widehat G_n}(\mathbb P,\widehat{\mathbb P}_n)\le\delta\}}
    \argmin_{\Pi\in\mathcal G_r}
    \mathbb E_{\mathbb P}\{\|(I_p-\Pi)Z\|_2^2\}.
\end{equation}

\begin{remark}
\label{prop:main_rwp_plausible_equiv}
The role of the RWPI statistic \eqref{eq:main_rwp_stat_randomG} is to calibrate a Wasserstein radius for which the
population PCA estimating equation is statistically plausible. For a fixed basis
\([U,U_\perp]\), \(R_n^{(\widehat G)}([U,U_\perp])\) is the squared minimum
\(W_{2,\widehat G_n}\)-distance from the empirical law to a distribution under which \(\mathbb E_{\mathbb P}\{h(Z;[U,U_\perp])\}=0 .\)
Thus \(R_n^{(\widehat G)}([U_\star,U_{\star,\perp}])\)
is the profile statistic attached to the population PCA subspace.

The plausible sets
\(\mathcal P_n^{(\widehat G)}(\delta)\) and
\(\widetilde{\mathcal P}_n^{(\widehat G)}(\delta)\) are random subsets of the projector
space. With the convention that
\(\{\Pi_\star\in\mathcal P_n^{(\widehat G)}(\delta)\}:=
    \{\omega:\Pi_\star\in\mathcal P_n^{(\widehat G)}(\delta;\omega)\}\) denotes the event that the random
set constructed from the sample contains the fixed population projector \(\Pi_\star\), we
have, for every \(\delta\ge0\),
\begin{equation}
\label{eq:Plausibility containment}
        \{\Pi_\star\in\mathcal P_n^{(\widehat G)}(\delta)\}
    \subseteq
    \{\Pi_\star\in\widetilde{\mathcal P}_n^{(\widehat G)}(\delta)\}
    =
    \left\{
    \sqrt{R_n^{(\widehat G)}([U_\star,U_{\star,\perp}])}\le\delta
    \right\}.
\end{equation}
The profile statistic should not be evaluated at the ordinary sample PCA subspace. Indeed,
if \(\widehat U_{\mathrm{PCA}}\) is the empirical PCA basis and
\(\widehat U_{\mathrm{PCA},\perp}\) is the corresponding empirical orthogonal complement,
then \(    R_n^{(\widehat G)}
    ([\widehat U_{\mathrm{PCA}},\widehat U_{\mathrm{PCA},\perp}])
    =
    0 .\)
This degeneracy is analogous to evaluating a profile or estimating-equation statistic at
the empirical estimator itself. The sample PCA estimator may instead be used as a plug-in
device for estimating the nuisance matrices in the limiting distribution of
\(R_n^{(\widehat G)}([U_\star,U_{\star,\perp}])\).
\end{remark}
The choice of radius follows directly from this profile interpretation. By \eqref{eq:Plausibility containment}, 
\(R_n^{(\widehat G)}([U_\star,U_{\star,\perp}])\) is equivalently the squared minimum
Wasserstein perturbation required for the first-order plausibility set to contain the
population PCA projector, the
event \(\sqrt{R_n^{(\widehat G)}([U_\star,U_{\star,\perp}])}\le \delta\)
means that the population PCA estimating equation can be made to hold by perturbing the
empirical law by at most radius \(\delta\). The asymptotic theory below shows that
\(nR_n^{(\widehat G)}([U_\star,U_{\star,\perp}])\) has a non-degenerate limiting
distribution. Hence, if \(q_{1-\alpha}\) denotes the \((1-\alpha)\)-quantile of this
limiting law, the choice \(\delta_{n,\alpha}=\sqrt{q_{1-\alpha}/n}\)
satisfies
\[
        \Pr\left\{
    \Pi_\star\in
    \widetilde{\mathcal P}_n^{(\widehat G)}(\delta_{n,\alpha})
    \right\}
    =
    \Pr\left\{
    \sqrt{R_n^{(\widehat G)}([U_\star,U_{\star,\perp}])}
    \le
    \delta_{n,\alpha}
    \right\}
    \to 1-\alpha .
\]
In this sense, the RWPI radius is chosen so that the first-order plausibility set contains
the population PCA projector with asymptotic probability \(1-\alpha\).

We now derive the limiting law used to choose \(\delta\). For the true subspace, write
\begin{align}
        h(x):=h(x;[U_\star,U_{\star,\perp}]),
    \qquad
    J(x):=D_xh(\cdot;[U_\star,U_{\star,\perp}])\big|_x, \nonumber\\
        \Sigma_h:=\operatorname{Var}\{h(X)\},
    \qquad
    A_G:=\mathbb E\{J(X)G^{-1}J(X)^\top\}, \label{def: J(x) A_G}
\end{align}
where \(D_x\) denotes the differential operator with respect to the data argument.
\begin{assumption}
\label{ass:profile_calibration}
The following hold.
\begin{enumerate}[label=(\roman*), ref=\theassumption(\roman*)]
\item \label{ass:main_moment_cov}
\label{ass:main_rwp_gap}
\label{ass:main_rwp_moment}
\(X_1,\ldots,X_n\) are i.i.d., centred, and satisfy
\(\mathbb E\|X_1\|_2^4<\infty\). Let $r$, \(1 \leq r \leq p\), be such that the covariance matrix \(\Sigma=\operatorname{Cov}(X_1)\) has the spectral decomposition
\[
    \Sigma
    =
    \Gamma
    \begin{pmatrix}
    \Lambda_1 & 0\\
    0 & \Lambda_2
    \end{pmatrix}
    \Gamma^\top,
    \qquad
    \Lambda_1=\operatorname{diag}(\theta_1,\ldots,\theta_r),
    \quad
    \Lambda_2=\operatorname{diag}(\theta_{r+1},\ldots,\theta_p),
\]
with \(\theta_1\ge\cdots\ge\theta_r>\theta_{r+1}\ge\cdots\ge\theta_p .\)
\item \label{ass:main_random_G}
\label{ass:main_adaptive_G_clt}
The random transport matrices \(\widehat G_n\) are symmetric positive definite and,
for some deterministic \(G\succ0\), \(\|\widehat G_n-G\|_{\operatorname{op}}\xrightarrow{p}0.\)
Moreover, for constants \(0<c<C<\infty\),
\[
    \Pr\{c\le\lambda_{\min}(\widehat G_n)
    \le\lambda_{\max}(\widehat G_n)\le C\}\to1 .
\]

\item \label{ass:main_rwp_nondeg}
The matrix \(A_G\) defined in (\ref{def: J(x) A_G}) is positive definite, where $G$ here refers to the limiting deterministic matrix of \(\widehat G_n\).
\end{enumerate}
\end{assumption}

The second part of Assumption~\ref{ass:main_random_G} is deliberately general. It permits
any data-adaptive transport geometry satisfying operator-norm consistency and spectral
stability. The particular construction of \(\widehat G_n\) used in our numerical experiments
is therefore an implementation choice, not a restriction of the theory.

Define \(    B_G
    :=
    \Sigma_h^{1/2}A_G^{-1}\Sigma_h^{1/2},\)
and let \(\lambda_1(B_G),\ldots,\lambda_{rs}(B_G)\)
be the eigenvalues of \(B_G\). If
\(\chi_{1,1}^2,\ldots,\chi_{1,rs}^2\) are independent \(\chi_1^2\) random variables, then \(\sum_{j=1}^{rs}\lambda_j(B_G)\chi_{1,j}^2\) has the same distribution as \(Z_h^\top A_G^{-1}Z_h\) for
\(Z_h\sim N(0,\Sigma_h)\).
\begin{theorem}
\label{thm:main_rwp_limit}
\label{thm:main_rwp_limit_randomG}
Under Assumption~\ref{ass:profile_calibration},
\[
    nR_n^{(\widehat G)}([U_\star,U_{\star,\perp}])
    \Rightarrow
    Z_h^\top A_G^{-1}Z_h\overset{d}{=}\sum_{j=1}^{rs}\lambda_j(B_G)\chi_{1,j}^2,
    \qquad
    Z_h\sim N(0,\Sigma_h).
\]
\end{theorem}
Let \(q_{1-\alpha}^{(G)}\) denote the \((1-\alpha)\)-quantile of \( Z_h^\top A_G^{-1}Z_h\). The corresponding
data-adaptive radius choice is
\begin{equation}
\label{eq:main_delta_choice}
    \delta_{n,\alpha}^{(G)}
    :=
    \sqrt{\frac{q_{1-\alpha}^{(G)}}{n}} ,
\end{equation}
 then \(    \Pr\left\{
    \sqrt{R_n^{(\widehat G)}([U_\star,U_{\star,\perp}])}
    \le
    \delta_{n,\alpha}^{(G)}
    \right\}
    =
    \Pr\left\{
    \Pi_\star\in
    \widetilde{\mathcal P}_n^{(\widehat G)}(\delta_{n,\alpha}^{(G)})
    \right\}
    \to
    1-\alpha .\)
Thus \(\delta_{n,\alpha}^{(G)}\) is the oracle radius that gives
asymptotic \((1-\alpha)\)-level first-order plausibility for the population PCA projector. Theorem~\ref{thm:main_rwp_limit} shows that the adaptive metric affects the first-order
calibration only through its deterministic limit \(G\), through the eigenvalues of
\(B_G=\Sigma_h^{1/2}A_G^{-1}\Sigma_h^{1/2}\). In particular, the fixed-metric version is
recovered by taking \(\widehat G_n\equiv G\).

The oracle quantile \(q_{1-\alpha}^{(G)}\) is not directly available because
\(\Sigma_h\) and \(A_G\) are unknown. We estimate it by a plug-in weighted chi-square law.
Let \(\widehat U_n\in\mathbb V_{p,r}\) be any preliminary estimator satisfying
\begin{align}
        \widehat\Pi_n^{(0)}:=\widehat U_n\widehat U_n^\top,
    \qquad
    \|\widehat\Pi_n^{(0)}-\Pi_\star\|_F\xrightarrow{p}0, \label{def:PCA_estimator_condition}
\end{align}
and let \(\widehat U_{n,\perp}\) be an orthonormal complement. Define
\begin{align}
        \widehat h_i:=h(X_i;[\widehat U_n,\widehat U_{n,\perp}]),
    \qquad
    \bar h_n:=\frac1n\sum_{i=1}^n\widehat h_i, 
    \qquad
        \widehat\Sigma_{h,n}
    :=
    \frac1n\sum_{i=1}^n
    (\widehat h_i-\bar h_n)(\widehat h_i-\bar h_n)^\top , \label{def:hat h Sigma_h}
\end{align}
and
\begin{align}
        \widehat J_i
    :=
    D_xh(\cdot;[\widehat U_n,\widehat U_{n,\perp}])\big|_{x=X_i},
    \qquad
    \widehat A_n^{(\widehat G)}
    :=
    \frac1n\sum_{i=1}^n
    \widehat J_i\widehat G_n^{-1}\widehat J_i^\top . \label{def: hat J_i A_n}
\end{align}
Let \(    \widehat B_n^{(\widehat G)}
    :=
    \widehat\Sigma_{h,n}^{1/2}
    \{\widehat A_n^{(\widehat G)}\}^{-1}
    \widehat\Sigma_{h,n}^{1/2},\)
and denote its eigenvalues by
\(\widehat\lambda_{1,n},\ldots,\widehat\lambda_{rs,n}\). Conditioning on $\widehat \lambda_{1,n}, \ldots, \widehat\lambda_{rs,n}$, we define
\(\widehat q_{1-\alpha}\) as the conditional \((1-\alpha)\) quantile of
\begin{equation}
\label{eq:main_plugin_weighted_chisq}
    \sum_{j=1}^{rs}\widehat\lambda_{j,n}\chi_{1,j}^2,
\end{equation}
where \(\chi_{1,1}^2,\ldots,\chi_{1,rs}^2\) are independent \(\chi_1^2\) variables.
\begin{remark}
\label{rem:plugin_basis_invariance}
The plug in distribution in \eqref{eq:main_plugin_weighted_chisq} is independent of the
particular orthonormal bases used for the estimated principal and residual subspaces.
\end{remark}
\begin{theorem}
\label{thm:main_plugin_quantile}
Under Assumption~\ref{ass:profile_calibration} and the preceding preliminary-estimator
condition \eqref{def:PCA_estimator_condition},
\(
    \widehat q_{1-\alpha}\xrightarrow{p}q_{1-\alpha}^{(G)},
\)
where $\widehat q_{1-\alpha}$ is defined in (\ref{eq:main_plugin_weighted_chisq}) and \(q_{1-\alpha}^{(G)}\) denote the \((1-\alpha)\)-quantile of \( Z_h^\top A_G^{-1}Z_h\). Consequently,
\begin{equation}
\label{eq:main_plugin_delta}
    \widehat\delta_{n,\alpha}
    :=
    \sqrt{\frac{\widehat q_{1-\alpha}}{n}}
\end{equation}
satisfies
\[
    \Pr\left\{
    \sqrt{R_n^{(\widehat G)}([U_\star,U_{\star,\perp}])}
    \le
    \widehat\delta_{n,\alpha}
    \right\}
    \to
    1-\alpha .
\]
\end{theorem}

\section{Asymptotic theory}
\label{subsec:main_adaptive_clt}

The preceding sections introduced two ingredients of the proposed method: an adaptive
transport geometry \(\widehat G_n\) and an RWPI-calibrated radius
\(\widehat\delta_{n,\alpha}\). We now study the statistical effect of combining these two
ingredients in the surrogate criterion. The purpose of this section is not merely to
establish consistency. Rather, we identify the local first-order change induced by the
Wasserstein penalty and show how it modifies the usual PCA fluctuation on the Grassmannian.

Recall that, for \(Q\in\mathcal G_s\), $    \rho_{\widehat G_n}(Q)
    :=
    \lambda_{\max}(\widehat G_n^{-1/2}Q\widehat G_n^{-1/2}).$
The adaptive surrogate criterion is
\begin{equation}
\label{eq:main_surrogate_objective_adaptive}
    \mathcal J_n^{\mathrm{ad}}(Q)
    :=
    \mathcal J_{n,\widehat\delta_{n,\alpha}}^{\widehat G_n}(Q)
    =
    \sqrt{\operatorname{tr}(\widehat\Sigma_nQ)}
    +
    \widehat\delta_{n,\alpha}\sqrt{\rho_{\widehat G_n}(Q)},
    \qquad Q\in\mathcal G_s .
\end{equation}
We define the adaptive surrogate estimator by 
\begin{equation}
\label{eq:surrogate_optimiser}
        \widehat Q_n^{\mathrm{ad}}
    \in
    \argmin_{Q\in\mathcal G_s}\mathcal J_n^{\mathrm{ad}}(Q).
\end{equation}

For comparison, the corresponding exact adaptive DRO objective is
\begin{equation}
\label{eq:main_exact_objective}
    \Phi_n^{\mathrm{ad}}(Q)
    :=
    \Phi_{n,\widehat\delta_{n,\alpha}}^{\widehat G_n}(Q)
    =
    \sup_{\mathbb P:
    W_{2,\widehat G_n}(\mathbb P,\widehat{\mathbb P}_n)
    \le
    \widehat\delta_{n,\alpha}}
    \mathbb E_{\mathbb P}(Z^\top QZ),
    \qquad Q\in\mathcal G_s .
\end{equation}
The exact objective is the original DRO risk, whereas
\(\mathcal J_n^{\mathrm{ad}}\) is the tractable spectral surrogate in \eqref{eq:main_surrogate_objective}. The local limit theory below is developed for the
surrogate estimator. We also show that the exact adaptive DRO minimiser is consistent for
the same population residual projector, so that the surrogate and exact criterion agree at
the first population level even though their finite-sample minimisers need not coincide.

We work locally around the population residual projector \(Q_\star\). For
\(K\in\mathbb R^{r\times s}\), define
\[
    W(K)
    :=
    \begin{pmatrix}
    K\\ I_s
    \end{pmatrix}
    (I_s+K^\top K)^{-1/2},
    \qquad
    Q(K)
    :=
    \Gamma W(K)W(K)^\top\Gamma^\top ,
\]
where $\Gamma$ is the population orthogonal eigenbasis adapted to the PCA decomposition in Assumption~\ref{ass:main_moment_cov}.
Then \(Q(0)=Q_\star\), and every residual projector sufficiently close to \(Q_\star\) has a
unique representation of this form. Moreover,
\[
DQ(0)[\Xi]
:=
    \left.\frac{d}{dt}Q(t\Xi)\right|_{t=0}
    =
\Gamma
\begin{pmatrix}
0 & \Xi\\
\Xi^\top & 0
\end{pmatrix}
\Gamma^\top,
\qquad
\Xi\in\mathbb R^{r\times s}.
\]
Thus \(K\) parametrises local rotations between the leading and residual population
eigenspaces. This is the natural coordinate system for deriving a local PCA limit, because
the identifiable object is the projector rather than a particular eigenvector basis.

We impose the following local regularity conditions.

\begin{assumption}
\label{ass:main_local_regular}
Suppose that
\begin{enumerate}[label=(\roman*), ref=\theassumption(\roman*)]
\item \label{ass:main_a_star}
the scalar residual variance
$a_\star:=\operatorname{tr}(\Lambda_2)$
is positive, where $\Lambda_2$ is defined in Assumption~\ref{ass:main_moment_cov};

\item \label{ass:main_simple}
the largest eigenvalue
\(
    \ell_\star
    :=
    \rho_G(Q_\star)
    =
    \lambda_{\max}(G^{-1/2}Q_\star G^{-1/2})
\)
is simple, where $G$ is the limiting matrix in Assumptions~\ref{ass:main_random_G}.
\end{enumerate}
\end{assumption}

Assumption~\ref{ass:main_a_star} excludes the degenerate case in which the square-root
reconstruction term has zero population value at \(Q_\star\). Assumption~\ref{ass:main_simple}
ensures differentiability of the spectral exposure \(Q\mapsto \rho_G(Q)^{1/2}\) at
\(Q_\star\). Let \(v_\star\) be a unit eigenvector associated with
\(\ell_\star\), and write
\[
    \Gamma^\top G^{-1/2}v_\star
    =
    \begin{pmatrix}
    \alpha_\star\\
    \beta_\star
    \end{pmatrix},
    \qquad
    \alpha_\star\in\mathbb R^r,\quad
    \beta_\star\in\mathbb R^s .
\]
The derivative of the spectral penalty at \(Q_\star\) is represented in local coordinates by
\begin{align}
\label{def:g_G}
        g_G
    :=
    \ell_\star^{-1/2}\alpha_\star\beta_\star^\top
    \in\mathbb R^{r\times s}.
\end{align}
This quantity is invariant to the sign of \(v_\star\). Let \(g_{\widehat G_n}\) denote the
analogous object obtained by replacing \(G\) with \(\widehat G_n\). Under
Assumptions~\ref{ass:main_random_G} and \ref{ass:main_simple},
\begin{align}
\label{def: g_hatG}
    g_{\widehat G_n}\xrightarrow{p}g_G .
\end{align}
By Theorem~\ref{thm:main_plugin_quantile},
\begin{align}
    \sqrt n\,\widehat\delta_{n,\alpha}
    \xrightarrow{p}
    \tau_\alpha,
    \qquad
    \tau_\alpha:=\sqrt{q_{1-\alpha}^{(G)}}. \label{def: tau alpha}
\end{align}
Write the sample covariance in the population eigenbasis as
\[
    \Gamma^\top\widehat\Sigma_n\Gamma
    =
    \begin{pmatrix}
    \widehat\Sigma_{11,n} & \widehat\Sigma_{12,n}\\
    \widehat\Sigma_{21,n} & \widehat\Sigma_{22,n}
    \end{pmatrix}.
\]
Define the empirical local score and the population local Hessian operator by
\begin{align}
\label{def:Z_n H_star}
        Z_n
    :=
    a_\star^{-1/2}\sqrt n\,\widehat\Sigma_{12,n},
    \qquad
H_\star\Xi
:=
a_\star^{-1/2}(\Lambda_1\Xi-\Xi\Lambda_2),
\quad
\Xi\in\mathbb R^{r\times s}.
\end{align}
The eigengap in Assumption~\ref{ass:main_moment_cov} implies that \(H_\star\) is positive
definite, where positivity is understood with
respect to the Frobenius inner product in the operator sense. In this notation, \(Z_n\) is the usual off-diagonal sample covariance fluctuation
that drives local PCA asymptotics, while \(H_\star\) is the local curvature of the population
reconstruction criterion.

The following theorem is the main expansion behind the asymptotic theory.
\begin{theorem}
\label{thm:main_local_quadratic}
For \(\Xi\in\mathbb R^{r\times s}\), define
\[
    \widetilde{\mathcal J}_n^{\mathrm{ad}}(\Xi)
    :=
    n\left[
    \mathcal J_n^{\mathrm{ad}}\{Q(n^{-1/2}\Xi)\}
    -
    \mathcal J_n^{\mathrm{ad}}(Q_\star)
    \right],
\]
where $\mathcal J_n^{\mathrm{ad}}$ is defined in \eqref{eq:main_surrogate_objective_adaptive}. Under Assumptions~\ref{ass:main_moment_cov}, \ref{ass:main_random_G},
\ref{ass:main_rwp_nondeg}, \ref{ass:main_a_star} and \ref{ass:main_simple}, and under
the conditions of Theorem~\ref{thm:main_plugin_quantile} so that \(    \sqrt n\,\widehat\delta_{n,\alpha}\xrightarrow{p}\tau_\alpha,\)
for every fixed \(R<\infty\),
\[
    \widetilde{\mathcal J}_n^{\mathrm{ad}}(\Xi)
    =
    \left\langle
    Z_n+\sqrt n\,\widehat\delta_{n,\alpha}g_{\widehat G_n},
    \Xi
    \right\rangle_F
    +
    \frac12\langle\Xi,H_\star\Xi\rangle_F
    +
    o_p(1),
\]
uniformly over \(\{\Xi:\|\Xi\|_F\le R\}\), where \(\langle A,B\rangle_F=\operatorname{tr}(A^\top B)\) and $g_{\widehat G_n}$, $Z_n $ and $ H_\star$  are defined in  \eqref{def:g_G} with $G=\widehat{G}_n$ and \eqref{def:Z_n H_star} respectively.
\end{theorem}

The interpretation of Theorem~\ref{thm:main_local_quadratic} is central. Without the
Wasserstein penalty, the local objective consists of the usual PCA score \(Z_n\) plus the
quadratic curvature \(H_\star\). The adaptive robustification adds the extra linear term
\(\sqrt n\,\widehat\delta_{n,\alpha}g_{\widehat G_n}\). Since the RWPI radius is of order
\(n^{-1/2}\), this penalty is neither asymptotically negligible nor dominant: it enters at
exactly the same local scale as the sampling fluctuation. Thus the calibrated Wasserstein
penalty acts as a first-order local regularisation drift, whose direction is determined by
the limiting transport geometry \(G\) and whose magnitude is determined by the RWPI quantile
\(\tau_\alpha\).

The local expansion immediately gives consistency and root-\(n\) localisation.

\begin{theorem}
\label{thm:main_consistency_rootn}
Under the assumptions of Theorem~\ref{thm:main_local_quadratic},
$\widehat Q_n^{\mathrm{ad}}\xrightarrow{p}Q_\star$,
where $Q_\star$ is the population residual projector and $\widehat Q_n^{\mathrm{ad}}$ is the optimiser defined in (\ref{eq:surrogate_optimiser}). Moreover, with probability tending to one, $\widehat Q_n^{\mathrm{ad}}$ lies in the
coordinate neighbourhood of $Q_\star$. On this event there exists a unique
$\widehat K_n^{\mathrm{ad}}\in\mathbb R^{r\times s}$ such that
$\widehat Q_n^{\mathrm{ad}}=Q(\widehat K_n^{\mathrm{ad}})$, where
$\widehat K_n^{\mathrm{ad}}=O_p(n^{-1/2})$.
\end{theorem}

Theorem~\ref{thm:main_consistency_rootn} shows that the adaptive robustification does not
change the population target at the consistency level. This should not be interpreted as
saying that the DRO correction is asymptotically irrelevant. Rather, the RWPI calibration
places the Wasserstein radius on the local scale \(\widehat\delta_{n,\alpha}=O_p(n^{-1/2})\).
At this scale, the robust penalty is too small to move the estimator to a different
population limit, but large enough to perturb the local optimisation problem. Consequently,
the estimator remains consistent for the nominal PCA subspace while retaining a
first order regularisation drift that can improve finite sample out of sample performance
in certain scenarios.
\begin{remark}
\label{rem:main_exact_consistency}
Although the local distribution theory is developed for the surrogate estimator, the exact
adaptive DRO minimiser is consistent for the same population residual projector. Indeed, if
\[
    \widehat Q_n^{\mathrm{ex,ad}}
    \in
    \argmin_{Q\in\mathcal G_s}\Phi_n^{\mathrm{ad}}(Q),
\]
then, under the assumptions of Theorem~\ref{thm:main_local_quadratic},
$\widehat Q_n^{\mathrm{ex,ad}}\xrightarrow{p}Q_\star$.
Hence,
$\|\widehat Q_n^{\mathrm{ex,ad}}-\widehat Q_n^{\mathrm{ad}}\|_F\xrightarrow{p}0$.
The proof is given in Supplementary Section~\ref{supp:sec_exact_consistency}.
\end{remark}
We now state the local distributional consequences. Let
\[
    \Gamma^\top X_i
    =
    \begin{pmatrix}
    \xi_i\\ \eta_i
    \end{pmatrix},
    \qquad
    \xi_i\in\mathbb R^r,\quad
    \eta_i\in\mathbb R^s,
\]
and define
$\Omega_{12}
    :=
    \operatorname{Var}\{\operatorname{vec}(\xi_1\eta_1^\top)\}$ and
    $\Omega_Z:=a_\star^{-1}\Omega_{12}$.

\begin{theorem}
\label{thm:main_coordinate_clt}
\label{thm:main_coordinate_clt_adaptive}
Let \(Z\in\mathbb R^{r\times s}\) satisfy $\operatorname{vec}(Z)\sim N(0,\Omega_Z)$ defined above, with $g_G, \tau_\alpha$ and $H_\star$ defined in (\ref{def:g_G}), (\ref{def: tau alpha}) and (\ref{def:Z_n H_star}) respectively, then under the assumptions of Theorem~\ref{thm:main_local_quadratic},
\[
    \sqrt n\,\widehat K_n^{\mathrm{ad}}
    \Rightarrow
    -H_\star^{-1}(Z+\tau_\alpha g_G)
    \qquad
    \text{in }\mathbb R^{r\times s},
\]
where \(\widehat K_n^{\mathrm{ad}}\) is defined in
Theorem~\ref{thm:main_consistency_rootn}. Furthermore,
\[
    \operatorname{vec}(\sqrt n\,\widehat K_n^{\mathrm{ad}})
    \Rightarrow
    N(\mu_{\star,\alpha},V_\star),
\]
where
$\mu_{\star,\alpha}
    =
    -\tau_\alpha\mathcal H_\star^{-1}\operatorname{vec}(g_G)$,
    $V_\star
    =
    \mathcal H_\star^{-1}\Omega_Z(\mathcal H_\star^{-1})^\top$
and
\[
    \mathcal H_\star
    :=
    a_\star^{-1/2}(I_s\otimes\Lambda_1-\Lambda_2\otimes I_r),
\]
where $\otimes$ denotes the Kronecker product.

\end{theorem}

Theorem~\ref{thm:main_coordinate_clt} separates two effects. The covariance matrix \(V_\star\) is the ordinary local PCA covariance after adjustment by
the population Hessian \(H_\star\).
The mean shift $-\tau_\alpha H_\star^{-1}g_G$
is new and is entirely due to the adaptive Wasserstein penalty. Hence the proposed estimator
has the same root-\(n\) scale as PCA, but it is locally centred at a robustified perturbation
of the ordinary PCA target. The direction of this perturbation is determined by the geometry
\(G\), through the derivative \(g_G\), and its magnitude is determined by the RWPI calibration
level through \(\tau_\alpha\). In particular, the CLT does not indicate that robustification
uniformly improves PCA under every possible test distribution. Rather, it clarifies precisely how the
method shifts the local subspace estimate, which is the mechanism behind improved out-of-sample performance when the target shift is aligned with the chosen
transport geometry.

Finally, because the projector is the identifiable object, we translate the coordinate CLT
back to the Grassmannian.

\begin{theorem}
\label{thm:main_projector_clt}
\label{thm:main_projector_clt_adaptive}
Under the assumptions and conditions of Theorem~\ref{thm:main_coordinate_clt},
\[
    \sqrt n(\widehat Q_n^{\mathrm{ad}}-Q_\star)
    \Rightarrow
    \Gamma
    \begin{pmatrix}
    0 & \Xi_{\infty,\alpha}\\
    \Xi_{\infty,\alpha}^\top & 0
    \end{pmatrix}
    \Gamma^\top,
\]
where $\Xi_{\infty,\alpha}
    :=
    -H_\star^{-1}(Z+\tau_\alpha g_G)$.
Equivalently, for the adaptive principal projector
$\widehat\Pi_n^{\mathrm{ad}}
    :=
    I_p-\widehat Q_n^{\mathrm{ad}}$, and writing 
    $\Pi_\star:=I_p-Q_\star$,
we have
\[
    \sqrt n(\widehat\Pi_n^{\mathrm{ad}}-\Pi_\star)
    \Rightarrow
    -
    \Gamma
    \begin{pmatrix}
    0 & \Xi_{\infty,\alpha}\\
    \Xi_{\infty,\alpha}^\top & 0
    \end{pmatrix}
    \Gamma^\top .
\]
\end{theorem}

Theorem \ref{thm:main_projector_clt_adaptive} is the coordinate-free version of the preceding result. It confirms that
the adaptive surrogate estimator is asymptotically normal on the tangent space of the
Grassmannian at \(Q_\star\), with the same linearisation map as ordinary PCA but with a
Wasserstein-induced mean shift. This makes the statistical role of the proposed procedure
explicit: the adaptive metric and RWPI radius do not change the first-order rate, but they
change the local centring of the estimated subspace in a direction dictated by the transport
geometry.

\section{Choice of adaptive transport geometry}
\label{sec:adaptive_transport_geometry}
 This section discusses two
data-adaptive geometries used in the numerical studies. They are both diagonal and are
therefore statistically stable in moderate dimensions, but they encode different
assumptions about where source information about future target variation is located.

Recall that the weighted surrogate criterion is
\[
    \mathcal J_{n,\delta}^{G}(Q)
    =
    \sqrt{\operatorname{tr}(\widehat\Sigma_n Q)}
    +
    \delta\sqrt{\rho_G(Q)},
    \qquad
    \rho_G(Q)
    =
    \lambda_{\max}\{G^{-1/2}QG^{-1/2}\}.
\]
The first term is the ordinary empirical reconstruction term. The second term measures the
exposure of the residual subspace \(Q\) to directions that are inexpensive under the
transport cost \(c_G(x,y)=(x-y)^\top G(x-y)\). Directions with small \(G\)-cost are
amplified by \(G^{-1/2}\), and residual subspaces containing such directions receive a
larger value of \(\rho_G(Q)\). Thus the role of \(G\) is to specify which directions of distributional perturbation are
relatively cheap or expensive. Candidate residual subspaces that contain low-cost
directions receive larger exposure penalties, so the surrogate encourages the fitted
principal subspace to capture directions that are inexpensive to perturb.

Let \(\widehat U_{\mathrm{PCA}}\in\mathbb V_{p,r}\) denote the ordinary rank-\(r\) PCA
loading matrix computed from the centred training sample, and write \(    \widehat P_{\mathrm{PCA}}
    :=
    \widehat U_{\mathrm{PCA}}\widehat U_{\mathrm{PCA}}^\top, \quad
    \widehat Q_{\mathrm{PCA}}
    :=
    I_p-\widehat P_{\mathrm{PCA}} .\)
The two geometries below are constructed from the decomposition of empirical variation into
the empirical PCA block and its residual complement.

\subsection{Residual variance geometry}
\label{subsec:residual_variance_geometry}

The residual variance geometry uses coordinatewise variation left unexplained by the
initial rank-\(r\) PCA fit. Define
\[
    \widehat v_\ell
    :=
    \frac1n\sum_{i=1}^n
    \left[
        (\widehat Q_{\mathrm{PCA}}X_i)_\ell
    \right]^2
    =
    \left[
        \widehat Q_{\mathrm{PCA}}
        \widehat\Sigma_n
        \widehat Q_{\mathrm{PCA}}
    \right]_{\ell\ell},
    \qquad
    \ell=1,\ldots,p .
\]
For a ridge parameter \(\tau_n>0\), define
\begin{equation}
\label{eq:supp_residual_G}
    \widehat G_n^{\mathrm{res}}
    =
    c_n^{\mathrm{res}}
    \operatorname{diag}
    \left\{
        (\widehat v_1+\tau_n)^{-1},
        \ldots,
        (\widehat v_p+\tau_n)^{-1}
    \right\},
\end{equation}
where \(c_n^{\mathrm{res}}>0\) is a scale normalising constant.

This geometry makes directions with large empirical residual variance inexpensive under
the transport cost. Since the surrogate penalises residual exposure, directions with large residual variance
under the initial PCA fit become costly to leave in the final residual subspace. The
residual geometry therefore encourages the fitted principal subspace to recover directions
that ordinary PCA may have initially missed. It is therefore
most natural when the ordinary source PCA fit may miss directions that are relevant under a
future target distribution or under a clean distribution after contamination. In such
settings, the target-relevant information is not primarily contained in the leading
empirical PCA block, but it remains visible in the empirical residual covariance.

\subsection{PCA-block variance geometry}
\label{subsec:pca_block_geometry}

A complementary construction uses the coordinatewise variation already represented inside
the empirical PCA block. Define
\[
    \widehat w_\ell
    :=
    \frac1n\sum_{i=1}^n
    \left[
        (\widehat P_{\mathrm{PCA}}X_i)_\ell
    \right]^2
    =
    \left[
        \widehat P_{\mathrm{PCA}}
        \widehat\Sigma_n
        \widehat P_{\mathrm{PCA}}
    \right]_{\ell\ell},
    \qquad
    \ell=1,\ldots,p .
\]
The corresponding PCA-block geometry is
\begin{equation}
\label{eq:pca_block_G}
    \widehat G_n^{\mathrm{pca}}
    =
    c_n^{\mathrm{pca}}
    \operatorname{diag}
    \left\{
        (\widehat w_1+\tau_n)^{-1},
        \ldots,
        (\widehat w_p+\tau_n)^{-1}
    \right\},
\end{equation}
where \(c_n^{\mathrm{pca}}>0\) is a scale normalising constant.

Coordinates with large \(\widehat w_\ell\) contribute strongly to the empirical PCA block.
Under \(\widehat G_n^{\mathrm{pca}}\), these coordinates have smaller transport cost, so a
candidate residual subspace that leaves them unexplained receives a larger exposure
penalty. The PCA-block geometry therefore reinforces directions that are already visible
in the empirical PCA block. This is appropriate when source PCA contains partial
information about the directions that will be important for the target distribution, or
when contamination does not fully displace the clean low-rank structure from the empirical
PCA block.

\subsection{Interpretation and use}

The residual and PCA-block geometries are not intended to estimate a single canonical
transport metric. They encode different structural hypotheses. The residual geometry
\(\widehat G_n^{\mathrm{res}}\) is designed for situations in which target-relevant
directions are missed by the leading empirical PCA fit but remain detectable in the
residual variation. The PCA-block geometry \(\widehat G_n^{\mathrm{pca}}\) is designed for
situations in which target-relevant directions are already partially represented in the
empirical PCA block and should be reinforced rather than sought in the residual
covariance.

This distinction is useful for interpreting the numerical studies. In a same-distribution
finite-sample setting, the relevant question is whether the population signal directions
are already visible but unstable in the empirical PCA block, in which case the PCA-block
geometry can reinforce them. Under covariance shift, the question is whether future
target-relevant directions are visible in the source PCA block or are primarily residual
directions relative to the source sample. Under training contamination, the analogous
question is whether the clean oracle subspace remains visible in the PCA block of the
contaminated empirical covariance, or whether contamination has pushed it into the
empirical residual block. More generally, the relative performance of the two geometries
depends on whether target-relevant variation is visible through the coordinatewise
summaries used to construct the transport cost: in the empirical PCA block, in the
empirical residual block, or not clearly in either diagonal summary. The numerical studies
in Section~\ref{sec:numerical} are designed to make this visibility mechanism explicit.

This interpretation also suggests a practical diagnostic principle. When no substantial
future shift is expected, or when prior knowledge suggests that future target-relevant
directions are already visible in the leading empirical PCA block, the PCA-block geometry
\(\widehat G_n^{\mathrm{pca}}\) is the more conservative choice: it reinforces directions
already represented in the initial PCA fit and can act as a finite-sample stabiliser. By
contrast, the residual geometry \(\widehat G_n^{\mathrm{res}}\) is more appropriate when
one expects important directions to be under-represented by ordinary PCA but still
detectable through the empirical residual variation. It should not be interpreted as a
device for discovering directions that are completely absent from the training
distribution; its usefulness depends on whether the relevant structure leaves a detectable
trace in the residual covariance. Thus the two geometries encode different assumptions
about where useful source information is located, rather than providing universally
dominant alternatives to PCA. In applications, these considerations can be used as qualitative guidance for choosing
between the two geometries, or as a basis for reporting sensitivity to both choices when
there is no clear prior reason to prefer one.
\begin{remark}
Multiplying \(G\) by a positive scalar rescales the Wasserstein distance but does not
change the relative directional weights encoded by the transport geometry. Specifically,
if \(G_c=cG\) with \(c>0\), then
$ W_{2,G_c}=\sqrt c\,W_{2,G}$ and $ \rho_{G_c}(Q)=c^{-1}\rho_G(Q)$.
The RWPI-calibrated radius rescales accordingly, so that the product
\(\widehat\delta_{n,\alpha}\sqrt{\rho_G(Q)}\) appearing in the surrogate criterion is
unchanged up to the corresponding calibration. Hence the optimiser depends on the relative
geometry encoded by \(G\), not on its arbitrary overall scale. In implementation we
therefore normalise each \(\widehat G_n\) so that its scale is stable across samples.
\end{remark}
Other adaptive geometries may be considered when domain knowledge or a pilot analysis
identifies specific directions of possible distributional shift.
\subsection{Implementation summary}
\label{subsec:implementation_summary}

We summarise the proposed adaptive DRO-PCA procedure. The description is intentionally
algorithmic but solver-agnostic: different numerical methods may be used to minimise the
surrogate criterion in the final step.

\begin{enumerate}[label=\textbf{Step \arabic*.}, leftmargin=*]

\item \textbf{Initial centring and empirical covariance.}
Given observations \(X_1,\ldots,X_n\in\mathbb R^p\), centre the sample and compute
\(
    \widehat\Sigma_n
    =
    n^{-1}\sum_{i=1}^n X_iX_i^\top .
\)

\item \textbf{Preliminary PCA fit.}
Compute the ordinary rank-\(r\) PCA estimator
\[
    \widehat U_{\mathrm{PCA}}\in\mathbb V_{p,r},
    \qquad
    \widehat P_{\mathrm{PCA}}
    =
    \widehat U_{\mathrm{PCA}}\widehat U_{\mathrm{PCA}}^\top,
    \qquad
    \widehat Q_{\mathrm{PCA}}
    =
    I_p-\widehat P_{\mathrm{PCA}} .
\]
This initial fit is used both to construct interpretable adaptive geometries and to obtain
a preliminary subspace for the RWPI plug-in calibration.

\item \textbf{Choose an adaptive transport geometry.}
Construct a symmetric positive definite transport matrix \(\widehat G_n\). In the numerical
studies we consider two diagonal choices. The residual variance geometry
\(\widehat G_n^{\mathrm{res}}\) assigns smaller transport cost to coordinates with large
empirical residual variation, while the PCA-block geometry \(\widehat G_n^{\mathrm{pca}}\)
assigns smaller transport cost to coordinates with large empirical variation inside the
initial PCA block. More generally, any positive definite \(\widehat G_n\) may be used,
provided it encodes the directions of perturbation regarded as plausible and is suitably
stabilised.

\item \textbf{Calibrate the Wasserstein radius.}
Using the preliminary subspace and the adaptive geometry \(\widehat G_n\), compute the
plug-in RWPI quantities
\[
    \widehat\Sigma_{h,n},
    \qquad
    \widehat A_n^{(\widehat G)},
    \qquad
    \widehat B_n^{(\widehat G)}
    =
    \widehat\Sigma_{h,n}^{1/2}
    \{\widehat A_n^{(\widehat G)}\}^{-1}
    \widehat\Sigma_{h,n}^{1/2},
\]
as defined in Section~\ref{sec:radius_calibration}. Let
\(\widehat\lambda_{1,n},\ldots,\widehat\lambda_{rs,n}\) be the eigenvalues of
\(\widehat B_n^{(\widehat G)}\), and let \(\widehat q_{1-\alpha}\) be the conditional
\((1-\alpha)\)-quantile of \(\sum_{j=1}^{rs}\widehat\lambda_{j,n}\chi_{1,j}^2 .\)
Set
\(
    \widehat\delta_{n,\alpha}
    =
    \sqrt{\frac{\widehat q_{1-\alpha}}{n}} .
\)

\item \textbf{Minimise the adaptive surrogate.}
Compute
\[
    \widehat Q_n^{\mathrm{ad}}
    \in
    \arg\min_{Q\in\mathcal G_s}
    \left[
        \sqrt{\operatorname{tr}(\widehat\Sigma_nQ)}
        +
        \widehat\delta_{n,\alpha}
        \sqrt{
            \lambda_{\max}
            (\widehat G_n^{-1/2}Q\widehat G_n^{-1/2})
        }
    \right].
\]
Equivalently, one may optimise over \(U\in\mathbb V_{p,r}\) with
\(Q=I_p-UU^\top\).

\item \textbf{Return the fitted subspace.}
The adaptive DRO-PCA estimate of the residual projector is
\(\widehat Q_n^{\mathrm{ad}}\), and the fitted rank-\(r\) principal projector is
$\widehat\Pi_n^{\mathrm{ad}}=I_p-\widehat Q_n^{\mathrm{ad}}$.

\end{enumerate}

The role of the two calibration components is distinct. The geometry \(\widehat G_n\)
determines which perturbation directions are inexpensive and hence which residual
directions are penalised by the exposure term. The RWPI radius
\(\widehat\delta_{n,\alpha}\) determines the magnitude of this penalty on the local
statistical scale. Thus the method combines geometry-aware regularisation with a
data-driven radius calibration.
\begin{remark}
\label{rem:sample_standardisation}
The theory is stated for centred observations. In applications, one may instead first
standardise the data, for example by applying
$\widetilde X_i=\widehat D_n^{-1}(X_i-\bar X_n)$, $i=1,\ldots,n$,
where \(\widehat D_n\) is a diagonal matrix of sample marginal scales. The adaptive DRO-PCA
procedure can then be applied directly to the transformed observations
\(\widetilde X_1,\ldots,\widetilde X_n\). In fixed dimension, under positive marginal
variances and standard moment conditions, the sample centring and scaling are
root-\(n\) consistent, so the same arguments and 
proof strategy can apply on the standardised scale.
\end{remark}

\section{Numerical studies}
\label{sec:numerical}

This section studies the finite-sample behaviour of adaptive DRO-PCA in four settings:
same-distribution out-of-sample reconstruction, controlled covariance shift, training
contamination, and a real-data source--target reconstruction task. The numerical studies focus on the two adaptive geometries introduced in
Section~\ref{sec:adaptive_transport_geometry}: the residual variance geometry
\(\widehat G_n^{\mathrm{res}}\) in \eqref{eq:supp_residual_G} and the PCA-block variance
geometry \(\widehat G_n^{\mathrm{pca}}\) in \eqref{eq:pca_block_G}.

Additional heavy-tailed \(t_5\) analogues with several robust PCA competitors of the main simulation designs are reported in Supplementary
Section~\ref{supp:sec_t5_additional_numerics}.
\subsection{Simulation protocol and performance measures}
\label{subsec:simulation_protocol}

For a fitted rank-\(r\) basis \(\widehat U\), the population target reconstruction risk is
\[
    R_{\mathrm{tar}}(\widehat U)
    =
    \operatorname{tr}
    \left[
        \Sigma_{\mathrm{tar}}
        \{I_p-\widehat U\widehat U^\top\}
    \right].
\]
Let \(U_{\mathrm{tar}}\) denote the oracle rank-\(r\) PCA basis of
\(\Sigma_{\mathrm{tar}}\). We also report the target excess risk
\[
    \operatorname{Excess}(\widehat U)
    =
    R_{\mathrm{tar}}(\widehat U)-R_{\mathrm{tar}}(U_{\mathrm{tar}}),
\]
and the percentage target-risk gain over ordinary PCA, 
\[
100\,
    \left (\frac{
        R_{\mathrm{tar}}(\widehat U_{\mathrm{PCA}})
        -
        R_{\mathrm{tar}}(\widehat U)
    }{
        R_{\mathrm{tar}}(\widehat U_{\mathrm{PCA}}) 
    }\right )\% .
\]
Positive values indicate lower target reconstruction risk than ordinary PCA fitted from
the same training sample. We additionally report win rates against PCA across Monte Carlo
replications when useful.

All DRO estimators are computed from the adaptive surrogate criterion
\eqref{eq:main_surrogate_objective_adaptive}. In the numerical studies we use the
deterministic path-based implementation described in Supplementary
Section~\ref{sec:supp_numerical_optimisation}, which generates candidate subspaces and selects
the one minimising the original surrogate criterion. Throughout the numerical studies, the
RWPI radius is calibrated at level \(\alpha=0.10\).
\subsection{Same-distribution out-of-sample performance}
\label{subsec:same_distribution_outsample}

The preceding discussion motivates adaptive DRO-PCA primarily through distributional
perturbations. We first consider a complementary same-distribution setting in which the
training and test samples are independent draws from the same law. This experiment is not
intended to demonstrate robustness to a literal source--target shift. Instead, it
illustrates the finite-sample regularisation effect of the adaptive surrogate criterion.
We set \(p=20\), \(r=3\), and \(n=200\). Let \(e_j\) denote the \(j\)th standard coordinate
vector in \(\mathbb R^{20}\), and let \(V_0=(e_1,e_2,e_3)\in\mathbb R^{20\times 3}.\)

The population oracle subspace is \(\operatorname{span}(V_0)\). Training and test
observations are generated independently from \(    X_i^{\mathrm{train}},X_i^{\mathrm{test}}
    \overset{\mathrm{i.i.d.}}{\sim}
    N(0,\Sigma),\)
where
\begin{equation}
\label{eq:same_dist_dense_nuisance}
    \Sigma
    =
    I_p
    +
    V_0\operatorname{diag}(5.5,5.0,4.5)V_0^\top
    +
    B\operatorname{diag}(3.4,3.1,2.8,2.5,2.2)B^\top .
\end{equation}
Here \(B\in\mathbb R^{20\times 5}\) is a random orthonormal matrix supported on
\(\operatorname{span}(e_4,\ldots,e_{20})\). The leading population eigenspace is therefore
\(\operatorname{span}(V_0)\), while the dense lower-variance component creates nuisance
directions that can compete with the signal directions in finite samples. The oracle
population reconstruction risk is \(31.00\).

All methods are fitted from the centred training sample and evaluated on an independent
test sample from the same distribution. For a fitted basis \(\widehat U\), we also report
the population risk
$R_{\Sigma}(\widehat U)
    =
    \operatorname{tr}\{\Sigma(I_p-\widehat U\widehat U^\top)\}
$
and the oracle projector distance
\(
\left\|
        \widehat U\widehat U^\top
        -
        V_0V_0^\top
    \right\|_F.
    \)
These quantities are particularly informative in this same-distribution design because the
population target subspace is known.

\begin{table}[t]
\centering
\small
\setlength{\tabcolsep}{4pt}
\caption{Same-distribution out-of-sample experiment under the dense-nuisance design
\eqref{eq:same_dist_dense_nuisance}. Entries are Monte Carlo averages over 1000
replications. Test risk is reported with Monte Carlo standard error in parentheses. Gain is
the percentage reduction in independent test reconstruction risk relative to ordinary PCA.}
\label{tab:same_dist_dense_nuisance}
\begin{tabular}{lccccc}
\hline
Method
& Test risk
& Gain
& Win rate
& Population risk
& Projector distance \\
\hline
PCA
& 32.268 {\scriptsize (0.015)}
& 0.00
& --
& 32.113
& 1.010 \\
\(\widehat G_n^{\mathrm{res}}\)
& 32.454 {\scriptsize (0.017)}
& -0.58
& 0.113
& 32.300
& 1.093 \\
\(\widehat G_n^{\mathrm{pca}}\)
& 31.388 {\scriptsize (0.011)}
& 2.72
& 0.994
& 31.237
& 0.403 \\
\hline
\end{tabular}
\end{table}

Table~\ref{tab:same_dist_dense_nuisance} shows that the PCA-block geometry improves
independent test reconstruction in this same-distribution design. The improvement is also
visible at the population level: the average population risk decreases from \(32.113\) for
ordinary PCA to \(31.237\) for adaptive DRO-PCA with
\(\widehat G_n^{\mathrm{pca}}\), and the average projector distance to the oracle subspace
decreases from \(1.010\) to \(0.403\). Thus the improvement is not merely due to test-sample
noise; the fitted subspace is, on average, closer to the population PCA subspace. This
behaviour is consistent with the diagnostic interpretation in
Section~\ref{sec:adaptive_transport_geometry}: when the population training and test
distributions coincide and the signal directions are already visible in the empirical PCA
block, \(\widehat G_n^{\mathrm{pca}}\) can stabilise the finite-sample PCA fit.

The residual geometry does not improve performance in this design. This is consistent with
the interpretation in Section~\ref{sec:adaptive_transport_geometry}. The relevant signal
directions are already visible through the empirical PCA block, whereas the residual block
mainly contains nuisance and noise variation. The PCA-block geometry therefore reinforces
the finite-sample signal structure, while the residual geometry can over-emphasise
variation left outside the initial PCA fit. This example shows that adaptive DRO-PCA may improve independent out-of-sample
reconstruction even when training and test distributions coincide, and that this
improvement can be accompanied by a fitted subspace closer to the oracle population
subspace. The gain therefore reflects geometry-aware finite-sample regularisation rather
than protection against an explicit distributional shift.
\subsection{Visibility of target-relevant directions under covariance shift}
\label{subsec:visibility_shift_design}

We next consider a covariance-shift experiment designed to isolate the role of the
adaptive transport geometry under source--target distributional shift. The design separates two features: how visible the future
target-relevant directions are in the source PCA block, and how strongly the target
covariance shifts toward those directions.
We set \(p=20\), \(r=3\), and \(n=100\). Let \(e_j\) denote the \(j\)th standard coordinate
vector and let \(V_0=(e_1,e_2,e_3)\in\mathbb R^{20\times 3}\)
be an orthonormal basis for the target-shift coordinate subspace
\(\mathcal V_0=\operatorname{span}(e_1,e_2,e_3)\). This is the subspace in which the
target covariance will later be amplified. For each Monte Carlo replication, let
\(A\in\mathbb R^{20\times 3}\) be a random orthonormal matrix supported on
\(\operatorname{span}(e_4,\ldots,e_{20})\). Hence \(V_0^\top A=0\). For a visibility
parameter \(\ell\in[0,1]\), define \(    U_\ell
    =
    \sqrt{\ell}\,V_0+\sqrt{1-\ell}\,A .\)

The columns of \(U_\ell\) are orthonormal. Geometrically, each source spike direction is a
rotation between a coordinate direction in \(\mathcal V_0\) and an orthogonal direction
outside \(\mathcal V_0\). In particular, \(\ell\) is the squared cosine overlap between the
source spike direction and the corresponding direction in \(\mathcal V_0\). Thus small
\(\ell\) means that the future target-relevant directions are mostly outside the leading
source spike directions, whereas larger \(\ell\) makes those directions increasingly
visible in the source PCA block.
The source covariance is
\begin{equation}
\label{eq:visibility_source_cov}
    \Sigma_{\mathrm{src}}(\ell)
    =
    I_p
    +
    U_\ell
    \operatorname{diag}(5.0,4.5,4.0)
    U_\ell^\top
    +
    V_0
    \operatorname{diag}(3.6,3.4,3.2)
    V_0^\top .
\end{equation}
The second signal component ensures that the target-shift subspace is present in the
source covariance for all values of \(\ell\). However, because this component has smaller
spikes than the \(U_\ell\) component and the fitted rank is \(r=3\), the parameter
\(\ell\) determines whether information about \(\mathcal V_0\) appears mainly in the
leading source PCA block or remains in the residual block.

The target covariance is
\begin{equation}
\label{eq:visibility_target_cov}
    \Sigma_{\mathrm{tar}}(\ell,\eta)
    =
    \Sigma_{\mathrm{src}}(\ell)
    +
    4\eta V_0V_0^\top,
    \qquad
    \eta\in\{0,0.5,1.0,1.5\}.
\end{equation}
Thus \(\eta\) controls the strength of the future covariance shift toward
\(\mathcal V_0\), while \(\ell\) controls how visible that same subspace is in the source
PCA structure. The training data are generated from
\(N\{0,\Sigma_{\mathrm{src}}(\ell)\}\), while target risk is evaluated under
\(\Sigma_{\mathrm{tar}}(\ell,\eta)\). We consider \(\ell\in\{0,0.05,0.15,0.30,0.50\},\)
and report Monte Carlo averages over 1000 replications.

Table~\ref{tab:visibility_diagnostics} verifies that the visibility parameter has the
intended effect. We report the empirical overlap \(    \operatorname{tr}
    (\widehat U_{\mathrm{PCA}}^\top V_0V_0^\top\widehat U_{\mathrm{PCA}}),\)
which measures the extent to which the ordinary source PCA subspace overlaps with the
target-shift subspace. We also report the coordinate contributions of \(\mathcal V_0\) to
the empirical PCA block and residual block: \[    \sum_{j=1}^3
    \left[
        \widehat P_{\mathrm{PCA}}
        \widehat\Sigma_n
        \widehat P_{\mathrm{PCA}}
    \right]_{jj},
    \sum_{j=1}^3
    \left[
        \widehat Q_{\mathrm{PCA}}
        \widehat\Sigma_n
        \widehat Q_{\mathrm{PCA}}
    \right]_{jj}.\]
As \(\ell\) increases, the target-shift subspace becomes increasingly visible in the
source PCA block and less visible in the residual block.

\begin{table}[t]
\centering
\scriptsize
\setlength{\tabcolsep}{4pt}
\caption{Visibility diagnostics for the covariance-shift design. Values are Monte Carlo
averages over 1000 replications.}
\label{tab:visibility_diagnostics}
\begin{tabular}{cccc}
\hline
\(\ell\)
& PCA overlap
& PCA-block sum
& Residual-block sum \\
\hline
0.00 & 0.947 &  5.604 & 7.377 \\
0.05 & 1.184 &  7.535 & 6.212 \\
0.15 & 1.556 & 10.645 & 4.361 \\
0.30 & 1.957 & 14.460 & 2.609 \\
0.50 & 2.315 & 18.444 & 1.330 \\
\hline
\end{tabular}
\end{table}
Table~\ref{tab:visibility_shift_gains} compares the residual geometry
\(\widehat G_n^{\mathrm{res}}\) and the PCA-block geometry
\(\widehat G_n^{\mathrm{pca}}\). Entries are percentage target-risk gains over ordinary
PCA fitted from the same source sample.
\begin{table}[t]
\centering
\scriptsize
\setlength{\tabcolsep}{4pt}
\caption{Visibility-controlled covariance shift. Entries are percentage target-risk gains
over ordinary PCA. Positive values indicate lower target reconstruction risk than PCA.
Monte Carlo averages are over 1000 replications.}
\label{tab:visibility_shift_gains}
\begin{tabular}{ccrrrrr}
\hline
Geometry & \(\eta\)
& \(\ell=0\) & \(0.05\) & \(0.15\) & \(0.30\) & \(0.50\) \\
\hline
\(\widehat G_n^{\mathrm{res}}\) & 0
& -2.72 & -1.06 &  0.42 &  0.05 & -0.17 \\
& 0.5
&  2.47 &  2.65 &  1.62 & -0.25 & -0.50 \\
& 1.0
&  6.43 &  5.53 &  2.56 & -0.51 & -0.80 \\
& 1.5
&  9.56 &  7.84 &  3.32 & -0.74 & -1.06 \\
\hline
\(\widehat G_n^{\mathrm{pca}}\) & 0
& -0.83 & -1.77 & -2.09 & -2.67 & -2.30 \\
& 0.5
&  0.68 &  1.07 &  2.48 &  2.94 &  2.34 \\
& 1.0
&  1.89 &  3.35 &  6.23 &  7.69 &  6.44 \\
& 1.5
&  2.87 &  5.21 &  9.35 & 11.78 & 10.09 \\
\hline
\end{tabular}
\end{table}

The results exhibit a clear transition. When \(\ell\) is close to zero, the
target-relevant directions are weakly represented in the source PCA block but remain
prominent in the residual block. In this regime, the residual geometry is more effective
under positive shift. For example, at \(\ell=0\), the gain of
\(\widehat G_n^{\mathrm{res}}\) increases from \(2.47\%\) at \(\eta=0.5\) to
\(9.56\%\) at \(\eta=1.5\), whereas the PCA-block geometry gives only \(0.68\%\) and
\(2.87\%\), respectively.

As the visibility parameter increases, the ranking reverses. At \(\ell=0.30\), the
target-shift subspace is already strongly represented in the empirical source PCA block.
The PCA-block geometry then gives gains of \(2.94\%\), \(7.69\%\), and \(11.78\%\) at
\(\eta=0.5,1.0,1.5\), while the residual geometry is slightly worse than PCA. At
\(\ell=0.50\), the same pattern persists. These findings are consistent with the
diagnostics in Table~\ref{tab:visibility_diagnostics}: once the target-shift directions
are already represented in the PCA block, the residual covariance no longer contains the
most relevant source information.

The no-shift case \(\eta=0\) in this design also supports the interpretation. When the
source and target covariances coincide and the adaptive geometry is not aligned with a
beneficial finite-sample correction, moving away from ordinary PCA can incur a small cost.
Thus the purpose of the adaptive geometry is not to uniformly dominate PCA, but to induce
a targeted regularisation whose usefulness depends on the covariance structure and the
chosen transport geometry.
\subsection{Training contamination and visibility in the contaminated PCA block}
\label{subsec:training_contamination_visibility}

We next consider training contamination. Unlike the covariance-shift experiment above, the
target distribution is fixed and clean, while the training sample used to fit PCA and
DRO-PCA is contaminated. The purpose of this design is to examine whether the relative
performance of the two geometries can be explained by where the clean target-relevant
directions are located after contamination: inside the empirical PCA block or in the
empirical residual block.
We again set \(p=20\), \(r=3\), and \(n=100\), and use the same matrix \( V_0=(e_1,e_2,e_3)\in\mathbb R^{20\times 3},\)
whose span is the clean target-relevant subspace. The clean covariance is
\begin{equation}
\label{eq:contam_clean_cov}
    \Sigma_{\mathrm{cl}}
    =
    I_p+
    V_0\operatorname{diag}(5.5,5.0,4.5)V_0^\top .
\end{equation}
Thus the clean oracle rank-\(r\) PCA subspace is \(\operatorname{span}(V_0)\). To generate
training contamination, let \(A\in\mathbb R^{20\times 3}\) be a random orthonormal matrix
supported on \(\operatorname{span}(e_4,\ldots,e_{20})\). For
\(\kappa\in[0,1]\), define \(    C_\kappa
    =
    \sqrt{\kappa}\,V_0+\sqrt{1-\kappa}\,A .\)

The parameter \(\kappa\) controls the alignment between the contamination subspace and the
clean target-relevant subspace. When \(\kappa=0\), contamination is orthogonal to the
clean subspace; when \(\kappa=1\), contamination is fully aligned with it.
For each replication, clean observations \(    X_1^{\mathrm{cl}},\ldots,X_n^{\mathrm{cl}}
    \overset{\mathrm{i.i.d.}}{\sim}
    N(0,\Sigma_{\mathrm{cl}})\)
are first generated. A proportion \(\epsilon\) of the observations is then contaminated by
\[
    X_i^{\mathrm{obs}}
    =
    X_i^{\mathrm{cl}}
    +
    a_{\mathrm{cont}} C_\kappa z_i,
    \qquad
    z_i\sim N(0,I_r),
\]
where \(a_{\mathrm{cont}}=8\). The remaining observations are left unchanged. The observed
training sample is centred before fitting. We consider 
\[
\kappa\in\{0,0.10,0.30,0.60,1.00\},  \quad
    \epsilon\in\{0,0.05,0.10,0.15,0.20\}.
 \]
All methods are fitted using the contaminated training sample, but target risk is evaluated
under the clean covariance \(\Sigma_{\mathrm{cl}}\). Results are Monte Carlo averages over
1000 replications.

Table~\ref{tab:contam_visibility_diagnostics} reports \(    \operatorname{tr}
    \left(
        \widehat U_{\mathrm{PCA}}^\top
        V_0V_0^\top
        \widehat U_{\mathrm{PCA}}
    \right),\) which is the average overlap between the clean
target subspace and the empirical PCA subspace fitted from the contaminated training
sample.
Since both subspaces have rank \(3\), the maximum possible value is \(3\). When
\(\kappa\) is small, increasing contamination pushes the clean subspace out of the
contaminated empirical PCA block. In contrast, when \(\kappa=1\), the contamination is
aligned with the clean signal and the clean subspace remains almost fully visible in the
empirical PCA block.

\begin{table}[t]
\centering
\scriptsize
\setlength{\tabcolsep}{4pt}
\caption{Training contamination diagnostics. Entries are Monte Carlo averages of the
overlap between the clean oracle subspace and the contaminated empirical PCA subspace.}
\label{tab:contam_visibility_diagnostics}
\begin{tabular}{crrrrr}
\hline
\(\kappa\)
& \(\epsilon=0\)
& \(0.05\)
& \(0.10\)
& \(0.15\)
& \(0.20\) \\
\hline
0.00 & 2.873 & 2.037 & 1.329 & 0.818 & 0.410 \\
0.10 & 2.873 & 2.208 & 1.634 & 1.186 & 0.877 \\
0.30 & 2.873 & 2.492 & 2.109 & 1.785 & 1.563 \\
0.60 & 2.874 & 2.704 & 2.531 & 2.393 & 2.279 \\
1.00 & 2.872 & 2.918 & 2.943 & 2.957 & 2.965 \\
\hline
\end{tabular}
\end{table}

Table~\ref{tab:contamination} compares the residual and PCA-block geometries.
Entries are percentage target-risk gains relative to ordinary PCA fitted from the same
contaminated sample.

\begin{table}[t]
\centering
\scriptsize
\setlength{\tabcolsep}{4pt}
\caption{Training contamination experiment. Entries are percentage target-risk gains over
ordinary PCA fitted from the same contaminated training sample. Positive values indicate
lower clean target reconstruction risk than PCA. Monte Carlo averages are over 1000
replications.}
\label{tab:contamination}
\begin{tabular}{ccrrrrr}
\hline
Geometry & \(\epsilon\)
& \(\kappa=0\) & \(0.10\) & \(0.30\) & \(0.60\) & \(1.00\) \\
\hline
\(\widehat G_n^{\mathrm{res}}\) & 0
& -0.26 & -0.28 & -0.28 & -0.28 & -0.29 \\
& 0.05
&  0.88 & -0.24 & -1.24 & -0.62 & -0.14 \\
& 0.10
&  8.21 &  3.84 & -0.65 & -0.68 & -0.07 \\
& 0.15
&  8.48 &  6.65 &  0.73 & -0.63 & -0.05 \\
& 0.20
&  4.88 &  5.77 &  1.94 & -0.54 & -0.03 \\
\hline
\(\widehat G_n^{\mathrm{pca}}\) & 0
&  2.43 &  2.40 &  2.42 &  2.38 &  2.43 \\
& 0.05
&  8.80 & 10.41 &  9.62 &  5.56 &  1.44 \\
& 0.10
&  4.07 &  7.97 & 12.43 &  8.21 &  0.88 \\
& 0.15
&  0.91 &  3.67 &  8.56 &  9.21 &  0.58 \\
& 0.20
& -0.32 &  1.43 &  5.39 &  8.93 &  0.43 \\
\hline
\end{tabular}
\end{table}

The results show a clear transition. When contamination is moderate, or when the contamination
subspace remains aligned with the clean subspace, the PCA-block geometry is more effective.
For example, when \(\kappa=0\) and \(\epsilon=0.05\),
\(\widehat G_n^{\mathrm{pca}}\) improves over PCA by \(8.80\%\), while
\(\widehat G_n^{\mathrm{res}}\) gives only \(0.88\%\). In this case the clean subspace is
still substantially visible in the contaminated PCA block.

When contamination is nearly orthogonal to the clean subspace and sufficiently strong, the
ranking reverses. For \(\kappa=0\), the residual geometry gives gains of \(8.21\%\),
\(8.48\%\), and \(4.88\%\) at \(\epsilon=0.10,0.15,0.20\), respectively, while the
PCA-block geometry gives \(4.07\%\), \(0.91\%\), and \(-0.32\%\). The same transition
occurs later for \(\kappa=0.10\). These findings support the interpretation that, under
training contamination, the relevant visibility question is whether the clean oracle
subspace remains in the PCA block of the contaminated empirical covariance or is pushed
into the residual block.

Taken together, the experiments show that the performance of adaptive DRO-PCA is governed by where the target-relevant directions are visible in the training sample. The residual geometry \(\widehat G_n^{\mathrm{res}}\) is useful when such directions remain in the empirical residual variation, whereas the PCA-block geometry \( \widehat G_n^{\mathrm{pca}} \) is useful when they are already partially represented in the empirical PCA block and should be reinforced rather than displaced.

\section{Image Segmentation data}
\label{subsec:image_segmentation_data}

We finally consider a real-data source--target reconstruction experiment based on the UCI
Image Segmentation data . The data consist of image regions
represented by \(p=19\) numerical descriptors, including location, line-density,
edge-contrast and colour-summary variables. The combined training and test files contain
2310 observations, with 330 observations in each of the seven image-region classes. Because
these variables have different numerical meanings and scales, we use source
standardisation before fitting either method.

For an ordered pair of image-region classes, one class is treated as the source population
and the other as the target population. In each replication, a source training sample of
size \(n_{\mathrm{train}}\) is drawn without replacement. Let
\(\widehat\mu_{\mathrm{src}}\) and \(\widehat D_{\mathrm{src}}\) denote the coordinatewise
mean and standard-deviation matrix computed from this source training sample. Both source
and target observations are transformed using \(    X \mapsto \widehat D_{\mathrm{src}}^{-1}(X-\widehat\mu_{\mathrm{src}}).\)
Thus the target data are used only for evaluation, not for estimating the preprocessing
parameters or the fitted subspace. Let \(n_{\mathrm{tar}}\) denote the number of target observations. For a fitted rank-\(r\)
basis \(\widehat U\), we report the empirical target reconstruction risk
\[
    \widehat R_{\mathrm{tar}}(\widehat U)
    =
    \frac{1}{n_{\mathrm{tar}}}
    \sum_{i=1}^{n_{\mathrm{tar}}}
    \left\|
        (I_p-\widehat U\widehat U^\top)
        \widehat D_{\mathrm{src}}^{-1}
        (X_i^{\mathrm{tar}}-\widehat\mu_{\mathrm{src}})
    \right\|_2^2 .
\]
The empirical target PCA oracle is used only to compute excess target risk and is not used
when fitting PCA or DRO-PCA.
Motivated by the visibility simulations, we consider two source--target transfers for
which target-relevant variation is plausibly not fully captured by the leading source PCA
directions: \(    \mathrm{FOLIAGE}\to\mathrm{PATH},
    \qquad
    \mathrm{WINDOW}\to\mathrm{BRICKFACE}.\)
For these examples, we compare ordinary PCA with adaptive DRO-PCA using the residual
variance geometry \(\widehat G_n^{\mathrm{res}}\).
We fix \(r=5\), use the RWPI-calibrated radius with \(\alpha=0.10\), and report Monte
Carlo averages over 1000 random source-training subsamples.

\begin{table}[t]
\centering
\small
\setlength{\tabcolsep}{4pt}
\caption{Image Segmentation source--target reconstruction experiment. Entries are Monte
Carlo averages over 1000 random source-training subsamples. Target risk is reported with
Monte Carlo standard error in parentheses. Gain is the percentage reduction in empirical
target reconstruction risk relative to PCA.}
\label{tab:segmentation_realdata}
\begin{tabular}{c c c l c c c}
\hline
Source & Target & \(n_{\mathrm{train}}\) & Method
& Target risk & Gain & Win rate \\
\hline
FOLIAGE & PATH & 100 & PCA
& 26.653 {\scriptsize (0.208)} & 0.00 & -- \\
        &      &     & \(\widehat G_n^{\mathrm{res}}\)
& 21.560 {\scriptsize (0.195)} & 16.72 & 0.804 \\
        &      & 200 & PCA
& 25.746 {\scriptsize (0.107)} & 0.00 & -- \\
        &      &     & \(\widehat G_n^{\mathrm{res}}\)
& 17.828 {\scriptsize (0.101)} & 29.76 & 0.939 \\
WINDOW  & BRICKFACE & 100 & PCA
& 10.940 {\scriptsize (0.103)} & 0.00 & -- \\
        &           &     & \(\widehat G_n^{\mathrm{res}}\)
&  9.231 {\scriptsize (0.068)} & 11.62 & 0.792 \\
        &           & 200 & PCA
& 10.412 {\scriptsize (0.046)} & 0.00 & -- \\
        &           &     & \(\widehat G_n^{\mathrm{res}}\)
&  8.292 {\scriptsize (0.035)} & 18.62 & 0.899 \\
\hline
\end{tabular}
\end{table}

Table~\ref{tab:segmentation_realdata} shows stable improvements for the residual geometry
in both transfers. In the \(\mathrm{FOLIAGE}\to\mathrm{PATH}\) experiment, increasing the
source training size from 100 to 200 strengthens the gain from \(16.72\%\) to \(29.76\%\),
and the win rate increases from \(0.804\) to \(0.939\). In the
\(\mathrm{WINDOW}\to\mathrm{BRICKFACE}\) experiment, the corresponding gains are
\(11.62\%\) and \(18.62\%\). These results suggest that, for these standardised image
features, target-relevant variation is not fully captured by the leading source PCA
directions but is reflected in the residual coordinatewise variation used by
\(\widehat G_n^{\mathrm{res}}\). The example should be interpreted as evidence of improved
transferability in this source--target reconstruction task, not as a claim of uniform
dominance over PCA.
\section{Discussion}
\label{sec:discussion}

This paper develops a weighted Wasserstein DRO formulation of PCA. The main preliminary conclusion is
that the transport geometry is essential: Euclidean Wasserstein robustification preserves
the ordinary PCA ordering, whereas anisotropic costs induce a genuine directional
regularisation through the residual exposure term. With a data-adaptive geometry and an
RWPI-calibrated radius, the surrogate estimator remains consistent for the population PCA
subspace, but its root-\(n\) limit contains an explicit Wasserstein-induced drift. Thus the
robustification acts on the local statistical scale rather than changing the nominal
population target.

The numerical results support this interpretation. The residual and PCA-block geometries
are useful in different regimes, depending on whether the relevant variation is visible in
the empirical residual block or already represented in the empirical PCA block. In the
same-distribution experiment, the PCA-block geometry acts as a finite-sample stabiliser
when signal directions are already visible but compete with nuisance directions. Under
covariance shift or contamination, the same principle determines whether the residual or
PCA-block geometry is preferable. Hence adaptive DRO-PCA should not be expected to
uniformly dominate ordinary PCA; its benefit depends on the alignment between the chosen
transport geometry and the relevant finite-sample or distributional perturbation
structure.

Several extensions remain open. The present theory is fixed-dimensional and local;
extending it to high-dimensional regimes is outside the scope of this paper and would require additional control of eigenspaces, adaptive
geometries and RWPI calibration, therefore it is left for future work.
A finer comparison between the exact DRO objective
and the surrogate criterion is also an interesting direction for future work.

 \bibliography{bibliography.bib}

\clearpage

\section*{Supplementary material for Distributionally Robust PCA with Data-Adaptive Wasserstein Geometry}

\setcounter{theorem}{0}
\setcounter{lemma}{0}
\setcounter{proposition}{0}
\setcounter{corollary}{0}

\renewcommand{\thetheorem}{S\arabic{theorem}}
\renewcommand{\thelemma}{S\arabic{lemma}}
\renewcommand{\theproposition}{S\arabic{proposition}}
\renewcommand{\thecorollary}{S\arabic{corollary}}

\section{Overview and notation}
\label{supp:sec_overview_notation}

This supplementary material contains the technical details and additional numerical
evidence supporting the main paper. Section~\ref{supp:sec_weighted_dual_full_sec} gives the
finite-sample Wasserstein dual calculations used to prove the Euclidean exact reduction
and the weighted surrogate bound. Section~\ref{supp:sec_rwp_proofs} proves the adaptive
RWPI limit and the consistency of the plug-in weighted chi-square calibration. 
Section~\ref{supp:sec_local_theory} develops the local Grassmannian expansions underlying
the consistency and central limit theory for the adaptive surrogate estimator.
Section~\ref{supp:sec_exact_consistency} proves consistency of the exact adaptive DRO
minimiser. Section~\ref{supp:sec_adaptive_geometries} verifies, in fixed dimension, the
operator-norm stability and positive definiteness of the two diagonal adaptive transport
geometries used in the numerical studies. Section~\ref{supp:sec_algorithm} describes
the numerical implementation. Finally, Section~\ref{sec:supp_additional_numerical}
reports additional numerical study summaries.

Unless otherwise stated, notation follows the main paper. The observations are denoted by
\(X_1,\ldots,X_n\in\mathbb R^p\), with empirical measure
\[
    \widehat{\mathbb P}_n
    =
    \frac1n\sum_{i=1}^n\delta_{X_i},
    \qquad
    \widehat\Sigma_n
    =
    \frac1n\sum_{i=1}^n X_iX_i^\top .
\]
For \(1\le r<p\), we write \(s=p-r\), denote the Stiefel manifold by
\[
    \mathbb V_{p,r}
    =
    \{U\in\mathbb R^{p\times r}:U^\top U=I_r\},
\]
and represent the Grassmannian as the set of orthogonal projectors
\begin{align}
\label{def:supplementary_Grassmanian}
        \mathcal G_m
    =
    \{Q\in\mathbb R^{p\times p}: Q^\top=Q,\ Q^2=Q,\ \operatorname{tr}(Q)=m\}.
\end{align}
For \(U\in\mathbb V_{p,r}\), we use
\[
    \Pi_U=UU^\top,
    \qquad
    Q_U=I_p-\Pi_U
\]
for the principal and residual projectors, respectively. The population leading and
residual projectors are denoted by \(\Pi_\star\) and \(Q_\star\).

For a positive definite transport matrix \(G\), the weighted quadratic cost is
\[
    c_G(x,y)=(x-y)^\top G(x-y),
\]
and \(W_{2,G}\) denotes the corresponding Wasserstein distance. All Wasserstein ambiguity
sets are understood to be taken over probability measures with finite second moment. When
writing \(\mathbb E_{\mathbb P}f(Z)\), we understand \(Z\sim\mathbb P\).

For a residual projector \(Q\in\mathcal G_s\), we use
\[
    \mathsf S_G(Q)=G^{-1/2}QG^{-1/2},
    \qquad
    \rho_G(Q)=\lambda_{\max}\{\mathsf S_G(Q)\}
\]
for the exposure operator and its largest eigenvalue, respectively. The exact weighted Wasserstein
DRO objective function is denoted by \(\Phi_{n,\delta}^G(Q)\), while
\(\mathcal J_{n,\delta}^G(Q)\) denotes the surrogate criterion. For the adaptive method,
\(\widehat G_n\) denotes the data-dependent transport geometry and
\(\widehat\delta_{n,\alpha}\) denotes the RWPI-calibrated radius.

We use \(\|\cdot\|_{\operatorname{op}}\) and \(\|\cdot\|_F\) for the operator and
Frobenius norms, respectively. If \(A\) is symmetric,
\(\lambda_{\min}(A)\) and \(\lambda_{\max}(A)\) denote its smallest and largest
eigenvalues. The notation \(\operatorname{vec}(A)\) denotes column-wise vectorisation.
Convergence in probability and convergence in distribution are written as
\(\xrightarrow{p}\) and \(\Rightarrow\), respectively.
\section{Wasserstein duality and surrogate bounds: proofs of Theorem \ref{thm:main_euclidean_exact} and \ref{thm:main_weighted_upper}}
\label{supp:sec_weighted_dual_full_sec}
\subsection{Euclidean transport cost: exact reduction to classical PCA}
\label{supp:sec_euclidean_dual}

We prove Theorem~\ref{thm:main_euclidean_exact}. Throughout this subsection,
\(Q\in\mathcal G_s\) defined in \eqref{def:supplementary_Grassmanian} is fixed. By Wasserstein strong duality from \cite{blanchet2019modelrisk,gao2023},
\begin{equation}
\label{eq:supp_dual_euclidean}
\sup_{\mathbb P:W_2(\mathbb P,\widehat{\mathbb P}_n)\le\delta}
\mathbb E_{\mathbb P}(Z^\top QZ)
=
\inf_{\lambda\ge 0}
\left\{
\lambda\delta^2
+
\frac1n\sum_{i=1}^n
\sup_{y\in\mathbb R^p}
\bigl(
y^\top Qy-\lambda\|X_i-y\|_2^2
\bigr)
\right\}.
\end{equation}

\begin{lemma}
\label{lem:supp_euclidean_inner}
For any \(Q\in\mathcal G_s\) and any \(x\in\mathbb R^p\),
\[
\sup_{y\in\mathbb R^p}
\bigl(
y^\top Qy-\lambda\|x-y\|_2^2
\bigr)
=
\begin{cases}
\dfrac{\lambda}{\lambda-1}x^\top Qx, & \lambda>1,\\[1ex]
+\infty, & 0\le \lambda\le 1.
\end{cases}
\]
\end{lemma}

\begin{proof}
Since \(Q\) is an orthogonal projector, \(I_p-Q\) is the complementary orthogonal projector.
Decompose
\[
    y=(I_p-Q)y+Qy,
    \qquad
    x=(I_p-Q)x+Qx .
\]
Then
\[
y^\top Qy-\lambda\|x-y\|_2^2
=
-\lambda\|(I_p-Q)x-(I_p-Q)y\|_2^2
+
\Bigl\{
\|Qy\|_2^2-\lambda\|Qx-Qy\|_2^2
\Bigr\}.
\]
The first term is maximised by taking \((I_p-Q)y=(I_p-Q)x\). Writing
\(v:=Qy\) and \(V:=Qx\), the remaining problem is
\[
\sup_v
\left\{
\|v\|_2^2-\lambda\|V-v\|_2^2
\right\},
\]
where \(v\) ranges over \(\operatorname{range}(Q)\). If \(\lambda\le1\), this quadratic is
not bounded above. If \(\lambda>1\), its maximiser is
\[
    v^\star=\frac{\lambda}{\lambda-1}V,
\]
and the optimal value is
\[
    \frac{\lambda}{\lambda-1}\|V\|_2^2
    =
    \frac{\lambda}{\lambda-1}x^\top Qx .
\]
\end{proof}

\begin{proof}[Proof of Theorem~\ref{thm:main_euclidean_exact}]
By \eqref{eq:supp_dual_euclidean} and Lemma~\ref{lem:supp_euclidean_inner},
\[
\sup_{\mathbb P:W_2(\mathbb P,\widehat{\mathbb P}_n)\le\delta}
\mathbb E_{\mathbb P}(Z^\top QZ)
=
\inf_{\lambda>1}
\left\{
\lambda\delta^2
+
\frac{\lambda}{\lambda-1}\operatorname{tr}(\widehat\Sigma_nQ)
\right\}.
\]
Let \(a_Q:=\operatorname{tr}(\widehat\Sigma_nQ)\ge0\). We claim that
\[
    \inf_{\lambda>1}
    \left\{
        \lambda\delta^2+
        \frac{\lambda}{\lambda-1}a_Q
    \right\}
    =
    (\sqrt{a_Q}+\delta)^2 .
\]
If \(\delta>0\) and \(a_Q>0\), differentiating in \(\lambda\) gives
\[
    \delta^2-\frac{a_Q}{(\lambda-1)^2}=0,
    \qquad
    \lambda^\star=1+\frac{\sqrt{a_Q}}{\delta}.
\]
Substitution yields
\[
    \lambda^\star\delta^2
    +
    \frac{\lambda^\star}{\lambda^\star-1}a_Q
    =
    (\sqrt{a_Q}+\delta)^2 .
\]
If \(a_Q=0\), then
\[
    \inf_{\lambda>1}
    \left\{
        \lambda\delta^2+
        \frac{\lambda}{\lambda-1}a_Q
    \right\}
    =
    \inf_{\lambda>1}\lambda\delta^2
    =
    \delta^2,
\]
where the infimum is attained in the limit \(\lambda\downarrow1\). This agrees with
\((\sqrt{a_Q}+\delta)^2\). If \(\delta=0\), then
\[
    \inf_{\lambda>1}
    \left\{
        \lambda\delta^2+
        \frac{\lambda}{\lambda-1}a_Q
    \right\}
    =
    \inf_{\lambda>1}
    \frac{\lambda}{\lambda-1}a_Q
    =
    a_Q,
\]
where the infimum is attained in the limit \(\lambda\to\infty\). This again agrees with
\((\sqrt{a_Q}+\delta)^2\). Therefore, for all \(\delta\ge0\) and \(a_Q\ge0\),
\[
\sup_{\mathbb P:W_2(\mathbb P,\widehat{\mathbb P}_n)\le\delta}
\mathbb E_{\mathbb P}(Z^\top QZ)
=
\left\{
\sqrt{\operatorname{tr}(\widehat\Sigma_nQ)}+\delta
\right\}^2 .
\]
Since \(x\mapsto(\sqrt{x}+\delta)^2\) is increasing on \([0,\infty)\), the minimisers over
\(\mathcal G_s\) coincide with those of
\(\operatorname{tr}(\widehat\Sigma_nQ)\).
\end{proof}

\subsection{Weighted transport geometry: dual calculation and upper bound}
\label{supp:sec_weighted_dual}

We prove the weighted dual identity used in Section~\ref{sec:dro_pca} and then prove
Theorem~\ref{thm:main_weighted_upper}. Throughout this subsection, \(Q\in\mathcal G_s\) is
fixed and
\[
    \mathsf S_G(Q):=G^{-1/2}QG^{-1/2},
    \qquad
    \rho_G(Q):=\lambda_{\max}\{\mathsf S_G(Q)\},
    \qquad
    \widetilde\Sigma_{n,G}:=G^{1/2}\widehat\Sigma_nG^{1/2}.
\]
By Wasserstein strong duality,
\[
\Phi_{n,\delta}^G(Q)
=
\inf_{\lambda\ge0}
\left\{
\lambda\delta^2
+
\frac1n\sum_{i=1}^n
\phi_{X_i}(\lambda)
\right\},
\]
where
\[
\phi_x(\lambda)
:=
\sup_{y\in\mathbb R^p}
\bigl\{
y^\top Qy-\lambda(x-y)^\top G(x-y)
\bigr\}.
\]
Expanding the quadratic form gives
\[
\phi_x(\lambda)
=
\sup_{y\in\mathbb R^p}
\left[
y^\top(Q-\lambda G)y
+
2\lambda x^\top Gy
-
\lambda x^\top Gx
\right].
\]
The supremum is finite if and only if
\[
    Q-\lambda G\prec0,
\]
which is equivalent to
\[
    \lambda>\lambda_{\max}(G^{-1/2}QG^{-1/2})
    =
    \rho_G(Q).
\]
For such \(\lambda\), the maximiser is
\[
    y^\star=(\lambda G-Q)^{-1}(\lambda Gx),
\]
and therefore
\[
\phi_x(\lambda)
=
\lambda^2x^\top G(\lambda G-Q)^{-1}Gx
-
\lambda x^\top Gx .
\]
Using
\[
    (\lambda G-Q)^{-1}
    =
    G^{-1/2}
    \{\lambda I_p-\mathsf S_G(Q)\}^{-1}
    G^{-1/2},
\]
we obtain
\begin{align*}
\frac1n\sum_{i=1}^n\phi_{X_i}(\lambda)
&=
\lambda^2
\operatorname{tr}
\left[
\widehat\Sigma_nG^{1/2}
\{\lambda I_p-\mathsf S_G(Q)\}^{-1}
G^{1/2}
\right]
-
\lambda\operatorname{tr}(\widehat\Sigma_nG).
\end{align*}
Finally, the matrix identity
\[
    \lambda^2(\lambda I-S)^{-1}-\lambda I
    =
    \lambda S(\lambda I-S)^{-1}
\]
gives
\[
\frac1n\sum_{i=1}^n\phi_{X_i}(\lambda)
=
\lambda
\operatorname{tr}
\left[
\widetilde\Sigma_{n,G}
\mathsf S_G(Q)
\{\lambda I_p-\mathsf S_G(Q)\}^{-1}
\right].
\]
Hence
\[
    \Phi_{n,\delta}^G(Q)
    =
    \inf_{\lambda>\rho_G(Q)}
    \left[
    \lambda\delta^2
    +
    \lambda
    \operatorname{tr}
    \left\{
    \widetilde\Sigma_{n,G}
    \mathsf S_G(Q)
    \bigl(\lambda I_p-\mathsf S_G(Q)\bigr)^{-1}
    \right\}
    \right].
\]

\begin{proof}[Proof of Theorem~\ref{thm:main_weighted_upper}]
Since \(\mathsf S_G(Q)\) is symmetric positive semidefinite, for
\(\lambda>\rho_G(Q)\),
\[
    \mathsf S_G(Q)\{\lambda I_p-\mathsf S_G(Q)\}^{-1}
    \preceq
    \frac{1}{\lambda-\rho_G(Q)}\mathsf S_G(Q).
\]
The weighted dual identity therefore implies
\[
\Phi_{n,\delta}^G(Q)
\le
\inf_{\lambda>\rho_G(Q)}
\left[
\lambda\delta^2
+
\frac{\lambda}{\lambda-\rho_G(Q)}
\operatorname{tr}
\{\widetilde\Sigma_{n,G}\mathsf S_G(Q)\}
\right].
\]
By cyclicity of the trace,
\[
\operatorname{tr}\{\widetilde\Sigma_{n,G}\mathsf S_G(Q)\}
=
\operatorname{tr}
\{G^{1/2}\widehat\Sigma_nG^{1/2}G^{-1/2}QG^{-1/2}\}
=
\operatorname{tr}(\widehat\Sigma_nQ).
\]
Let \(a_Q:=\operatorname{tr}(\widehat\Sigma_nQ)\) and
\(\rho_Q:=\rho_G(Q)\). We then have
\[
\Phi_{n,\delta}^G(Q)
\le
\inf_{\lambda>\rho_Q}
\left\{
\lambda\delta^2
+
\frac{\lambda}{\lambda-\rho_Q}a_Q
\right\}.
\]
We claim that
\[
    \inf_{\lambda>\rho_Q}
    \left\{
        \lambda\delta^2+
        \frac{\lambda}{\lambda-\rho_Q}a_Q
    \right\}
    =
    \left(\sqrt{a_Q}+\delta\sqrt{\rho_Q}\right)^2 .
\]
If \(\delta>0\) and \(a_Q>0\), differentiating in \(\lambda\) gives the interior minimiser
\[
    \lambda^\star
    =
    \rho_Q+\frac{\sqrt{\rho_Qa_Q}}{\delta}.
\]
Substitution gives
\[
    \lambda^\star\delta^2
    +
    \frac{\lambda^\star}{\lambda^\star-\rho_Q}a_Q
    =
    \left(\sqrt{a_Q}+\delta\sqrt{\rho_Q}\right)^2 .
\]
If \(a_Q=0\), then
\[
    \inf_{\lambda>\rho_Q}
    \left\{
        \lambda\delta^2+
        \frac{\lambda}{\lambda-\rho_Q}a_Q
    \right\}
    =
    \inf_{\lambda>\rho_Q}\lambda\delta^2
    =
    \rho_Q\delta^2,
\]
where the infimum is attained in the limit \(\lambda\downarrow\rho_Q\). This agrees with
\((\sqrt{a_Q}+\delta\sqrt{\rho_Q})^2\). If \(\delta=0\), then
\[
    \inf_{\lambda>\rho_Q}
    \left\{
        \lambda\delta^2+
        \frac{\lambda}{\lambda-\rho_Q}a_Q
    \right\}
    =
    \inf_{\lambda>\rho_Q}
    \frac{\lambda}{\lambda-\rho_Q}a_Q
    =
    a_Q,
\]
where the infimum is attained in the limit \(\lambda\to\infty\). This again agrees with
\((\sqrt{a_Q}+\delta\sqrt{\rho_Q})^2\). Hence, for all \(\delta\ge0\) and \(a_Q\ge0\),
\[
    \Phi_{n,\delta}^G(Q)
    \le
    \left\{
    \sqrt{a_Q}+\delta\sqrt{\rho_Q}
    \right\}^2
    =
    \left\{
    \sqrt{\operatorname{tr}(\widehat\Sigma_nQ)}
    +
    \delta\sqrt{\rho_G(Q)}
    \right\}^2 .
\]
\end{proof}

\section{RWPI calibration: proofs of Theorem \ref{thm:main_rwp_limit_randomG} and \ref{thm:main_plugin_quantile}}
\label{supp:sec_rwp_proofs}

The proof is organised in three steps. First, we derive the PCA estimating equation and establish a fixed-\(G\) expansion of the Wasserstein profile statistic. Lemmas~\ref{lem:supp_linearized_minimizer}--\ref{lem:supp_Rbar_lower} show that the Dirac-restricted profile problem has the expansion \(\bar R_n=M_n+o_p(n^{-1})\), and Theorem~\ref{thm:supp_dirac_tight} shows that this restriction is asymptotically tight for the original profile statistic \(R_n^{(G)}\). This yields the fixed-metric limit in Theorem~\ref{thm:supp_rwp_limit}. Second, Lemma~\ref{lem:supp_randomG_metric_equiv} compares the random metric \(W_{2,\widehat G_n}\) with its deterministic limit \(W_{2,G}\), which transfers the fixed-\(G\) limit to the adaptive statistic and proves Theorem~\ref{thm:main_rwp_limit_randomG}. Third, Subsection~\ref{supp:sec_plugin_quantile} proves Theorem~\ref{thm:main_plugin_quantile} by showing that the plug-in weighted chi-square law consistently estimates the limiting calibration distribution.

The calibration argument is motivated by the robust Wasserstein profile inference framework of \citet{blanchet2019rwpi}, in which the radius is chosen through the Wasserstein distance from the empirical law to a set of distributions satisfying an estimating equation. The proof below is specialised to the PCA first-order condition and uses a direct primal perturbation argument.

\subsection{Estimating equation for PCA and RWPI calibration}
\label{supp:sec_rwp_estimating_eq}

We derive the estimating equation used in the Wasserstein profile statistic. Let
\(U\in\mathbb V_{p,r}\), let \(U_\perp\in\mathbb R^{p\times s}\) be an orthonormal
complement, and define
\[
    \Pi_U:=UU^\top,
    \qquad
    Q_U:=I_p-\Pi_U=U_\perp U_\perp^\top .
\]
The population PCA problem is
\[
    \min_{\Pi\in\mathcal G_r}
    \mathbb E\{\|(I_p-\Pi)X\|_2^2\},
\]
or equivalently
\[
    \max_{U\in\mathbb V_{p,r}}
    \operatorname{tr}(U^\top\Sigma U),
    \qquad
    \Sigma:=\mathbb E(XX^\top).
\]
The first-order condition for an \(r\)-dimensional invariant subspace is
\[
    Q_U\Sigma U=0 .
\]

\begin{lemma}
\label{lem:supp_pca_estimating_eq}
Assume \(U^\top U=I_r\). Then
\[
    \Sigma U=UM
    \quad\text{for some }M\in\mathbb R^{r\times r}
    \iff
    (I_p-UU^\top)\Sigma U=0 .
\]
\end{lemma}

\begin{proof}
If \(\Sigma U=UM\), then
\[
    (I_p-UU^\top)\Sigma U=(I_p-UU^\top)UM=0.
\]
Conversely, if \((I_p-UU^\top)\Sigma U=0\), then each column of \(\Sigma U\) lies in
\(\operatorname{span}(U)\), so there exists \(M\in\mathbb R^{r\times r}\) such that
\(\Sigma U=UM\).
\end{proof}

Since
\[
    Q_U\Sigma U
    =
    \mathbb E(Q_UXX^\top U),
\]
a natural matrix-valued estimating equation for the subspace is
\[
    \mathbb E\{U_\perp^\top XX^\top U\}=0 .
\]
We therefore define
\begin{equation}
\label{eq:supp_h_vector}
    h(X;[U,U_\perp])
    :=
    \operatorname{vec}(U_\perp^\top XX^\top U)
    \in\mathbb R^{rs}.
\end{equation}

\begin{lemma}
\label{lem:supp_rotation_invariant_zero_set}
Let \(O\in\mathbb R^{r\times r}\) be orthogonal. Then
\[
    \mathbb E\{h(X;[UO,U_\perp])\}=0
    \iff
    \mathbb E\{h(X;[U,U_\perp])\}=0 .
\]
Hence the zero set of the estimating equation depends on the subspace
\(\operatorname{span}(U)\), not on the particular orthonormal basis.
\end{lemma}

\begin{proof}
Since
\[
    h(X;[UO,U_\perp])
    =
    \operatorname{vec}(U_\perp^\top XX^\top UO),
\]
we have
\[
    \operatorname{vec}(U_\perp^\top XX^\top UO)
    =
    (O^\top\otimes I_s)
    \operatorname{vec}(U_\perp^\top XX^\top U).
\]
Because \(O^\top\otimes I_s\) is invertible, the two expectations vanish simultaneously.
\end{proof}

For the true principal subspace, write
\[
    h(x):=h(x;[U_\star,U_{\star,\perp}]).
\]
The Fr\'echet derivative of \(h\) at \(x\in\mathbb R^p\) is the linear map
\[
    J(x):\mathbb R^p\to\mathbb R^{rs},
    \qquad
    J(x)\delta
    :=
    \operatorname{vec}
    \{U_{\star,\perp}^\top(\delta x^\top+x\delta^\top)U_\star\}.
\]
We identify \(J(x)\) with its representing matrix in \(\mathbb R^{rs\times p}\).

\begin{lemma}
\label{lem:supp_h_derivative}
For all \(x,\delta\in\mathbb R^p\),
\[
    h(x+\delta)=h(x)+J(x)\delta+q(\delta),
\]
where
\[
    q(\delta)
    :=
    \operatorname{vec}
    (U_{\star,\perp}^\top\delta\delta^\top U_\star).
\]
Moreover, there exists a constant \(C>0\), depending only on
\((U_\star,U_{\star,\perp})\), such that
\[
    \|h(x)\|\le \|x\|_2^2,
    \qquad
    \|q(\delta)\|\le \|\delta\|_2^2,
    \qquad
    \|J(x)\|_{\operatorname{op}}\le 2\|x\|_2,
    \qquad
    \|J(x)\|_F\le C\|x\|_2.
\]
\end{lemma}

\begin{proof}
Expanding \((x+\delta)(x+\delta)^\top\) gives
\[
U_{\star,\perp}^\top (x+\delta)(x+\delta)^\top U_\star
=
U_{\star,\perp}^\top xx^\top U_\star
+
U_{\star,\perp}^\top(\delta x^\top+x\delta^\top)U_\star
+
U_{\star,\perp}^\top\delta\delta^\top U_\star .
\]
Vectorising yields the stated expansion. Since \(U_\star\) and \(U_{\star,\perp}\) have
orthonormal columns,
\[
    \|U_{\star,\perp}^\top ab^\top U_\star\|_F
    \le
    \|a\|_2\|b\|_2
    \qquad
    \text{for all }a,b\in\mathbb R^p .
\]
Taking \(a=b=x\) gives \(\|h(x)\|\le\|x\|_2^2\), and taking \(a=b=\delta\) gives
\(\|q(\delta)\|\le\|\delta\|_2^2\). Finally, for any \(\delta\in\mathbb R^p\),
\[
\begin{aligned}
    \|J(x)\delta\|
    &\le
    \|U_{\star,\perp}^\top \delta x^\top U_\star\|_F
    +
    \|U_{\star,\perp}^\top x\delta^\top U_\star\|_F  \\
    &\le
    2\|x\|_2\|\delta\|_2 .
\end{aligned}
\]
Thus \(\|J(x)\|_{\operatorname{op}}\le2\|x\|_2\). Since the dimensions are fixed,
\(\|J(x)\|_F\le C\|x\|_2\) for some constant \(C>0\).
\end{proof}

Throughout the fixed-\(G\) part of this subsection we assume
Assumptions~\ref{ass:main_moment_cov} and \ref{ass:main_rwp_nondeg} from the main text.
For a fixed transport matrix \(G\), define
\[
    R_n^{(G)}
    :=
    \inf_{\mathbb P}
    \left\{
    W_{2,G}^2(\mathbb P,\widehat{\mathbb P}_n):
    \mathbb E_{\mathbb P}\{h(Z)\}=0
    \right\}.
\]
To simplify notation in the fixed-\(G\) arguments, we write \(R_n:=R_n^{(G)}\). Define
\[
    \bar h_{n,\star}:=\frac1n\sum_{i=1}^n h(X_i),
    \qquad
    \Sigma_h:=\operatorname{Var}\{h(X_1)\},
    \qquad
    J_i:=J(X_i),
\]
and
\[
    \widehat A_{n,G}:=\frac1n\sum_{i=1}^n J_iG^{-1}J_i^\top .
\]

\begin{lemma}
\label{lem:supp_rwp_basic_objects}
Under Assumptions~\ref{ass:main_moment_cov} and \ref{ass:main_rwp_nondeg},
\[
    \sqrt n\,\bar h_{n,\star}\Rightarrow N(0,\Sigma_h),
    \qquad
    \widehat A_{n,G}\xrightarrow{p}A_G,
    \qquad
    \widehat A_{n,G}^{-1}=O_p(1).
\]
\end{lemma}

\begin{proof}
By Lemma~\ref{lem:supp_h_derivative},
\[
    \|h(X_1)\|\le \|X_1\|_2^2.
\]
Assumption~\ref{ass:main_moment_cov} implies \(\mathbb E\|X_1\|_2^4<\infty\), hence
\(\mathbb E\|h(X_1)\|^2<\infty\). Since \(\mathbb E\{h(X_1)\}=0\), the multivariate central
limit theorem gives
\[
    \sqrt n\,\bar h_{n,\star}\Rightarrow N(0,\Sigma_h).
\]
Also, Lemma~\ref{lem:supp_h_derivative} implies \(\|J(X_1)\|_F\le C\|X_1\|_2\) for some
constant \(C>0\), and therefore
\[
    \mathbb E\|J(X_1)\|_F^2<\infty.
\]
The law of large numbers gives
\[
    \widehat A_{n,G}
    =
    \frac1n\sum_{i=1}^n J_iG^{-1}J_i^\top
    \xrightarrow{p}
    \mathbb E\{J(X_1)G^{-1}J(X_1)^\top\}
    =
    A_G.
\]
By Assumption~\ref{ass:main_rwp_nondeg}, \(A_G\) is positive definite. Hence
\[
    \lambda_{\min}(\widehat A_{n,G})
    \xrightarrow{p}
    \lambda_{\min}(A_G)>0,
\]
which implies \(\widehat A_{n,G}^{-1}=O_p(1)\).
\end{proof}

\medskip
\noindent\textbf{Dirac upper bound.}
Using the definition of the Wasserstein distance and conditioning on the atoms of
\(\widehat{\mathbb P}_n\), the profile statistic can be written as
\[
    R_n
    =
    \inf_{\mu_1,\dots,\mu_n}
    \left\{
    \frac1n\sum_{i=1}^n\int \|z-X_i\|_G^2\,\mu_i(dz):
    \frac1n\sum_{i=1}^n\int h(z)\,\mu_i(dz)=0
    \right\},
\]
where each \(\mu_i\) ranges over probability measures on \(\mathbb R^p\). Restricting to
Dirac measures \(\mu_i=\delta_{z_i}=\delta_{X_i+\Delta_i}\) yields the upper bound
\begin{equation}
\label{eq:supp_Rbar}
    \bar R_n
    :=
    \inf_{\Delta\in(\mathbb R^p)^n}
    \left\{
    C_n^G(\Delta):F_n(\Delta)=0
    \right\},
\end{equation}
where
\[
    \Delta=(\Delta_1,\ldots,\Delta_n),
    \qquad
    C_n^G(\Delta):=\frac1n\sum_{i=1}^n\Delta_i^\top G\Delta_i,
    \qquad
    F_n(\Delta):=\frac1n\sum_{i=1}^n h(X_i+\Delta_i).
\]
Clearly,
\begin{equation}
\label{eq:supp_R_le_Rbar}
    R_n\le \bar R_n.
\end{equation}

\medskip
\noindent\textbf{Linearised problem.}
By Lemma~\ref{lem:supp_h_derivative},
\[
    F_n(\Delta)
    =
    \bar h_{n,\star}
    +
    \frac1n\sum_{i=1}^n J_i\Delta_i
    +
    \mathcal R_n^{\operatorname{quad}}(\Delta),
\]
where
\[
    \mathcal R_n^{\operatorname{quad}}(\Delta)
    :=
    \frac1n\sum_{i=1}^n q(\Delta_i)
    =
    \frac1n\sum_{i=1}^n
    \operatorname{vec}
    \left(
    U_{\star,\perp}^\top\Delta_i\Delta_i^\top U_\star
    \right).
\]
This motivates the linearised problem
\begin{equation}
\label{eq:supp_linearized_problem}
    R_n^{\operatorname{lin}}
    :=
    \inf_{\Delta\in(\mathbb R^p)^n}
    \left\{
    C_n^G(\Delta):
    \bar h_{n,\star}+\frac1n\sum_{i=1}^nJ_i\Delta_i=0
    \right\}.
\end{equation}

\begin{lemma}
\label{lem:supp_linearized_minimizer}
If \(\widehat A_{n,G}\) is invertible, then the minimiser of
\eqref{eq:supp_linearized_problem} is
\[
    \Delta_i^{(0)}
    =
    -G^{-1}J_i^\top\widehat A_{n,G}^{-1}\bar h_{n,\star},
    \qquad
    1\le i\le n,
\]
and the minimum value is
\begin{equation}
\label{eq:supp_Mn}
    M_n
    :=
    \bar h_{n,\star}^\top
    \widehat A_{n,G}^{-1}
    \bar h_{n,\star}.
\end{equation}
Moreover,
\[
    M_n=O_p(n^{-1}),
    \qquad
    \|\Delta^{(0)}\|_{n,G}=O_p(n^{-1/2}).
\]
\end{lemma}

\begin{proof}
The problem is convex with quadratic objective and linear constraints. Its Lagrangian is
\[
    L(\Delta;\lambda)
    =
    \frac1n\sum_{i=1}^n\Delta_i^\top G\Delta_i
    +
    \lambda^\top
    \left(
    \bar h_{n,\star}
    +
    \frac1n\sum_{i=1}^nJ_i\Delta_i
    \right).
\]
Differentiating with respect to \(\Delta_i\) yields
\[
    2G\Delta_i+J_i^\top\lambda=0,
    \qquad
    \Delta_i=-\frac12G^{-1}J_i^\top\lambda .
\]
Substituting into the constraint gives
\[
    \bar h_{n,\star}-\frac12\widehat A_{n,G}\lambda=0,
    \qquad
    \lambda=2\widehat A_{n,G}^{-1}\bar h_{n,\star}.
\]
Therefore
\[
    \Delta_i^{(0)}
    =
    -G^{-1}J_i^\top\widehat A_{n,G}^{-1}\bar h_{n,\star}.
\]
Substituting this into the objective yields
\[
    M_n
    =
    \bar h_{n,\star}^\top
    \widehat A_{n,G}^{-1}
    \bar h_{n,\star}.
\]
Finally, Lemma~\ref{lem:supp_rwp_basic_objects} gives
\[
    \|\bar h_{n,\star}\|=O_p(n^{-1/2}),
    \qquad
    \widehat A_{n,G}^{-1}=O_p(1),
\]
and hence \(M_n=O_p(n^{-1})\). Since
\[
    M_n=C_n^G(\Delta^{(0)})=\|\Delta^{(0)}\|_{n,G}^2,
\]
we obtain \(\|\Delta^{(0)}\|_{n,G}=O_p(n^{-1/2})\).
\end{proof}

\medskip
\noindent\textbf{Correction lemma.}
For \(\Delta,\Gamma\in(\mathbb R^p)^n\), define
\[
    \langle\Delta,\Gamma\rangle_{n,G}
    :=
    \frac1n\sum_{i=1}^n\Delta_i^\top G\Gamma_i,
    \qquad
    \|\Delta\|_{n,G}:=\sqrt{\langle\Delta,\Delta\rangle_{n,G}}.
\]

\begin{lemma}
\label{lem:supp_correction}
Under Assumptions~\ref{ass:main_moment_cov} and \ref{ass:main_rwp_nondeg}, let
\(\Delta\in(\mathbb R^p)^n\) satisfy
\[
    C_n^G(\Delta)=O_p(n^{-1}),
    \qquad
    \|F_n(\Delta)\|=O_p(n^{-1}).
\]
Then there exists a random vector \(\widetilde\Delta\in(\mathbb R^p)^n\) such that
\[
    \Pr\{F_n(\widetilde\Delta)=0\}\to1,
    \qquad
    C_n^G(\widetilde\Delta)=C_n^G(\Delta)+o_p(n^{-1}),
\]
and
\[
    \|\widetilde\Delta-\Delta\|_{n,G}=O_p(n^{-1}).
\]
\end{lemma}

\begin{proof}
Since \(C_n^G(\Delta)=\|\Delta\|_{n,G}^2=O_p(n^{-1})\), we have
\[
    \|\Delta\|_{n,G}=O_p(n^{-1/2}).
\]
Define the linear operator
\[
    L_n(\eta):=\frac1n\sum_{i=1}^nJ_i\eta_i,
    \qquad
    \eta=(\eta_1,\ldots,\eta_n)\in(\mathbb R^p)^n.
\]
For \(b\in\mathbb R^{rs}\), define
\[
    (K_nb)_i:=G^{-1}J_i^\top\widehat A_{n,G}^{-1}b,
    \qquad
    1\le i\le n.
\]
Then
\[
    L_n(K_nb)
    =
    \frac1n\sum_{i=1}^nJ_iG^{-1}J_i^\top\widehat A_{n,G}^{-1}b
    =
    \widehat A_{n,G}\widehat A_{n,G}^{-1}b
    =
    b,
\]
so \(K_n\) is a right inverse of \(L_n\). Moreover,
\[
\begin{aligned}
    \|K_nb\|_{n,G}^2
    &=
    \frac1n\sum_{i=1}^n
    (K_nb)_i^\top G(K_nb)_i  \\
    &=
    b^\top\widehat A_{n,G}^{-1}
    \left(
    \frac1n\sum_{i=1}^nJ_iG^{-1}J_i^\top
    \right)
    \widehat A_{n,G}^{-1}b
    =
    b^\top\widehat A_{n,G}^{-1}b.
\end{aligned}
\]
Hence, on any event on which \(\lambda_{\min}(\widehat A_{n,G})\ge c_0>0\),
\[
    \|K_nb\|_{n,G}\le c_0^{-1/2}\|b\|.
\]
For \(\eta\in(\mathbb R^p)^n\), define
\[
    B_n(\Delta,\eta)
    :=
    \frac1n\sum_{i=1}^n
    \operatorname{vec}
    \left(
    U_{\star,\perp}^\top
    (\Delta_i\eta_i^\top+\eta_i\Delta_i^\top)
    U_\star
    \right),
\]
and
\[
    \mathcal R_n^{\operatorname{quad}}(\eta)
    :=
    \frac1n\sum_{i=1}^n
    \operatorname{vec}
    \left(
    U_{\star,\perp}^\top\eta_i\eta_i^\top U_\star
    \right).
\]
Then
\[
    F_n(\Delta+\eta)
    =
    F_n(\Delta)+L_n(\eta)+B_n(\Delta,\eta)
    +
    \mathcal R_n^{\operatorname{quad}}(\eta).
\]
Because \(G\) is positive definite, there exists a constant \(c_G>0\) such that
\[
    \|v\|_2\le c_G\|v\|_G,
    \qquad v\in\mathbb R^p.
\]
Also, since \(U_\star\) and \(U_{\star,\perp}\) have orthonormal columns,
\[
    \left\|
    \operatorname{vec}(U_{\star,\perp}^\top ab^\top U_\star)
    \right\|
    \le
    \|a\|_2\|b\|_2
    \qquad
    \text{for all }a,b\in\mathbb R^p.
\]
Therefore, for all \(\eta,\eta'\in(\mathbb R^p)^n\),
\begin{align}
    \|\mathcal R_n^{\operatorname{quad}}(\eta)\|
    &\le
    \frac1n\sum_{i=1}^n\|\eta_i\|_2^2
    \le
    c_G^2\|\eta\|_{n,G}^2, \label{eq:supp_quad_bound}
    \\
    \|B_n(\Delta,\eta)\|
    &\le
    \frac2n\sum_{i=1}^n\|\Delta_i\|_2\|\eta_i\|_2
    \le
    2c_G^2\|\Delta\|_{n,G}\|\eta\|_{n,G}, \label{eq:supp_bilin_bound}
    \\
    \|B_n(\Delta,\eta)-B_n(\Delta,\eta')\|
    &\le
    2c_G^2\|\Delta\|_{n,G}\|\eta-\eta'\|_{n,G}, \label{eq:supp_bilin_diff_bound}
    \\
    \|\mathcal R_n^{\operatorname{quad}}(\eta)
    -\mathcal R_n^{\operatorname{quad}}(\eta')\|
    &\le
    c_G^2(\|\eta\|_{n,G}+\|\eta'\|_{n,G})
    \|\eta-\eta'\|_{n,G}. \label{eq:supp_quad_diff_bound}
\end{align}
Let
\[
    r_n:=F_n(\Delta),
    \qquad
    T_n(\eta):=
    -K_n\{r_n+B_n(\Delta,\eta)+\mathcal R_n^{\operatorname{quad}}(\eta)\}.
\]
Since \(L_n(K_nb)=b\), the fixed-point equation \(\eta=T_n(\eta)\) is equivalent to
\(F_n(\Delta+\eta)=0\). Fix \(R>0\), and define
\[
    \mathcal B_n(R)
    :=
    \{\eta\in(\mathbb R^p)^n:\|\eta\|_{n,G}\le Rn^{-1}\}.
\]
By Lemma~\ref{lem:supp_rwp_basic_objects} and Assumption~\ref{ass:main_rwp_nondeg}, there
exist \(c_0,C_0,C_1>0\) and an event \(\Omega_n\) with \(\Pr(\Omega_n)\to1\) such that on
\(\Omega_n\),
\[
    \lambda_{\min}(\widehat A_{n,G})\ge c_0,
    \qquad
    \|\Delta\|_{n,G}\le C_0n^{-1/2},
    \qquad
    \|r_n\|\le C_1n^{-1}.
\]
For \(\eta\in\mathcal B_n(R)\), using
\eqref{eq:supp_quad_bound} and \eqref{eq:supp_bilin_bound},
\[
\begin{aligned}
    \|T_n(\eta)\|_{n,G}
    &\le
    c_0^{-1/2}
    \{\|r_n\|+\|B_n(\Delta,\eta)\|
    +\|\mathcal R_n^{\operatorname{quad}}(\eta)\|\}  \\
    &\le
    c_0^{-1/2}
    \left(
    C_1n^{-1}
    +
    2c_G^2C_0R\,n^{-3/2}
    +
    c_G^2R^2n^{-2}
    \right).
\end{aligned}
\]
Choose \(R>2c_0^{-1/2}C_1\). Then, for all sufficiently large \(n\),
\[
    T_n\{\mathcal B_n(R)\}\subseteq\mathcal B_n(R)
    \qquad
    \text{on }\Omega_n.
\]
For \(\eta,\eta'\in\mathcal B_n(R)\), using
\eqref{eq:supp_bilin_diff_bound} and \eqref{eq:supp_quad_diff_bound},
\[
\begin{aligned}
    \|T_n(\eta)-T_n(\eta')\|_{n,G}
    &\le
    c_0^{-1/2}
    \bigl[
    2c_G^2\|\Delta\|_{n,G}
    +
    c_G^2\{\|\eta\|_{n,G}+\|\eta'\|_{n,G}\}
    \bigr]\|\eta-\eta'\|_{n,G}  \\
    &\le
    c_0^{-1/2}
    \bigl(
    2c_G^2C_0n^{-1/2}
    +
    2c_G^2Rn^{-1}
    \bigr)
    \|\eta-\eta'\|_{n,G}.
\end{aligned}
\]
The factor multiplying \(\|\eta-\eta'\|_{n,G}\) tends to zero. Hence, for all sufficiently
large \(n\), \(T_n\) is a contraction on \(\mathcal B_n(R)\) on \(\Omega_n\).

By Banach's fixed-point theorem, on \(\Omega_n\) there exists a unique
\(\eta_n\in\mathcal B_n(R)\) such that \(\eta_n=T_n(\eta_n)\). Set
\[
    \widetilde\Delta:=\Delta+\eta_n
    \qquad
    \text{on }\Omega_n.
\]
Then, on \(\Omega_n\),
\[
    F_n(\widetilde\Delta)=0,
    \qquad
    \|\widetilde\Delta-\Delta\|_{n,G}
    =
    \|\eta_n\|_{n,G}
    \le Rn^{-1}.
\]
Thus
\[
    \Pr\{F_n(\widetilde\Delta)=0\}\to1,
    \qquad
    \|\widetilde\Delta-\Delta\|_{n,G}=O_p(n^{-1}).
\]
Finally, on \(\Omega_n\),
\[
\begin{aligned}
    C_n^G(\widetilde\Delta)
    &=
    \|\Delta+\eta_n\|_{n,G}^2  \\
    &=
    \|\Delta\|_{n,G}^2
    +
    2\langle\Delta,\eta_n\rangle_{n,G}
    +
    \|\eta_n\|_{n,G}^2  \\
    &=
    C_n^G(\Delta)+O_p(n^{-1/2})O_p(n^{-1})+O_p(n^{-2}) \\
    &=
    C_n^G(\Delta)+o_p(n^{-1}).
\end{aligned}
\]
This proves the lemma.
\end{proof}

\begin{lemma}
\label{lem:supp_Rbar_upper}
Under Assumptions~\ref{ass:main_moment_cov} and \ref{ass:main_rwp_nondeg},
\[
    \bar R_n\le M_n+o_p(n^{-1}).
\]
\end{lemma}

\begin{proof}
By Lemma~\ref{lem:supp_linearized_minimizer},
\[
    C_n^G(\Delta^{(0)})=M_n=O_p(n^{-1}),
    \qquad
    \|\Delta^{(0)}\|_{n,G}=O_p(n^{-1/2}).
\]
Because \(\Delta^{(0)}\) satisfies the linearised constraint exactly,
\[
    \bar h_{n,\star}
    +
    \frac1n\sum_{i=1}^nJ_i\Delta_i^{(0)}
    =
    0.
\]
Hence, by the exact expansion from Lemma~\ref{lem:supp_h_derivative},
\[
    F_n(\Delta^{(0)})
    =
    \mathcal R_n^{\operatorname{quad}}(\Delta^{(0)}).
\]
Using \eqref{eq:supp_quad_bound}, we obtain
\[
    \|F_n(\Delta^{(0)})\|
    \le
    c_G^2\|\Delta^{(0)}\|_{n,G}^2
    =
    O_p(n^{-1}).
\]
Therefore \(\Delta^{(0)}\) satisfies the hypotheses of
Lemma~\ref{lem:supp_correction}. Applying that lemma with \(\Delta=\Delta^{(0)}\), we obtain
a corrected perturbation \(\widetilde\Delta\in(\mathbb R^p)^n\) such that
\[
    \Pr\{F_n(\widetilde\Delta)=0\}\to1
\]
and
\[
    C_n^G(\widetilde\Delta)
    =
    C_n^G(\Delta^{(0)})+o_p(n^{-1})
    =
    M_n+o_p(n^{-1}).
\]
Since \(\widetilde\Delta\) is feasible for the Dirac problem with probability tending to one,
\[
    \bar R_n
    \le
    C_n^G(\widetilde\Delta)
    =
    M_n+o_p(n^{-1}).
\]
This proves the claim.
\end{proof}

For \(b\in\mathbb R^{rs}\), define the shifted linearised value function
\[
    \phi_n(b)
    :=
    \inf_{\Delta\in(\mathbb R^p)^n}
    \left\{
    C_n^G(\Delta):
    \bar h_{n,\star}
    +
    \frac1n\sum_{i=1}^nJ_i\Delta_i=b
    \right\}.
\]

\begin{lemma}
\label{lem:supp_shifted_linearized_value}
If \(\widehat A_{n,G}\) is invertible, then
\[
    \phi_n(b)
    =
    (b-\bar h_{n,\star})^\top
    \widehat A_{n,G}^{-1}
    (b-\bar h_{n,\star}).
\]
In particular,
\[
    \phi_n(0)=M_n.
\]
\end{lemma}

\begin{proof}
This is the same quadratic-programming calculation as in
Lemma~\ref{lem:supp_linearized_minimizer}, with right-hand side \(b\) instead of \(0\). The
Lagrangian is
\[
    L(\Delta;\lambda)
    =
    \frac1n\sum_{i=1}^n\Delta_i^\top G\Delta_i
    +
    \lambda^\top
    \left(
    \bar h_{n,\star}
    +
    \frac1n\sum_{i=1}^nJ_i\Delta_i
    -
    b
    \right).
\]
Differentiating with respect to \(\Delta_i\) gives
\[
    2G\Delta_i+J_i^\top\lambda=0,
    \qquad
    \Delta_i=-\frac12G^{-1}J_i^\top\lambda.
\]
Substituting into the constraint yields
\[
    \bar h_{n,\star}-\frac12\widehat A_{n,G}\lambda=b,
    \qquad
    \lambda=2\widehat A_{n,G}^{-1}(\bar h_{n,\star}-b).
\]
Therefore the minimiser is
\[
    \Delta_i(b)
    =
    G^{-1}J_i^\top
    \widehat A_{n,G}^{-1}
    (b-\bar h_{n,\star}),
    \qquad
    1\le i\le n.
\]
Substituting this into the objective gives
\[
    \phi_n(b)
    =
    (b-\bar h_{n,\star})^\top
    \widehat A_{n,G}^{-1}
    (b-\bar h_{n,\star}).
\]
Taking \(b=0\) yields \(\phi_n(0)=M_n\).
\end{proof}

\begin{lemma}
\label{lem:supp_Rbar_lower}
Under Assumptions~\ref{ass:main_moment_cov} and \ref{ass:main_rwp_nondeg},
\[
    \bar R_n\ge M_n-o_p(n^{-1}).
\]
\end{lemma}

\begin{proof}
Let \(\varepsilon_n\downarrow0\) be deterministic with \(\varepsilon_n=o(n^{-1})\). By the
definition of \(\bar R_n\), we may choose \(\Delta_n\in(\mathbb R^p)^n\) such that
\[
    F_n(\Delta_n)=0,
    \qquad
    C_n^G(\Delta_n)\le \bar R_n+\varepsilon_n.
\]
By Lemma~\ref{lem:supp_Rbar_upper},
\[
    \bar R_n\le M_n+o_p(n^{-1})=O_p(n^{-1}),
\]
and therefore
\[
    C_n^G(\Delta_n)=O_p(n^{-1}).
\]
Define
\[
    b_n
    :=
    -\mathcal R_n^{\operatorname{quad}}(\Delta_n)
    =
    -\frac1n\sum_{i=1}^n
    \operatorname{vec}
    \left(
    U_{\star,\perp}^\top\Delta_{n,i}\Delta_{n,i}^\top U_\star
    \right).
\]
Since \(F_n(\Delta_n)=0\), the exact expansion
\[
    F_n(\Delta_n)
    =
    \bar h_{n,\star}
    +
    \frac1n\sum_{i=1}^nJ_i\Delta_{n,i}
    +
    \mathcal R_n^{\operatorname{quad}}(\Delta_n)
\]
implies
\[
    \bar h_{n,\star}
    +
    \frac1n\sum_{i=1}^nJ_i\Delta_{n,i}
    =
    b_n.
\]
Hence \(\Delta_n\) is feasible for the shifted linearised problem with right-hand side
\(b_n\), so
\[
    C_n^G(\Delta_n)\ge \phi_n(b_n).
\]
By Lemma~\ref{lem:supp_shifted_linearized_value},
\[
\begin{aligned}
    \phi_n(b_n)
    &=
    (b_n-\bar h_{n,\star})^\top
    \widehat A_{n,G}^{-1}
    (b_n-\bar h_{n,\star}) =
    M_n
    -
    2b_n^\top\widehat A_{n,G}^{-1}\bar h_{n,\star}
    +
    b_n^\top\widehat A_{n,G}^{-1}b_n .
\end{aligned}
\]
By \eqref{eq:supp_quad_bound},
\[
    \|b_n\|
    \le
    c_G^2\|\Delta_n\|_{n,G}^2
    =
    c_G^2C_n^G(\Delta_n)
    =
    O_p(n^{-1}).
\]
Also, Lemma~\ref{lem:supp_rwp_basic_objects} gives
\[
    \widehat A_{n,G}^{-1}=O_p(1),
    \qquad
    \|\bar h_{n,\star}\|=O_p(n^{-1/2}).
\]
Therefore
\[
    |b_n^\top\widehat A_{n,G}^{-1}\bar h_{n,\star}|
    =
    O_p(n^{-3/2}),
    \qquad
    b_n^\top\widehat A_{n,G}^{-1}b_n
    =
    O_p(n^{-2}).
\]
Thus
\[
    \phi_n(b_n)
    =
    M_n+o_p(n^{-1}).
\]
In particular, \(\phi_n(b_n)\ge M_n-o_p(n^{-1})\). Consequently,
\[
    C_n^G(\Delta_n)\ge M_n-o_p(n^{-1}).
\]
Since \(C_n^G(\Delta_n)\le\bar R_n+\varepsilon_n\) and
\(\varepsilon_n=o(n^{-1})\), we obtain
\[
    \bar R_n\ge M_n-o_p(n^{-1}).
\]
This proves the lower bound.
\end{proof}

Combining Lemmas~\ref{lem:supp_Rbar_upper} and \ref{lem:supp_Rbar_lower}, we obtain
\begin{equation}
\label{eq:supp_Rbar_expand}
    \bar R_n=M_n+o_p(n^{-1}).
\end{equation}

\medskip
\noindent\textbf{Asymptotic tightness of the Dirac upper bound.}

\begin{theorem}
\label{thm:supp_dirac_tight}
Under Assumptions~\ref{ass:main_moment_cov} and \ref{ass:main_rwp_nondeg},
\[
    0\le \bar R_n-R_n=o_p(n^{-1}).
\]
\end{theorem}

\begin{proof}
Since the class of Dirac measures is a subclass of the class of all admissible probability
measures, \(R_n\le\bar R_n\). Thus \(0\le\bar R_n-R_n\). It remains to prove
\[
    \bar R_n\le R_n+o_p(n^{-1}).
\]
By Lemma~\ref{lem:supp_Rbar_upper},
\[
    \bar R_n\le M_n+o_p(n^{-1}),
\]
and by Lemma~\ref{lem:supp_linearized_minimizer}, \(M_n=O_p(n^{-1})\). Hence
\(\bar R_n=O_p(n^{-1})\), and therefore \(R_n=O_p(n^{-1})\), since \(R_n\le\bar R_n\).\\
Fix a deterministic sequence \(\varepsilon_n\downarrow0\) such that
\[
    \varepsilon_n=o(n^{-1}).
\]
By the definition of the infimum, there exists a collection of probability measures
\(\{\mu_{n,i}\}_{i=1}^n\) on \(\mathbb R^p\) such that
\[
    \frac1n\sum_{i=1}^n\int h(z)\,\mu_{n,i}(dz)=0
\]
and
\begin{equation}
\label{eq:supp_eps_optimal_measures}
    \frac1n\sum_{i=1}^n
    \int \|z-X_i\|_G^2\,\mu_{n,i}(dz)
    \le
    R_n+\varepsilon_n.
\end{equation}
For each \(i\), define
\[
    m_{n,i}:=\int z\,\mu_{n,i}(dz),
    \qquad
    \Sigma_{n,i}:=
    \int (z-m_{n,i})(z-m_{n,i})^\top\,\mu_{n,i}(dz),
\]
and set
\[
    \Delta_{n,i}:=m_{n,i}-X_i,
    \qquad
    \Delta_n:=(\Delta_{n,1},\ldots,\Delta_{n,n}).
\]
By the bias--variance decomposition,
\[
\begin{aligned}
    \int\|z-X_i\|_G^2\,\mu_{n,i}(dz)
    &=
    (m_{n,i}-X_i)^\top G(m_{n,i}-X_i)
    +
    \operatorname{tr}(G\Sigma_{n,i})  \\
    &=
    \|\Delta_{n,i}\|_G^2+\operatorname{tr}(G\Sigma_{n,i}).
\end{aligned}
\]
Averaging over \(i\) gives
\begin{equation}
\label{eq:supp_cost_decomposition}
    C_n^G(\Delta_n)+V_n(\Sigma_n)\le R_n+\varepsilon_n,
\end{equation}
where
\[
    V_n(\Sigma_n):=\frac1n\sum_{i=1}^n\operatorname{tr}(G\Sigma_{n,i}).
\]
Since \(R_n=O_p(n^{-1})\) and \(\varepsilon_n=o(n^{-1})\), it follows that
\[
    C_n^G(\Delta_n)=O_p(n^{-1}),
    \qquad
    V_n(\Sigma_n)=O_p(n^{-1}).
\]
Next we use feasibility of the measures. Since
\[
    \frac1n\sum_{i=1}^n\int h(z)\,\mu_{n,i}(dz)=0,
\]
and
\[
    h(z)=\operatorname{vec}(U_{\star,\perp}^\top zz^\top U_\star),
\]
we have
\[
\begin{aligned}
    \int h(z)\,\mu_{n,i}(dz)
    &=
    \operatorname{vec}
    \left[
    U_{\star,\perp}^\top
    \{m_{n,i}m_{n,i}^\top+\Sigma_{n,i}\}
    U_\star
    \right]  \\
    &=
    h(X_i+\Delta_{n,i})
    +
    \operatorname{vec}
    (U_{\star,\perp}^\top\Sigma_{n,i}U_\star).
\end{aligned}
\]
Averaging over \(i\), and using feasibility, yields
\begin{equation}
\label{eq:supp_almost_feasible}
    F_n(\Delta_n)+S_n=0,
\end{equation}
where
\[
    S_n
    :=
    \frac1n\sum_{i=1}^n
    \operatorname{vec}
    (U_{\star,\perp}^\top\Sigma_{n,i}U_\star).
\]
For each \(i\),
\[
\begin{aligned}
    \left\|
    \operatorname{vec}
    (U_{\star,\perp}^\top\Sigma_{n,i}U_\star)
    \right\|
    &=
    \|U_{\star,\perp}^\top\Sigma_{n,i}U_\star\|_F  \\
    &=
    \|(\Sigma_{n,i}^{1/2}U_{\star,\perp})^\top
    (\Sigma_{n,i}^{1/2}U_\star)\|_F  \\
    &\le
    \frac12
    \{\|\Sigma_{n,i}^{1/2}U_{\star,\perp}\|_F^2
    +
    \|\Sigma_{n,i}^{1/2}U_\star\|_F^2\}  \\
    &\le
    \operatorname{tr}(\Sigma_{n,i}).
\end{aligned}
\]
Because \(G\) is positive definite, there exists \(c_G>0\) such that
\[
    \operatorname{tr}(\Sigma_{n,i})
    \le
    c_G\operatorname{tr}(G\Sigma_{n,i}).
\]
Therefore
\[
    \|S_n\|
    \le
    c_GV_n(\Sigma_n)
    =
    O_p(n^{-1}).
\]
Combining this with \eqref{eq:supp_almost_feasible}, we obtain
\[
    \|F_n(\Delta_n)\|=O_p(n^{-1}).
\]
We may now apply Lemma~\ref{lem:supp_correction}. Since
\[
    C_n^G(\Delta_n)=O_p(n^{-1}),
    \qquad
    \|F_n(\Delta_n)\|=O_p(n^{-1}),
\]
there exists a corrected perturbation \(\widetilde\Delta_n\in(\mathbb R^p)^n\), with
probability tending to one, such that
\[
    F_n(\widetilde\Delta_n)=0
\]
and
\[
    C_n^G(\widetilde\Delta_n)
    =
    C_n^G(\Delta_n)+o_p(n^{-1}).
\]
Set
\[
    \widetilde z_{n,i}:=X_i+\widetilde\Delta_{n,i},
    \qquad
    1\le i\le n.
\]
Then \(\{\widetilde z_{n,i}\}_{i=1}^n\) is feasible for the Dirac problem defining
\(\bar R_n\), so
\[
    \bar R_n
    \le
    C_n^G(\widetilde\Delta_n)
    =
    C_n^G(\Delta_n)+o_p(n^{-1}).
\]
Using \eqref{eq:supp_cost_decomposition},
\[
    C_n^G(\Delta_n)\le R_n+\varepsilon_n.
\]
Hence
\[
    \bar R_n
    \le
    R_n+\varepsilon_n+o_p(n^{-1})
    =
    R_n+o_p(n^{-1}).
\]
Together with \(R_n\le\bar R_n\), this proves
\[
    0\le \bar R_n-R_n=o_p(n^{-1}).
\]
\end{proof}

\begin{theorem}
\label{thm:supp_rwp_limit}
Let \(Z_h\sim N(0,\Sigma_h)\). Under Assumptions~\ref{ass:main_moment_cov} and
\ref{ass:main_rwp_nondeg},
\[
    nR_n^{(G)}
    \Rightarrow
    Z_h^\top A_G^{-1}Z_h.
\]
Equivalently, with the abbreviation \(R_n=R_n^{(G)}\),
\[
    nR_n\Rightarrow Z_h^\top A_G^{-1}Z_h.
\]
\end{theorem}

\begin{proof}
By \eqref{eq:supp_Rbar_expand} and Theorem~\ref{thm:supp_dirac_tight},
\[
    R_n
    =
    \bar R_n+o_p(n^{-1})
    =
    M_n+o_p(n^{-1}).
\]
Using
\[
    M_n
    =
    \bar h_{n,\star}^\top
    \widehat A_{n,G}^{-1}
    \bar h_{n,\star},
\]
we obtain
\[
    nR_n
    =
    (\sqrt n\,\bar h_{n,\star})^\top
    \widehat A_{n,G}^{-1}
    (\sqrt n\,\bar h_{n,\star})
    +
    o_p(1).
\]
By Lemma~\ref{lem:supp_rwp_basic_objects},
\[
    \sqrt n\,\bar h_{n,\star}\Rightarrow Z_h,
    \qquad
    \widehat A_{n,G}^{-1}\xrightarrow{p}A_G^{-1}.
\]
The result follows from the continuous mapping theorem.
\end{proof}

\subsubsection{Adaptive transport geometry}

\begin{lemma}
\label{lem:supp_randomG_metric_equiv}
Under Assumption~\ref{ass:main_random_G}, define
\[
    \Delta_{n,G}
    :=
    \bigl\|G^{-1/2}\widehat G_nG^{-1/2}-I_p\bigr\|_{\operatorname{op}}.
\]
Then
\[
    \Delta_{n,G}\xrightarrow{p}0.
\]
Moreover, on the event \(\{\Delta_{n,G}<1\}\), for every \(x,y\in\mathbb R^p\),
\[
    (1-\Delta_{n,G})\|x-y\|_G^2
    \le
    (x-y)^\top\widehat G_n(x-y)
    \le
    (1+\Delta_{n,G})\|x-y\|_G^2.
\]
Consequently, for any probability measures \(\mathbb P,\mathbb Q\),
\[
    (1-\Delta_{n,G})W_{2,G}^2(\mathbb P,\mathbb Q)
    \le
    W_{2,\widehat G_n}^2(\mathbb P,\mathbb Q)
    \le
    (1+\Delta_{n,G})W_{2,G}^2(\mathbb P,\mathbb Q),
\]
and therefore
\[
    (1-\Delta_{n,G})R_n^{(G)}
    \le
    R_n^{(\widehat G)}([U_\star,U_{\star,\perp}])
    \le
    (1+\Delta_{n,G})R_n^{(G)}.
\]
\end{lemma}

\begin{proof}
Set
\[
    B_n:=G^{-1/2}\widehat G_nG^{-1/2}.
\]
Then
\[
    B_n-I_p
    =
    G^{-1/2}(\widehat G_n-G)G^{-1/2},
\]
so Assumption~\ref{ass:main_random_G} implies
\[
    \|B_n-I_p\|_{\operatorname{op}}
    \le
    \|G^{-1/2}\|_{\operatorname{op}}^2
    \|\widehat G_n-G\|_{\operatorname{op}}
    \xrightarrow{p}0.
\]
This proves \(\Delta_{n,G}\xrightarrow{p}0\).

Let \(v:=x-y\) and \(u:=G^{1/2}v\). Then
\[
    v^\top\widehat G_nv
    =
    u^\top B_nu,
    \qquad
    \|v\|_G^2=u^\top u=\|u\|_2^2.
\]
Since \(\|B_n-I_p\|_{\operatorname{op}}=\Delta_{n,G}\),
\[
    |u^\top(B_n-I_p)u|
    \le
    \Delta_{n,G}\|u\|_2^2,
\]
which gives the pointwise cost comparison. Integrating the comparison with respect to any
coupling of \(\mathbb P\) and \(\mathbb Q\), and taking the infimum over couplings, gives
the Wasserstein comparison. Finally, the feasible set
\[
    \{\mathbb P:\mathbb E_{\mathbb P}\{h(Z)\}=0\}
\]
is the same in the definitions of \(R_n^{(G)}\) and
\(R_n^{(\widehat G)}([U_\star,U_{\star,\perp}])\). Taking the infimum over this common
feasible set gives the comparison of the profile statistics.
\end{proof}

\begin{lemma}
\label{lem:supp_adaptive_Ahat}
Under Assumptions~\ref{ass:main_moment_cov}, \ref{ass:main_random_G} and
\ref{ass:main_rwp_nondeg},
\[
    \widetilde A_n^{(\widehat G)}
    :=
    \frac1n\sum_{i=1}^nJ_i\widehat G_n^{-1}J_i^\top
    \xrightarrow{p}
    A_G,
    \qquad
    \{\widetilde A_n^{(\widehat G)}\}^{-1}=O_p(1).
\]
\end{lemma}

\begin{proof}
Write
\[
    \widetilde A_n^{(\widehat G)}-\widehat A_{n,G}
    =
    \frac1n\sum_{i=1}^n
    J_i(\widehat G_n^{-1}-G^{-1})J_i^\top .
\]
The identity
\[
    \widehat G_n^{-1}-G^{-1}
    =
    G^{-1}(G-\widehat G_n)\widehat G_n^{-1}
\]
and Assumption~\ref{ass:main_random_G} imply
\[
    \|\widehat G_n^{-1}-G^{-1}\|_{\operatorname{op}}
    \xrightarrow{p}0,
    \qquad
    \|\widehat G_n^{-1}\|_{\operatorname{op}}=O_p(1).
\]
Therefore
\[
    \|\widetilde A_n^{(\widehat G)}-\widehat A_{n,G}\|_{\operatorname{op}}
    \le
    \|\widehat G_n^{-1}-G^{-1}\|_{\operatorname{op}}
    \frac1n\sum_{i=1}^n\|J_i\|_F^2.
\]
By Lemma~\ref{lem:supp_h_derivative}, \(\|J_i\|_F\le C\|X_i\|_2\), and
Assumption~\ref{ass:main_moment_cov} gives
\[
    \frac1n\sum_{i=1}^n\|J_i\|_F^2=O_p(1).
\]
Hence
\[
    \widetilde A_n^{(\widehat G)}-\widehat A_{n,G}\xrightarrow{p}0.
\]
Since \(\widehat A_{n,G}\xrightarrow{p}A_G\) by
Lemma~\ref{lem:supp_rwp_basic_objects}, it follows that
\[
    \widetilde A_n^{(\widehat G)}\xrightarrow{p}A_G.
\]
Because \(A_G\) is positive definite, the inverse is \(O_p(1)\).
\end{proof}

\begin{proof}[Proof of Theorem~\ref{thm:main_rwp_limit}]
By Lemma~\ref{lem:supp_randomG_metric_equiv},
\[
    \left|
    R_n^{(\widehat G)}([U_\star,U_{\star,\perp}])
    -
    R_n^{(G)}
    \right|
    \le
    \Delta_{n,G}R_n^{(G)}
\]
on an event whose probability tends to one. Multiplying by \(n\),
\[
    n\left|
    R_n^{(\widehat G)}([U_\star,U_{\star,\perp}])
    -
    R_n^{(G)}
    \right|
    \le
    \Delta_{n,G}\,nR_n^{(G)}.
\]
By Lemma~\ref{lem:supp_randomG_metric_equiv}, \(\Delta_{n,G}\xrightarrow{p}0\), while
Theorem~\ref{thm:supp_rwp_limit} gives
\[
    nR_n^{(G)}
    \Rightarrow
    Z_h^\top A_G^{-1}Z_h.
\]
Thus \(nR_n^{(G)}=O_p(1)\), and hence
\[
    n\left|
    R_n^{(\widehat G)}([U_\star,U_{\star,\perp}])
    -
    R_n^{(G)}
    \right|
    =
    o_p(1).
\]
Therefore
\[
    nR_n^{(\widehat G)}([U_\star,U_{\star,\perp}])
    =
    nR_n^{(G)}+o_p(1)
    \Rightarrow
    Z_h^\top A_G^{-1}Z_h.
\]
If \(q_{1-\alpha}\) denotes the \((1-\alpha)\)-quantile of
\(Z_h^\top A_G^{-1}Z_h\), then
\[
    \delta_{n,\alpha}^{( G)}
    =
    \sqrt{\frac{q_{1-\alpha}}{n}}
\]
satisfies, at every continuity point \(q_{1-\alpha}\),
\[
\begin{aligned}
    \Pr\left\{
    \sqrt{R_n^{(\widehat G)}([U_\star,U_{\star,\perp}])}
    \le
    \delta_{n,\alpha}^{( G)}
    \right\}
    &=
    \Pr\left\{
    nR_n^{(\widehat G)}([U_\star,U_{\star,\perp}])
    \le
    q_{1-\alpha}
    \right\}  \\
    &\to
    1-\alpha.
\end{aligned}
\]
This proves the theorem.
\end{proof}

\subsection{Proof of the plug-in quantile consistency result}
\label{supp:sec_plugin_quantile}

\begin{proof}[Proof of Theorem~\ref{thm:main_plugin_quantile}]
By the assumption of Theorem~\ref{thm:main_plugin_quantile},
\[
    \|\widehat\Pi_n^{(0)}-\Pi_\star\|_F\xrightarrow{p}0,
    \qquad
    \widehat\Pi_n^{(0)}=\widehat U_n\widehat U_n^\top .
\]
We may choose orthonormal bases so that, after a random right orthogonal rotation if
necessary,
\[
    \widehat U_n\xrightarrow{p}U_\star,
    \qquad
    \widehat U_{n,\perp}\xrightarrow{p}U_{\star,\perp}.
\]
Define
\[
    h_i^\star:=h(X_i;[U_\star,U_{\star,\perp}]),
    \qquad
    \widehat h_i:=h(X_i;[\widehat U_n,\widehat U_{n,\perp}]).
\]
Let
\[
    a_n:=
    \|\widehat U_n-U_\star\|_F+
    \|\widehat U_{n,\perp}-U_{\star,\perp}\|_F .
\]
Then \(a_n\xrightarrow{p}0\). Moreover,
\[
\begin{aligned}
    \|\widehat h_i-h_i^\star\|
    &=
    \left\|
    \operatorname{vec}
    \left(
    \widehat U_{n,\perp}^\top X_iX_i^\top \widehat U_n
    -
    U_{\star,\perp}^\top X_iX_i^\top U_\star
    \right)
    \right\|   \\
    &\le
    a_n\|X_i\|_2^2 .
\end{aligned}
\]
Therefore
\[
    \frac1n\sum_{i=1}^n\|\widehat h_i-h_i^\star\|
    \le
    a_n\frac1n\sum_{i=1}^n\|X_i\|_2^2,
\]
and
\[
    \frac1n\sum_{i=1}^n\|\widehat h_i-h_i^\star\|^2
    \le
    a_n^2\frac1n\sum_{i=1}^n\|X_i\|_2^4.
\]
By Assumption~\ref{ass:main_moment_cov} and the law of large numbers,
\[
    \frac1n\sum_{i=1}^n\|X_i\|_2^2=O_p(1),
    \qquad
    \frac1n\sum_{i=1}^n\|X_i\|_2^4=O_p(1),
\]
so
\[
    \frac1n\sum_{i=1}^n\|\widehat h_i-h_i^\star\|
    \xrightarrow{p}0,
    \qquad
    \frac1n\sum_{i=1}^n\|\widehat h_i-h_i^\star\|^2
    \xrightarrow{p}0.
\]
Write
\[
    \bar h_n:=\frac1n\sum_{i=1}^n\widehat h_i,
    \qquad
    \bar h_{n,\star}:=\frac1n\sum_{i=1}^n h_i^\star.
\]
Then
\[
    \|\bar h_n-\bar h_{n,\star}\|
    \le
    \frac1n\sum_{i=1}^n\|\widehat h_i-h_i^\star\|
    \xrightarrow{p}0.
\]
Since \(\mathbb E\{h_i^\star\}=0\) and \(\mathbb E\|h_i^\star\|^2<\infty\),
\[
    \bar h_{n,\star}\xrightarrow{p}0,
\]
and hence \(\bar h_n\xrightarrow{p}0\).\\
Now
\[
    \widehat\Sigma_{h,n}
    =
    \frac1n\sum_{i=1}^n
    (\widehat h_i-\bar h_n)(\widehat h_i-\bar h_n)^\top
    =
    \frac1n\sum_{i=1}^n\widehat h_i\widehat h_i^\top-\bar h_n\bar h_n^\top .
\]
Similarly, define
\[
    \widehat\Sigma_{h,n}^\star
    :=
    \frac1n\sum_{i=1}^n h_i^\star h_i^{\star\top}
    -
    \bar h_{n,\star}\bar h_{n,\star}^\top .
\]
The preceding bounds imply
\[
    \widehat\Sigma_{h,n}-\widehat\Sigma_{h,n}^\star\xrightarrow{p}0.
\]
By the law of large numbers,
\[
    \frac1n\sum_{i=1}^n h_i^\star h_i^{\star\top}\xrightarrow{p}\Sigma_h,
    \qquad
    \bar h_{n,\star}\bar h_{n,\star}^\top\xrightarrow{p}0.
\]
Thus
\[
    \widehat\Sigma_{h,n}\xrightarrow{p}\Sigma_h.
\]
For each \(i\), define
\[
    J_i^\star:=D_xh(\cdot;[U_\star,U_{\star,\perp}])\big|_{x=X_i},
    \qquad
    \widehat J_i:=D_xh(\cdot;[\widehat U_n,\widehat U_{n,\perp}])\big|_{x=X_i}.
\]
Since
\[
    D_xh(\cdot;[U,U_\perp])\big|_x[\delta]
    =
    \operatorname{vec}
    \{U_\perp^\top(\delta x^\top+x\delta^\top)U\},
\]
the derivative map is continuous in \((U,U_\perp,x)\). Moreover, there exists \(C>0\) such
that, uniformly over orthonormal \((U,U_\perp)\),
\[
    \|D_xh(\cdot;[U,U_\perp])\big|_x\|_F
    \le
    C\|x\|_2,
\]
and
\[
    \|\widehat J_i-J_i^\star\|_F
    \le
    Ca_n\|X_i\|_2.
\]
Consequently,
\[
    \frac1n\sum_{i=1}^n\|\widehat J_i-J_i^\star\|_F^2
    \le
    C^2a_n^2\frac1n\sum_{i=1}^n\|X_i\|_2^2
    \xrightarrow{p}0,
\]
and
\[
    \frac1n\sum_{i=1}^n\|\widehat J_i\|_F^2=O_p(1),
    \qquad
    \frac1n\sum_{i=1}^n\|J_i^\star\|_F^2=O_p(1).
\]
Define the plug-in estimator
\[
    \widehat A_n^{(\widehat G)}
    :=
    \frac1n\sum_{i=1}^n
    \widehat J_i\widehat G_n^{-1}\widehat J_i^\top,
\]
and the oracle-basis adaptive version
\[
    \widetilde A_n^{(\widehat G)}
    :=
    \frac1n\sum_{i=1}^n
    J_i^\star\widehat G_n^{-1}J_i^{\star\top}.
\]
By Lemma~\ref{lem:supp_adaptive_Ahat},
\[
    \widetilde A_n^{(\widehat G)}\xrightarrow{p}A_G.
\]
Furthermore, Assumption~\ref{ass:main_random_G} gives
\(\|\widehat G_n^{-1}\|_{\operatorname{op}}=O_p(1)\). Hence
\[
\begin{aligned}
    \|\widehat A_n^{(\widehat G)}
    -
    \widetilde A_n^{(\widehat G)}\|_{\operatorname{op}}
    &\le
    \|\widehat G_n^{-1}\|_{\operatorname{op}}
    \frac1n\sum_{i=1}^n
    \|\widehat J_i-J_i^\star\|_F
    \{\|\widehat J_i\|_F+\|J_i^\star\|_F\}  \\
    &\le
    \|\widehat G_n^{-1}\|_{\operatorname{op}}
    \left(
    \frac1n\sum_{i=1}^n\|\widehat J_i-J_i^\star\|_F^2
    \right)^{1/2}
    \left(
    \frac1n\sum_{i=1}^n
    \{\|\widehat J_i\|_F+\|J_i^\star\|_F\}^2
    \right)^{1/2}  \\
    &=o_p(1).
\end{aligned}
\]
Therefore
\[
    \widehat A_n^{(\widehat G)}\xrightarrow{p}A_G.
\]
Since \(A_G\) is positive definite by Assumption~\ref{ass:main_rwp_nondeg},
\[
    \{\widehat A_n^{(\widehat G)}\}^{-1}
    \xrightarrow{p}
    A_G^{-1}.
\]
Now define
\[
    \widehat B_n^{(\widehat G)}
    :=
    \widehat\Sigma_{h,n}^{1/2}
    \{\widehat A_n^{(\widehat G)}\}^{-1}
    \widehat\Sigma_{h,n}^{1/2}.
\]
The matrix square-root map is continuous on the cone of symmetric positive semidefinite
matrices, so
\[
    \widehat\Sigma_{h,n}^{1/2}
    \xrightarrow{p}
    \Sigma_h^{1/2}.
\]
Therefore
\[
    \widehat B_n^{(\widehat G)}
    \xrightarrow{p}
    \Sigma_h^{1/2}A_G^{-1}\Sigma_h^{1/2}.
\]
Let \(\lambda_1,\ldots,\lambda_{rs}\) denote the eigenvalues of
\(\Sigma_h^{1/2}A_G^{-1}\Sigma_h^{1/2}\), and let
\(\widehat\lambda_{1,n},\ldots,\widehat\lambda_{rs,n}\) denote the eigenvalues of
\(\widehat B_n^{(\widehat G)}\). By continuity of eigenvalues on symmetric matrices,
\[
    \widehat\lambda_{j,n}\xrightarrow{p}\lambda_j,
    \qquad
    j=1,\ldots,rs.
\]
Let
\[
    Y_n:=\sum_{j=1}^{rs}\widehat\lambda_{j,n}\chi_{1,j}^2,
\]
where \(\chi_{1,1}^2,\ldots,\chi_{1,rs}^2\) are independent \(\chi_1^2\) random variables.
Then
\[
    Y_n\Rightarrow \sum_{j=1}^{rs}\lambda_j\chi_{1,j}^2.
\]
The limiting random variable has the same distribution as
\[
    Z_h^\top A_G^{-1}Z_h,
    \qquad
    Z_h\sim N(0,\Sigma_h).
\]
Hence, if \(q_{1-\alpha}\) is a continuity point of this limiting distribution, standard
quantile convergence gives
\[
    \widehat q_{1-\alpha}\xrightarrow{p}q_{1-\alpha}.
\]
Consequently,
\[
    \sqrt n\,\widehat\delta_{n,\alpha}
    =
    \sqrt{\widehat q_{1-\alpha}}
    \xrightarrow{p}
    \sqrt{q_{1-\alpha}}
    =:\tau_\alpha.
\]
Finally, by Theorem~\ref{thm:main_rwp_limit},
\[
    nR_n^{(\widehat G)}([U_\star,U_{\star,\perp}])
    \Rightarrow
    Z_h^\top A_G^{-1}Z_h.
\]
Since
\[
    \widehat\delta_{n,\alpha}
    =
    \sqrt{\frac{\widehat q_{1-\alpha}}{n}},
\]
we have
\[
\begin{aligned}
    \Pr\left\{
    \sqrt{R_n^{(\widehat G)}([U_\star,U_{\star,\perp}])}
    \le
    \widehat\delta_{n,\alpha}
    \right\}
    &=
    \Pr\left\{
    nR_n^{(\widehat G)}([U_\star,U_{\star,\perp}])
    \le
    \widehat q_{1-\alpha}
    \right\}  \\
    &\to
    1-\alpha.
\end{aligned}
\]
This proves the theorem.
\end{proof}
\section{Local asymptotic theory on the Grassmannian: proofs of Theorem \ref{thm:main_local_quadratic}-\ref{thm:main_projector_clt_adaptive}}
\label{supp:sec_local_theory}
The proofs in this section use a local coordinate parametrisation of the Grassmannian around
the population residual projector \(Q_\star\), following the standard perturbative
viewpoint underlying fixed-dimensional PCA asymptotics. The first group of lemmas
constructs the local chart \(Q(K)\) and records its first-order projector expansion. The
second group proves the two components of the local surrogate objective:
Lemma~\ref{lem:supp_nominal_expansion} gives the nominal PCA expansion, while
Lemma~\ref{lem:supp_adaptive_robust_expansion} gives the expansion of the adaptive
Wasserstein exposure term. Combining these two expansions proves
Theorem~\ref{thm:main_local_quadratic}. The consistency and root-\(n\) localisation result,
Theorem~\ref{thm:main_consistency_rootn}, then follows from the local quadratic expansion
with the strict population eigengap. Finally,
Lemma~\ref{lem:supp_Wn_adaptive} gives the limiting behaviour of the local random linear
term \(Z_n+\sqrt n\widehat\delta_{n,\alpha}g_{\widehat G_n}\), which yields the coordinate
CLT in Theorem~\ref{thm:main_coordinate_clt}. Lemma~\ref{lem:supp_Q_expand} then transfers
this coordinate CLT back to the projector scale, proving
Theorem~\ref{thm:main_projector_clt_adaptive}.

\subsection{Local coordinates near the population minimiser}
\label{supp:sec_local_coordinates}

Throughout this subsection, assume Assumption~\ref{ass:main_moment_cov}. Recall that
\[
    Q_\star=U_{\star,\perp}U_{\star,\perp}^\top,
    \qquad
    \Pi_\star=U_\star U_\star^\top,
    \qquad
    \Gamma=[U_\star,U_{\star,\perp}].
\]
For matrices \(A,B\) of the same size, define
\[
    \langle A,B\rangle_F:=\operatorname{tr}(AB^\top),
    \qquad
    \|A\|_F:=\sqrt{\langle A,A\rangle_F}.
\]
For \(K\in\mathbb R^{r\times s}\), define
\[
    W(K):=
    \begin{pmatrix}
    K\\ I_s
    \end{pmatrix}
    (I_s+K^\top K)^{-1/2}
    \in\mathbb R^{p\times s},
\]
and
\[
    Q(K):=\Gamma W(K)W(K)^\top\Gamma^\top .
\]
Then \(Q(K)\in\mathcal G_s\) and \(Q(0)=Q_\star\). Moreover, every projector in a
sufficiently small neighbourhood of \(Q_\star\) can be represented uniquely in this form.
Let
\[
    M(K):=(I_s+K^\top K)^{-1}.
\]
Then
\[
    Q(K)
    =
    \Gamma
    \begin{pmatrix}
    KM(K)K^\top & KM(K)\\
    M(K)K^\top & M(K)
    \end{pmatrix}
    \Gamma^\top .
\]

\begin{lemma}
\label{lem:supp_Q_expand}
For fixed \(\Xi\in\mathbb R^{r\times s}\) and \(t\to0\),
\[
    Q(t\Xi)
    =
    Q_\star
    +
    t\,DQ(0)[\Xi]
    +
    t^2R_2(\Xi)
    +
    O(t^3),
\]
uniformly for \(\|\Xi\|_F\le R\), where
\[
    DQ(0)[\Xi]
    =
    \Gamma
    \begin{pmatrix}
    0 & \Xi\\
    \Xi^\top & 0
    \end{pmatrix}
    \Gamma^\top,
\]
and
\[
    R_2(\Xi)
    =
    \Gamma
    \begin{pmatrix}
    \Xi\Xi^\top & 0\\
    0 & -\Xi^\top\Xi
    \end{pmatrix}
    \Gamma^\top .
\]
\end{lemma}

\begin{proof}
Since
\[
    M(t\Xi)
    =
    (I_s+t^2\Xi^\top\Xi)^{-1}
    =
    I_s-t^2\Xi^\top\Xi+O(t^4),
\]
uniformly for \(\|\Xi\|_F\le R\), we have
\[
    t\Xi M(t\Xi)(t\Xi)^\top
    =
    t^2\Xi\Xi^\top+O(t^4),
\]
\[
    t\Xi M(t\Xi)=t\Xi+O(t^3),
    \qquad
    M(t\Xi)(t\Xi)^\top=t\Xi^\top+O(t^3),
\]
and
\[
    M(t\Xi)=I_s-t^2\Xi^\top\Xi+O(t^4).
\]
Substituting these expansions into the block representation of \(Q(K)\) yields the claim.
\end{proof}

\subsection{Nominal local expansion}
\label{supp:sec_local_expansions}

Write
\[
    \Gamma^\top\widehat\Sigma_n\Gamma
    =
    \begin{pmatrix}
    \widehat\Sigma_{11,n} & \widehat\Sigma_{12,n}\\
    \widehat\Sigma_{21,n} & \widehat\Sigma_{22,n}
    \end{pmatrix}.
\]

\begin{lemma}
\label{lem:supp_Zn_clt}
Let
\[
    \Gamma^\top X_i=
    \begin{pmatrix}
    \xi_i\\
    \eta_i
    \end{pmatrix},
    \qquad
    \xi_i\in\mathbb R^r,\quad \eta_i\in\mathbb R^s.
\]
Under Assumption~\ref{ass:main_moment_cov},
\[
    \widehat\Sigma_{12,n}
    =
    \frac1n\sum_{i=1}^n\xi_i\eta_i^\top,
    \qquad
    \mathbb E(\xi_i\eta_i^\top)=0,
\]
and
\[
    \sqrt n\,\operatorname{vec}(\widehat\Sigma_{12,n})
    \Rightarrow
    N(0,\Omega_{12}),
    \qquad
    \Omega_{12}:=\operatorname{Var}\{\operatorname{vec}(\xi_1\eta_1^\top)\}.
\]
Hence
\[
    \operatorname{vec}(Z_n)\Rightarrow N(0,\Omega_Z),
    \qquad
    Z_n:=a_\star^{-1/2}\sqrt n\,\widehat\Sigma_{12,n},
    \qquad
    \Omega_Z:=a_\star^{-1}\Omega_{12}.
\]
\end{lemma}

\begin{proof}
Since
\[
    \mathbb E(\Gamma^\top X_iX_i^\top\Gamma)
    =
    \Gamma^\top\Sigma\Gamma
    =
    \begin{pmatrix}
    \Lambda_1 & 0\\
    0 & \Lambda_2
    \end{pmatrix},
\]
the off-diagonal block vanishes, so \(\mathbb E(\xi_i\eta_i^\top)=0\). The moment condition
\(\mathbb E\|X_i\|_2^4<\infty\) implies
\[
    \mathbb E\|\xi_i\eta_i^\top\|_F^2<\infty.
\]
The multivariate central limit theorem yields the claim.
\end{proof}

\begin{lemma}
\label{lem:supp_nominal_expansion}
Let
\[
    A_n(K):=\operatorname{tr}\{\widehat\Sigma_nQ(K)\}.
\]
Under Assumptions~\ref{ass:main_moment_cov} and \ref{ass:main_a_star}, uniformly for
\(\|\Xi\|_F\le R\),
\begin{equation}
\label{eq:supp_An_expand}
    A_n(n^{-1/2}\Xi)-A_n(0)
    =
    \frac{2}{\sqrt n}\langle\widehat\Sigma_{12,n},\Xi\rangle_F
    +
    \frac1n
    \langle
    \Xi,\widehat\Sigma_{11,n}\Xi-\Xi\widehat\Sigma_{22,n}
    \rangle_F
    +
    o_p(n^{-1}).
\end{equation}
Moreover,
\[
    A_n(0)=\operatorname{tr}(\widehat\Sigma_{22,n})
    \xrightarrow{p}
    a_\star:=\operatorname{tr}(\Lambda_2).
\]
Hence, uniformly for \(\|\Xi\|_F\le R\),
\begin{equation}
\label{eq:supp_sqrt_expand}
    n\left[
    \sqrt{A_n(n^{-1/2}\Xi)}-\sqrt{A_n(0)}
    \right]
    =
    \langle Z_n,\Xi\rangle_F
    +
    \frac12\langle\Xi,H_\star\Xi\rangle_F
    +
    o_p(1),
\end{equation}
where
\[
    H_\star\Xi
    :=
    a_\star^{-1/2}(\Lambda_1\Xi-\Xi\Lambda_2).
\]
\end{lemma}

\begin{proof}
Set \(t=n^{-1/2}\). By Lemma~\ref{lem:supp_Q_expand},
\[
    Q(t\Xi)
    =
    \Gamma
    \begin{pmatrix}
    t^2\Xi\Xi^\top+O(t^4) & t\Xi+O(t^3)\\
    t\Xi^\top+O(t^3) & I_s-t^2\Xi^\top\Xi+O(t^4)
    \end{pmatrix}
    \Gamma^\top ,
\]
uniformly for \(\|\Xi\|_F\le R\). Therefore
\[
\begin{aligned}
    A_n(t\Xi)
    &=
    \operatorname{tr}(\widehat\Sigma_{22,n})
    +
    2t\langle\widehat\Sigma_{12,n},\Xi\rangle_F 
    +
    t^2
    \left\{
    \operatorname{tr}(\widehat\Sigma_{11,n}\Xi\Xi^\top)
    -
    \operatorname{tr}(\widehat\Sigma_{22,n}\Xi^\top\Xi)
    \right\}
    +
    O_p(t^3),
\end{aligned}
\]
uniformly for \(\|\Xi\|_F\le R\), because
\(\|\widehat\Sigma_n\|_{\operatorname{op}}=O_p(1)\) under
Assumption~\ref{ass:main_moment_cov}. Since \(t^3=n^{-3/2}=o(n^{-1})\), this proves
\eqref{eq:supp_An_expand}. Also,
\[
    A_n(0)=\operatorname{tr}(\widehat\Sigma_{22,n})
    \xrightarrow{p}
    \operatorname{tr}(\Lambda_2)=a_\star
\]
by the law of large numbers.

Define
\[
    \mathfrak a_n(\Xi):=A_n(n^{-1/2}\Xi)-A_n(0).
\]
By \eqref{eq:supp_An_expand} and
\(\widehat\Sigma_{11,n},\widehat\Sigma_{22,n}=O_p(1)\),
\[
    \sup_{\|\Xi\|_F\le R}|\mathfrak a_n(\Xi)|=O_p(n^{-1}).
\]
Using
\[
    \sqrt{a+u}-\sqrt a
    =
    \frac{u}{2\sqrt a}
    -
    \frac{u^2}{2\sqrt a(\sqrt{a+u}+\sqrt a)^2},
\]
with \(a=A_n(0)\) and \(u=\mathfrak a_n(\Xi)\), we obtain
\[
    n\left[
    \sqrt{A_n(n^{-1/2}\Xi)}-\sqrt{A_n(0)}
    \right]
    =
    \frac{n\,\mathfrak a_n(\Xi)}{2\sqrt{A_n(0)}}+o_p(1)
\]
uniformly for \(\|\Xi\|_F\le R\), because \(A_n(0)\xrightarrow{p}a_\star>0\) and
\[
    n\sup_{\|\Xi\|_F\le R}|\mathfrak a_n(\Xi)|^2=o_p(1).
\]
Substituting \eqref{eq:supp_An_expand}, replacing \(A_n(0)\) by \(a_\star\), and using
\[
    \widehat\Sigma_{11,n}\xrightarrow{p}\Lambda_1,
    \qquad
    \widehat\Sigma_{22,n}\xrightarrow{p}\Lambda_2,
\]
yields \eqref{eq:supp_sqrt_expand}.
\end{proof}

\subsection{Adaptive transport geometry and adaptive radius}
\label{supp:sec_adaptive_local_expansions}

Throughout this subsection, write
\[
    \mathcal J_n^{\mathrm{ad}}(Q)
    :=
    \sqrt{\operatorname{tr}(\widehat\Sigma_nQ)}
    +
    \widehat\delta_{n,\alpha}
    \sqrt{\rho_{\widehat G_n}(Q)},
    \qquad
    Q\in\mathcal G_s,
\]
where
\[
    \rho_{\widehat G_n}(Q)
    :=
    \lambda_{\max}(\widehat G_n^{-1/2}Q\widehat G_n^{-1/2}).
\]

\begin{lemma}
\label{lem:supp_adaptive_delta_rate}
Under the assumptions of Theorem~\ref{thm:main_plugin_quantile},
\[
    \sqrt n\,\widehat\delta_{n,\alpha}
    \xrightarrow{p}
    \tau_\alpha,
    \qquad
    \tau_\alpha:=\sqrt{q_{1-\alpha}}.
\]
In particular,
\[
    \widehat\delta_{n,\alpha}=O_p(n^{-1/2}),
    \qquad
    \widehat\delta_{n,\alpha}\xrightarrow{p}0.
\]
\end{lemma}

\begin{proof}
By Theorem~\ref{thm:main_plugin_quantile},
\[
    \widehat q_{1-\alpha}\xrightarrow{p}q_{1-\alpha}.
\]
Since
\[
    \widehat\delta_{n,\alpha}
    =
    \sqrt{\frac{\widehat q_{1-\alpha}}{n}},
\]
the conclusion follows from the continuous mapping theorem.
\end{proof}

Define
\[
    A_{\star,n}^{\mathrm{ad}}
    :=
    \widehat G_n^{-1/2}Q_\star\widehat G_n^{-1/2},
    \qquad
    A_\star
    :=
    G^{-1/2}Q_\star G^{-1/2},
\]
and
\[
    L_n^{\mathrm{ad}}(K)
    :=
    \lambda_{\max}\{\widehat G_n^{-1/2}Q(K)\widehat G_n^{-1/2}\},
    \qquad
    \ell_{\star,n}^{\mathrm{ad}}
    :=
    L_n^{\mathrm{ad}}(0)
    =
    \lambda_{\max}(A_{\star,n}^{\mathrm{ad}}).
\]
By Assumption~\ref{ass:main_simple},
\[
    \ell_\star
    :=
    \lambda_{\max}(A_\star)
    =
    \rho_G(Q_\star)
\]
is simple.

Let \(v_\star\in\mathbb R^p\) be a unit eigenvector of \(A_\star\) associated with
\(\ell_\star\), and write
\[
    \Gamma^\top G^{-1/2}v_\star
    =
    \begin{pmatrix}
    \alpha_\star\\
    \beta_\star
    \end{pmatrix},
    \qquad
    \alpha_\star\in\mathbb R^r,\quad \beta_\star\in\mathbb R^s.
\]
Define
\[
    g_G
    :=
    \ell_\star^{-1/2}\alpha_\star\beta_\star^\top
    \in\mathbb R^{r\times s}.
\]
On the event that \(\ell_{\star,n}^{\mathrm{ad}}\) is simple, let \(v_{\star,n}\) be a unit
eigenvector of \(A_{\star,n}^{\mathrm{ad}}\) associated with
\(\ell_{\star,n}^{\mathrm{ad}}\), and write
\[
    \Gamma^\top\widehat G_n^{-1/2}v_{\star,n}
    =
    \begin{pmatrix}
    \alpha_{\star,n}\\
    \beta_{\star,n}
    \end{pmatrix},
    \qquad
    \alpha_{\star,n}\in\mathbb R^r,\quad \beta_{\star,n}\in\mathbb R^s.
\]
Define
\[
    g_{\widehat G_n}
    :=
    (\ell_{\star,n}^{\mathrm{ad}})^{-1/2}
    \alpha_{\star,n}\beta_{\star,n}^\top .
\]
Changing the sign of \(v_\star\) or \(v_{\star,n}\) changes both block components by
\(-1\), so \(g_G\) and \(g_{\widehat G_n}\) are well defined.

\begin{lemma}
\label{lem:supp_adaptive_spectral_objects}
Under Assumptions~\ref{ass:main_simple} and \ref{ass:main_random_G},
\[
    \ell_{\star,n}^{\mathrm{ad}}\xrightarrow{p}\ell_\star,
    \qquad
    g_{\widehat G_n}\xrightarrow{p}g_G.
\]
Moreover, with probability tending to one, \(\ell_{\star,n}^{\mathrm{ad}}\) is simple.
\end{lemma}

\begin{proof}
By Assumption~\ref{ass:main_random_G},
\[
    \|\widehat G_n-G\|_{\operatorname{op}}\xrightarrow{p}0,
\]
and, with probability tending to one, all eigenvalues of \(\widehat G_n\) lie in a fixed
compact interval \([c,C]\subset(0,\infty)\). Since \(x\mapsto x^{-1/2}\) is continuous on
\([c,C]\), standard matrix functional calculus gives
\[
    \|\widehat G_n^{-1/2}-G^{-1/2}\|_{\operatorname{op}}\xrightarrow{p}0.
\]
Hence
\[
\begin{aligned}
    \|A_{\star,n}^{\mathrm{ad}}-A_\star\|_{\operatorname{op}}
    &\le
    \|\widehat G_n^{-1/2}-G^{-1/2}\|_{\operatorname{op}}
    \|Q_\star\|_{\operatorname{op}}
    \|\widehat G_n^{-1/2}\|_{\operatorname{op}}   \\
    &\quad
    +
    \|G^{-1/2}\|_{\operatorname{op}}
    \|Q_\star\|_{\operatorname{op}}
    \|\widehat G_n^{-1/2}-G^{-1/2}\|_{\operatorname{op}}
    \xrightarrow{p}0.
\end{aligned}
\]
By Weyl's theorem,
\[
    \ell_{\star,n}^{\mathrm{ad}}
    =
    \lambda_{\max}(A_{\star,n}^{\mathrm{ad}})
    \xrightarrow{p}
    \lambda_{\max}(A_\star)
    =
    \ell_\star.
\]
Since \(Q_\star\neq0\) and \(G\) is positive definite, \(A_\star\) is a nonzero symmetric
positive semidefinite matrix, and therefore \(\ell_\star>0\).

Let
\[
    \gamma_\star
    :=
    \lambda_{\max}(A_\star)-\lambda_2(A_\star)>0,
\]
where \(\lambda_2(A_\star)\) is the second-largest eigenvalue of \(A_\star\). Weyl's theorem
gives
\[
    \lambda_{\max}(A_{\star,n}^{\mathrm{ad}})
    -
    \lambda_2(A_{\star,n}^{\mathrm{ad}})
    \ge
    \gamma_\star
    -
    2\|A_{\star,n}^{\mathrm{ad}}-A_\star\|_{\operatorname{op}}.
\]
Therefore
\[
    \Pr\left\{
    \lambda_{\max}(A_{\star,n}^{\mathrm{ad}})
    -
    \lambda_2(A_{\star,n}^{\mathrm{ad}})
    >
    \frac{\gamma_\star}{2}
    \right\}
    \to1,
\]
so \(\ell_{\star,n}^{\mathrm{ad}}\) is simple with probability tending to one.

On this event, choose the sign of \(v_{\star,n}\) so that
\[
    \langle v_{\star,n},v_\star\rangle\ge0.
\]
By the Davis--Kahan theorem and the positive eigengap \(\gamma_\star\),
\[
    \|v_{\star,n}-v_\star\|_2\xrightarrow{p}0.
\]
Combining this with
\[
    \|\widehat G_n^{-1/2}-G^{-1/2}\|_{\operatorname{op}}\xrightarrow{p}0
\]
gives
\[
    \Gamma^\top\widehat G_n^{-1/2}v_{\star,n}
    \xrightarrow{p}
    \Gamma^\top G^{-1/2}v_\star.
\]
Hence
\[
    \alpha_{\star,n}\xrightarrow{p}\alpha_\star,
    \qquad
    \beta_{\star,n}\xrightarrow{p}\beta_\star.
\]
Finally, since \(\ell_{\star,n}^{\mathrm{ad}}\xrightarrow{p}\ell_\star>0\),
\[
    g_{\widehat G_n}
    =
    (\ell_{\star,n}^{\mathrm{ad}})^{-1/2}
    \alpha_{\star,n}\beta_{\star,n}^\top
    \xrightarrow{p}
    \ell_\star^{-1/2}\alpha_\star\beta_\star^\top
    =
    g_G.
\]
\end{proof}

We now expand the adaptive spectral penalty
\[
    K\mapsto
    \sqrt{
    \lambda_{\max}(\widehat G_n^{-1/2}Q(K)\widehat G_n^{-1/2})
    }
\]
along local perturbations of size \(n^{-1/2}\).

\begin{lemma}
\label{lem:supp_adaptive_robust_expansion}
Under Assumptions~\ref{ass:main_simple} and \ref{ass:main_random_G}, for every fixed
\(R>0\),
\[
    L_n^{\mathrm{ad}}(n^{-1/2}\Xi)
    =
    \ell_{\star,n}^{\mathrm{ad}}
    +
    \frac{2}{\sqrt n}
    \langle\alpha_{\star,n}\beta_{\star,n}^\top,\Xi\rangle_F
    +
    o_p(n^{-1/2}),
\]
uniformly for \(\|\Xi\|_F\le R\). Consequently,
\begin{equation}
\label{eq:supp_adaptive_robust_expand}
    n\widehat\delta_{n,\alpha}
    \left[
    \sqrt{L_n^{\mathrm{ad}}(n^{-1/2}\Xi)}
    -
    \sqrt{\ell_{\star,n}^{\mathrm{ad}}}
    \right]
    =
    \sqrt n\,\widehat\delta_{n,\alpha}
    \langle g_{\widehat G_n},\Xi\rangle_F
    +
    o_p(1),
\end{equation}
uniformly for \(\|\Xi\|_F\le R\).
\end{lemma}

\begin{proof}
By Lemma~\ref{lem:supp_Q_expand}, with \(t=n^{-1/2}\),
\[
    Q(n^{-1/2}\Xi)
    =
    Q_\star
    +
    n^{-1/2}DQ(0)[\Xi]
    +
    n^{-1}R_2(\Xi)
    +
    O(n^{-3/2}),
\]
uniformly for \(\|\Xi\|_F\le R\). Hence
\[
    \widehat G_n^{-1/2}Q(n^{-1/2}\Xi)\widehat G_n^{-1/2}
    =
    A_{\star,n}^{\mathrm{ad}}+E_n(\Xi),
\]
where
\[
\begin{aligned}
    E_n(\Xi)
    &=
    n^{-1/2}
    \widehat G_n^{-1/2}DQ(0)[\Xi]\widehat G_n^{-1/2}
    +
    n^{-1}
    \widehat G_n^{-1/2}R_2(\Xi)\widehat G_n^{-1/2}
    +
    O_p(n^{-3/2}),
\end{aligned}
\]
uniformly for \(\|\Xi\|_F\le R\). Since
\(\|\widehat G_n^{-1/2}\|_{\operatorname{op}}=O_p(1)\) under
Assumption~\ref{ass:main_random_G}, and \(DQ(0)[\Xi]\) and \(R_2(\Xi)\) are respectively
linear and quadratic in \(\Xi\),
\begin{equation}
\label{eq:supp_En_bound}
    \sup_{\|\Xi\|_F\le R}\|E_n(\Xi)\|_{\operatorname{op}}
    =
    O_p(n^{-1/2}).
\end{equation}

Let
\[
    \gamma_\star
    :=
    \lambda_{\max}(A_\star)-\lambda_2(A_\star)>0.
\]
By Lemma~\ref{lem:supp_adaptive_spectral_objects}, with probability tending to one,
\(A_{\star,n}^{\mathrm{ad}}\) has a simple largest eigenvalue and its spectral gap is at
least \(\gamma_\star/2\). On this event, the map
\(A\mapsto\lambda_{\max}(A)\) is \(C^1\) on a neighbourhood of
\(A_{\star,n}^{\mathrm{ad}}\). Since \eqref{eq:supp_En_bound} implies
\[
    \sup_{\|\Xi\|_F\le R}\|E_n(\Xi)\|_{\operatorname{op}}=o_p(1),
\]
the first-order perturbation expansion holds uniformly:
\[
    \lambda_{\max}(A_{\star,n}^{\mathrm{ad}}+E_n(\Xi))
    =
    \lambda_{\max}(A_{\star,n}^{\mathrm{ad}})
    +
    v_{\star,n}^\top E_n(\Xi)v_{\star,n}
    +
    O_p(\|E_n(\Xi)\|_{\operatorname{op}}^2),
\]
uniformly for \(\|\Xi\|_F\le R\). Therefore
\[
    L_n^{\mathrm{ad}}(n^{-1/2}\Xi)
    =
    \ell_{\star,n}^{\mathrm{ad}}
    +
    v_{\star,n}^\top E_n(\Xi)v_{\star,n}
    +
    o_p(n^{-1/2}),
\]
uniformly for \(\|\Xi\|_F\le R\), because
\[
    \sup_{\|\Xi\|_F\le R}
    \|E_n(\Xi)\|_{\operatorname{op}}^2
    =
    O_p(n^{-1})
    =
    o_p(n^{-1/2}).
\]
It remains to compute the linear term. By the definition of \(E_n(\Xi)\),
\[
    v_{\star,n}^\top E_n(\Xi)v_{\star,n}
    =
    \frac1{\sqrt n}
    v_{\star,n}^\top
    \widehat G_n^{-1/2}DQ(0)[\Xi]\widehat G_n^{-1/2}
    v_{\star,n}
    +
    O_p(n^{-1}),
\]
uniformly for \(\|\Xi\|_F\le R\). Since
\[
    DQ(0)[\Xi]
    =
    \Gamma
    \begin{pmatrix}
    0 & \Xi\\
    \Xi^\top & 0
    \end{pmatrix}
    \Gamma^\top
\]
and
\[
    \Gamma^\top\widehat G_n^{-1/2}v_{\star,n}
    =
    \begin{pmatrix}
    \alpha_{\star,n}\\
    \beta_{\star,n}
    \end{pmatrix},
\]
we obtain
\[
\begin{aligned}
    v_{\star,n}^\top
    \widehat G_n^{-1/2}DQ(0)[\Xi]\widehat G_n^{-1/2}
    v_{\star,n}
    &=
    \begin{pmatrix}
    \alpha_{\star,n}\\
    \beta_{\star,n}
    \end{pmatrix}^{\!\top}
    \begin{pmatrix}
    0 & \Xi\\
    \Xi^\top & 0
    \end{pmatrix}
    \begin{pmatrix}
    \alpha_{\star,n}\\
    \beta_{\star,n}
    \end{pmatrix}   \\
    &=
    2\alpha_{\star,n}^\top\Xi\beta_{\star,n}  \\
    &=
    2\langle\alpha_{\star,n}\beta_{\star,n}^\top,\Xi\rangle_F .
\end{aligned}
\]
This proves the first display.

For the square-root expansion, Lemma~\ref{lem:supp_adaptive_spectral_objects} gives
\[
    \ell_{\star,n}^{\mathrm{ad}}\xrightarrow{p}\ell_\star>0.
\]
Moreover, by the first part,
\[
    \sup_{\|\Xi\|_F\le R}
    |L_n^{\mathrm{ad}}(n^{-1/2}\Xi)-\ell_{\star,n}^{\mathrm{ad}}|
    =
    O_p(n^{-1/2}).
\]
Thus, uniformly for \(\|\Xi\|_F\le R\), 
\[
    \sqrt{L_n^{\mathrm{ad}}(n^{-1/2}\Xi)}
    -
    \sqrt{\ell_{\star,n}^{\mathrm{ad}}}
    =
    \frac{
    L_n^{\mathrm{ad}}(n^{-1/2}\Xi)-\ell_{\star,n}^{\mathrm{ad}}
    }{
    2(\ell_{\star,n}^{\mathrm{ad}})^{1/2}
    }
    +
    O_p(n^{-1}),
\]
because the derivative of \(x\mapsto x^{1/2}\) is Lipschitz on a neighbourhood of
\(\ell_\star>0\). Multiplying by \(n\widehat\delta_{n,\alpha}\), using
\[
    \widehat\delta_{n,\alpha}=O_p(n^{-1/2})
\]
from Lemma~\ref{lem:supp_adaptive_delta_rate}, and recalling that
\[
    g_{\widehat G_n}
    =
    (\ell_{\star,n}^{\mathrm{ad}})^{-1/2}
    \alpha_{\star,n}\beta_{\star,n}^\top,
\]
we obtain \eqref{eq:supp_adaptive_robust_expand}.
\end{proof}

\begin{proof}[Proof of Theorem~\ref{thm:main_local_quadratic}]
For \(\Xi\in\mathbb R^{r\times s}\), define
\[
    \widetilde{\mathcal J}_n^{\mathrm{ad}}(\Xi)
    :=
    n\left[
    \mathcal J_n^{\mathrm{ad}}\{Q(n^{-1/2}\Xi)\}
    -
    \mathcal J_n^{\mathrm{ad}}(Q_\star)
    \right].
\]
Combining Lemma~\ref{lem:supp_nominal_expansion} with
Lemma~\ref{lem:supp_adaptive_robust_expansion} gives, uniformly for \(\|\Xi\|_F\le R\),
\[
    \widetilde{\mathcal J}_n^{\mathrm{ad}}(\Xi)
    =
    \langle
    Z_n+\sqrt n\,\widehat\delta_{n,\alpha}g_{\widehat G_n},
    \Xi
    \rangle_F
    +
    \frac12\langle\Xi,H_\star\Xi\rangle_F
    +
    o_p(1).
\]
Moreover,
\[
    \langle\Xi,H_\star\Xi\rangle_F
    =
    a_\star^{-1/2}
    \sum_{i=1}^r\sum_{j=1}^s
    (\theta_i-\theta_{r+j})\Xi_{ij}^2 .
\]
The eigengap in Assumption~\ref{ass:main_moment_cov} implies that \(H_\star\) is positive
definite.
\end{proof}

\subsection{Consistency and root-\(n\) localisation for the adaptive criterion}
\label{supp:sec_consistency_rootn}

\begin{lemma}[Argmin consistency on a compact parameter space]
\label{lem:supp_argmin}
Let \(\Theta\) be a compact metric space, let \(M_n:\Theta\to\mathbb R\) be random
functions, and let \(M:\Theta\to\mathbb R\) be deterministic. Assume:
\begin{enumerate}
\item \(M\) is continuous on \(\Theta\);
\item \(M\) has a unique minimiser \(\theta_0\in\Theta\);
\item
\[
    \sup_{\theta\in\Theta}|M_n(\theta)-M(\theta)|\xrightarrow{p}0.
\]
\end{enumerate}
If \(\widehat\theta_n\in\arg\min_{\theta\in\Theta}M_n(\theta)\), then
\[
    \widehat\theta_n\xrightarrow{p}\theta_0.
\]
\end{lemma}

\begin{proof}
Fix \(\varepsilon>0\), and define
\[
    F_\varepsilon:=\{\theta\in\Theta:d(\theta,\theta_0)\ge\varepsilon\}.
\]
Since \(\Theta\) is compact, \(F_\varepsilon\) is compact. Because \(M\) is continuous and
\(\theta_0\) is the unique minimiser,
\[
    \inf_{\theta\in F_\varepsilon}M(\theta)>M(\theta_0).
\]
Choose \(\eta_\varepsilon>0\) such that
\[
    \inf_{\theta\in F_\varepsilon}M(\theta)
    \ge
    M(\theta_0)+2\eta_\varepsilon.
\]
Set
\[
    A_{n,\varepsilon}
    :=
    \left\{
    \sup_{\theta\in\Theta}|M_n(\theta)-M(\theta)|
    \le
    \eta_\varepsilon
    \right\}.
\]
Then \(\Pr(A_{n,\varepsilon})\to1\). On \(A_{n,\varepsilon}\),
\[
    M_n(\theta_0)\le M(\theta_0)+\eta_\varepsilon,
\]
whereas for \(\theta\in F_\varepsilon\),
\[
    M_n(\theta)
    \ge
    M(\theta)-\eta_\varepsilon
    \ge
    M(\theta_0)+\eta_\varepsilon
    \ge
    M_n(\theta_0).
\]
Thus no minimiser can lie in \(F_\varepsilon\) on \(A_{n,\varepsilon}\), and therefore
\[
    \Pr\{d(\widehat\theta_n,\theta_0)\ge\varepsilon\}
    \le
    \Pr(A_{n,\varepsilon}^c)
    \to0.
\]
\end{proof}

\begin{lemma}
\label{lem:supp_Wn_adaptive}
Under Assumptions~\ref{ass:main_moment_cov}, \ref{ass:main_a_star},
\ref{ass:main_simple}, \ref{ass:main_rwp_nondeg}, and \ref{ass:main_random_G}, let
\[
    W_n^{\mathrm{ad}}
    :=
    Z_n+\sqrt n\,\widehat\delta_{n,\alpha}g_{\widehat G_n}.
\]
Then
\[
    W_n^{\mathrm{ad}}
    \Rightarrow
    Z+\tau_\alpha g_G
    \qquad
    \text{in }\mathbb R^{r\times s},
\]
where \(\operatorname{vec}(Z)\sim N(0,\Omega_Z)\). In particular,
\[
    W_n^{\mathrm{ad}}=O_p(1).
\]
\end{lemma}

\begin{proof}
By Lemma~\ref{lem:supp_Zn_clt},
\[
    Z_n\Rightarrow Z.
\]
By Lemma~\ref{lem:supp_adaptive_delta_rate},
\[
    \sqrt n\,\widehat\delta_{n,\alpha}\xrightarrow{p}\tau_\alpha.
\]
By Lemma~\ref{lem:supp_adaptive_spectral_objects},
\[
    g_{\widehat G_n}\xrightarrow{p}g_G.
\]
Therefore
\[
    \sqrt n\,\widehat\delta_{n,\alpha}g_{\widehat G_n}
    \xrightarrow{p}
    \tau_\alpha g_G.
\]
Slutsky's theorem yields the result.
\end{proof}

\begin{proof}[Proof of Theorem~\ref{thm:main_consistency_rootn}]
Define
\[
    \mathcal J_0(Q):=\sqrt{\operatorname{tr}(\Sigma Q)},
    \qquad
    Q\in\mathcal G_s.
\]
Since \(\mathcal G_s\) is compact and \(Q\mapsto\mathcal J_n^{\mathrm{ad}}(Q)\) is
continuous, a minimiser \(\widehat Q_n^{\mathrm{ad}}\) exists.
For the nominal term,
\[
    \sup_{Q\in\mathcal G_s}
    \left|
    \operatorname{tr}\{(\widehat\Sigma_n-\Sigma)Q\}
    \right|
    \le
    s\|\widehat\Sigma_n-\Sigma\|_{\operatorname{op}}
    \xrightarrow{p}0.
\]
Since \(a,b\ge0\) implies
\[
    |\sqrt a-\sqrt b|\le \sqrt{|a-b|},
\]
it follows that
\[
    \sup_{Q\in\mathcal G_s}
    \left|
    \sqrt{\operatorname{tr}(\widehat\Sigma_nQ)}
    -
    \sqrt{\operatorname{tr}(\Sigma Q)}
    \right|
    \xrightarrow{p}0.
\]
For the adaptive penalty,
\[
    \sup_{Q\in\mathcal G_s}
    \widehat\delta_{n,\alpha}\sqrt{\lambda_{\max}(\hat G_n^{-1/2}Q\hat G_n^{-1/2})}
    \le
    \widehat\delta_{n,\alpha}\|\widehat G_n^{-1/2}\|_{\operatorname{op}}
    \xrightarrow{p}0,
\]
because \(\widehat\delta_{n,\alpha}=O_p(n^{-1/2})\) and
\(\|\widehat G_n^{-1/2}\|_{\operatorname{op}}=O_p(1)\) under
Assumption~\ref{ass:main_random_G}. Therefore
\[
    \sup_{Q\in\mathcal G_s}
    |\mathcal J_n^{\mathrm{ad}}(Q)-\mathcal J_0(Q)|
    \xrightarrow{p}0.
\]
Since \(x\mapsto\sqrt x\) is strictly increasing on \([0,\infty)\), minimising
\(\mathcal J_0(Q)\) is equivalent to minimising \(\operatorname{tr}(\Sigma Q)\), whose
unique minimiser is \(Q_\star\) by the eigengap in Assumption~\ref{ass:main_moment_cov}.
Lemma~\ref{lem:supp_argmin} gives
\[
    \widehat Q_n^{\mathrm{ad}}\xrightarrow{p}Q_\star .
\]

Because \(Q(\cdot)\) is a local chart at \(0\), there exist \(\rho>0\) and an open
neighbourhood \(\mathcal U_\rho\subset\mathcal G_s\) of \(Q_\star\) such that \(Q(\cdot)\)
is one-to-one on
\[
    B_\rho:=\{K\in\mathbb R^{r\times s}:\|K\|_F<\rho\},
    \qquad
    Q(B_\rho)=\mathcal U_\rho.
\]
Since \(\widehat Q_n^{\mathrm{ad}}\xrightarrow{p}Q_\star\),
\[
    \Pr\{\widehat Q_n^{\mathrm{ad}}\in\mathcal U_\rho\}\to1.
\]
On this event, there is a unique
\(\widehat K_n^{\mathrm{ad}}\in B_\rho\) such that
\[
    \widehat Q_n^{\mathrm{ad}}=Q(\widehat K_n^{\mathrm{ad}}).
\]
It remains to prove that \(\widehat K_n^{\mathrm{ad}}=O_p(n^{-1/2})\). Define
\[
    D_n(K)
    :=
    \mathcal J_n^{\mathrm{ad}}\{Q(K)\}
    -
    \mathcal J_n^{\mathrm{ad}}(Q_\star),
    \qquad
    K\in B_\rho.
\]
On the event \(\{\widehat Q_n^{\mathrm{ad}}\in\mathcal U_\rho\}\), the point
\(\widehat K_n^{\mathrm{ad}}\) minimises \(K\mapsto D_n(K)\) over \(B_\rho\), and
\[
    D_n(\widehat K_n^{\mathrm{ad}})\le D_n(0)=0.
\]
Let
\[
    \kappa_\star
    :=
    \min_{1\le i\le r,\ 1\le j\le s}
    (\theta_i-\theta_{r+j})>0.
\]
From the block representation of \(Q(K)\), the same calculation as in
Lemma~\ref{lem:supp_nominal_expansion} gives
\[
    A_n(K)-A_n(0)
    =
    2\langle\widehat\Sigma_{12,n},K\rangle_F
    +
    \langle K,\widehat\Sigma_{11,n}K-K\widehat\Sigma_{22,n}\rangle_F
    +
    r_{1n}(K),
\]
where \(A_n(K):=\operatorname{tr}\{\widehat\Sigma_nQ(K)\}\), and for every fixed
\(\rho_0\in(0,\rho)\),
\[
    \sup_{\|K\|_F\le\rho_0}
    \frac{|r_{1n}(K)|}{\|K\|_F^3}
    =
    O_p(1).
\]
Similarly, the spectral perturbation argument used in
Lemma~\ref{lem:supp_adaptive_robust_expansion} gives
\[
    L_n^{\mathrm{ad}}(K)
    =
    \ell_{\star,n}^{\mathrm{ad}}
    +
    2\langle\alpha_{\star,n}\beta_{\star,n}^\top,K\rangle_F
    +
    r_{2n}(K),
\]
where
\[
    L_n^{\mathrm{ad}}(K)
    =
    \lambda_{\max}\{\widehat G_n^{-1/2}Q(K)\widehat G_n^{-1/2}\},
\]
and
\[
    \sup_{\|K\|_F\le\rho_0}
    \frac{|r_{2n}(K)|}{\|K\|_F^2}
    =
    O_p(1).
\]
Fix \(\varepsilon>0\). Choose \(M>0\) so large that, for all sufficiently large \(n\), the
event
\[
\begin{gathered}
    \|\widehat\Sigma_{11,n}-\Lambda_1\|_{\operatorname{op}}\le\frac{\kappa_\star}{4},
    \qquad
    \|\widehat\Sigma_{22,n}-\Lambda_2\|_{\operatorname{op}}\le\frac{\kappa_\star}{4},\\
    A_n(0)\in\left[\frac{a_\star}{2},\frac{3a_\star}{2}\right],
    \qquad
    \ell_{\star,n}^{\mathrm{ad}}\in
    \left[\frac{\ell_\star}{2},\frac{3\ell_\star}{2}\right],\\
    \sup_{\|K\|_F\le\rho_0}
    \frac{|r_{1n}(K)|}{\|K\|_F^3}\le M,
    \qquad
    \sup_{\|K\|_F\le\rho_0}
    \frac{|r_{2n}(K)|}{\|K\|_F^2}\le M,
    \qquad
    \|\alpha_{\star,n}\beta_{\star,n}^\top\|_F\le M
\end{gathered}
\]
has probability at least \(1-\varepsilon\). Choose \(\rho_0\in(0,\rho)\) so small that
\[
    M\rho_0\le\frac{\kappa_\star}{4}.
\]
On this event, for all \(\|K\|_F\le\rho_0\),
\[
    \langle K,\widehat\Sigma_{11,n}K-K\widehat\Sigma_{22,n}\rangle_F
    \ge
    \frac{\kappa_\star}{2}\|K\|_F^2,
\]
and therefore
\[
    A_n(K)-A_n(0)
    \ge
    -2\|\widehat\Sigma_{12,n}\|_F\|K\|_F
    +
    \frac{\kappa_\star}{4}\|K\|_F^2.
\]
Since \(A_n(0)\to a_\star>0\), after shrinking \(\rho_0\) if necessary, a Taylor expansion
of \(x\mapsto\sqrt x\) on a compact neighbourhood of \(a_\star\) gives constants
\(c_1,C_1>0\) such that
\begin{equation}
\label{eq:supp_nominal_lower_bound_K}
    \sqrt{A_n(K)}-\sqrt{A_n(0)}
    \ge
    -C_1\|\widehat\Sigma_{12,n}\|_F\|K\|_F
    +
    c_1\|K\|_F^2
\end{equation}
for all \(\|K\|_F\le\rho_0\).

Similarly, since \(\ell_{\star,n}^{\mathrm{ad}}\to\ell_\star>0\), a Taylor expansion of
\(x\mapsto\sqrt x\) gives a constant \(C_2>0\) such that
\begin{equation}
\label{eq:supp_penalty_lower_bound_K}
    \widehat\delta_{n,\alpha}
    \left[
    \sqrt{L_n^{\mathrm{ad}}(K)}
    -
    \sqrt{\ell_{\star,n}^{\mathrm{ad}}}
    \right]
    \ge
    -C_2\widehat\delta_{n,\alpha}\|K\|_F
\end{equation}
for all \(\|K\|_F\le\rho_0\).

Combining \eqref{eq:supp_nominal_lower_bound_K} and
\eqref{eq:supp_penalty_lower_bound_K}, we obtain
\[
    D_n(K)
    \ge
    -b_n\|K\|_F+c_1\|K\|_F^2
    \qquad
    \text{for all }\|K\|_F\le\rho_0,
\]
where
\[
    b_n:=C_1\|\widehat\Sigma_{12,n}\|_F+C_2\widehat\delta_{n,\alpha}.
\]
By Lemma~\ref{lem:supp_Zn_clt},
\[
    \|\widehat\Sigma_{12,n}\|_F=O_p(n^{-1/2}),
\]
and by Lemma~\ref{lem:supp_adaptive_delta_rate},
\[
    \widehat\delta_{n,\alpha}=O_p(n^{-1/2}),
\]
so
\[
    b_n=O_p(n^{-1/2}).
\]
Let \(F_n\) denote the high-probability event on which the preceding bounds hold uniformly
for \(\|K\|_F\le\rho_0\), and let
\[
    E_n:=\{\widehat Q_n^{\mathrm{ad}}\in\mathcal U_{\rho_0}\}.
\]
Then \(\Pr(E_n\cap F_n)\to1\). On \(E_n\cap F_n\),
\(\widehat K_n^{\mathrm{ad}}\in B_{\rho_0}\), and
\[
    0
    \ge
    D_n(\widehat K_n^{\mathrm{ad}})
    \ge
    -b_n\|\widehat K_n^{\mathrm{ad}}\|_F
    +
    c_1\|\widehat K_n^{\mathrm{ad}}\|_F^2 .
\]
Therefore
\[
    \|\widehat K_n^{\mathrm{ad}}\|_F\le \frac{b_n}{c_1}.
\]
Since \(b_n=O_p(n^{-1/2})\), it follows that
\[
    \widehat K_n^{\mathrm{ad}}=O_p(n^{-1/2}).
\]
This completes the proof.
\end{proof}

\subsection{Adaptive coordinate CLT}
\label{supp:sec_coordinate_clt}

\begin{proof}[Proof of Theorem~\ref{thm:main_coordinate_clt}]
Let
\[
    \widehat\Xi_n^{\mathrm{ad}}
    :=
    \sqrt n\,\widehat K_n^{\mathrm{ad}},
    \qquad
    W_n^{\mathrm{ad}}
    :=
    Z_n+\sqrt n\,\widehat\delta_{n,\alpha}g_{\widehat G_n}.
\]
Define
\[
    G_n^{\mathrm{ad}}(\Xi)
    :=
    \langle W_n^{\mathrm{ad}},\Xi\rangle_F
    +
    \frac12\langle\Xi,H_\star\Xi\rangle_F,
    \qquad
    \Xi\in\mathbb R^{r\times s}.
\]
Its unique minimiser is
\[
    \Xi_n^0:=-H_\star^{-1}W_n^{\mathrm{ad}}.
\]
By Lemma~\ref{lem:supp_Wn_adaptive},
\[
    W_n^{\mathrm{ad}}\Rightarrow Z+\tau_\alpha g_G,
    \qquad
    W_n^{\mathrm{ad}}=O_p(1).
\]
Therefore, by continuity of the linear map \(w\mapsto -H_\star^{-1}w\),
\[
    \Xi_n^0
    =
    -H_\star^{-1}W_n^{\mathrm{ad}}
    \Rightarrow
    -H_\star^{-1}(Z+\tau_\alpha g_G),
\]
and \(\Xi_n^0=O_p(1)\). Also, by Theorem~\ref{thm:main_consistency_rootn},
\[
    \widehat\Xi_n^{\mathrm{ad}}=O_p(1).
\]
Since \(H_\star\) is positive definite, there exists \(c_H>0\) such that
\[
    \langle\Xi,H_\star\Xi\rangle_F
    \ge
    c_H\|\Xi\|_F^2
    \qquad
    \text{for all }\Xi\in\mathbb R^{r\times s}.
\]
Because \(\nabla G_n^{\mathrm{ad}}(\Xi_n^0)=0\),
\begin{equation}
\label{eq:supp_strong_convexity_Gn}
    G_n^{\mathrm{ad}}(\Xi)-G_n^{\mathrm{ad}}(\Xi_n^0)
    \ge
    \frac{c_H}{2}\|\Xi-\Xi_n^0\|_F^2
    \qquad
    \text{for all }\Xi\in\mathbb R^{r\times s}.
\end{equation}
Fix \(\varepsilon>0\). Choose \(R>0\) so large that
\[
    \Pr\left(\|\widehat\Xi_n^{\mathrm{ad}}\|_F\le\frac{R}{2}\right)\ge1-\varepsilon,
    \qquad
    \Pr\left(\|\Xi_n^0\|_F\le\frac{R}{2}\right)\ge1-\varepsilon
\]
for all sufficiently large \(n\). Let
\[
    B_R:=\{\Xi\in\mathbb R^{r\times s}:\|\Xi\|_F\le R\},
\]
and
\[
    \eta_{n,R}
    :=
    \sup_{\Xi\in B_R}
    \left|
    \widetilde{\mathcal J}_n^{\mathrm{ad}}(\Xi)
    -
    G_n^{\mathrm{ad}}(\Xi)
    \right|.
\]
By Theorem~\ref{thm:main_local_quadratic},
\[
    \eta_{n,R}\xrightarrow{p}0.
\]
Let
\[
    E_{n,R}
    :=
    \left\{
    \|\widehat\Xi_n^{\mathrm{ad}}\|_F\le\frac{R}{2},
    \quad
    \|\Xi_n^0\|_F\le\frac{R}{2},
    \quad
    \eta_{n,R}\le1
    \right\}.
\]
Then \(\Pr(E_{n,R})\to1\), and on \(E_{n,R}\) both
\(\widehat\Xi_n^{\mathrm{ad}}\) and \(\Xi_n^0\) belong to \(B_R\).

For sufficiently large \(n\), \(n^{-1/2}R<\rho\), so the local chart is valid on \(B_R\).
Since \(\widehat Q_n^{\mathrm{ad}}\) is a global minimiser of
\(\mathcal J_n^{\mathrm{ad}}\) and
\[
    \widehat Q_n^{\mathrm{ad}}
    =
    Q(n^{-1/2}\widehat\Xi_n^{\mathrm{ad}})
\]
on \(E_{n,R}\), it follows that
\[
    \widetilde{\mathcal J}_n^{\mathrm{ad}}(\widehat\Xi_n^{\mathrm{ad}})
    \le
    \widetilde{\mathcal J}_n^{\mathrm{ad}}(\Xi)
    \qquad
    \text{for every }\Xi\in B_R.
\]
In particular,
\[
    \widetilde{\mathcal J}_n^{\mathrm{ad}}(\widehat\Xi_n^{\mathrm{ad}})
    \le
    \widetilde{\mathcal J}_n^{\mathrm{ad}}(\Xi_n^0).
\]
Therefore, on \(E_{n,R}\),
\[
\begin{aligned}
    G_n^{\mathrm{ad}}(\widehat\Xi_n^{\mathrm{ad}})
    -
    G_n^{\mathrm{ad}}(\Xi_n^0)
    &\le
    \left|
    G_n^{\mathrm{ad}}(\widehat\Xi_n^{\mathrm{ad}})
    -
    \widetilde{\mathcal J}_n^{\mathrm{ad}}(\widehat\Xi_n^{\mathrm{ad}})
    \right| \\
    &\quad
    +
    \widetilde{\mathcal J}_n^{\mathrm{ad}}(\widehat\Xi_n^{\mathrm{ad}})
    -
    \widetilde{\mathcal J}_n^{\mathrm{ad}}(\Xi_n^0)  \\
    &\quad
    +
    \left|
    \widetilde{\mathcal J}_n^{\mathrm{ad}}(\Xi_n^0)
    -
    G_n^{\mathrm{ad}}(\Xi_n^0)
    \right|  \\
    &\le
    2\eta_{n,R}.
\end{aligned}
\]
Combining this with \eqref{eq:supp_strong_convexity_Gn} gives
\[
    \frac{c_H}{2}
    \|\widehat\Xi_n^{\mathrm{ad}}-\Xi_n^0\|_F^2
    \le
    2\eta_{n,R}
    \qquad
    \text{on }E_{n,R}.
\]
Since \(\Pr(E_{n,R})\to1\) and \(\eta_{n,R}\xrightarrow{p}0\),
\[
    \widehat\Xi_n^{\mathrm{ad}}-\Xi_n^0\xrightarrow{p}0.
\]
Hence, by Slutsky's theorem,
\[
    \widehat\Xi_n^{\mathrm{ad}}
    \Rightarrow
    -H_\star^{-1}(Z+\tau_\alpha g_G).
\]
Recalling that
\[
    \widehat\Xi_n^{\mathrm{ad}}=\sqrt n\,\widehat K_n^{\mathrm{ad}},
\]
we obtain
\[
    \sqrt n\,\widehat K_n^{\mathrm{ad}}
    \Rightarrow
    -H_\star^{-1}(Z+\tau_\alpha g_G).
\]
For the vectorised form, note that
\[
    \operatorname{vec}(H_\star\Xi)
    =
    \mathcal H_\star\operatorname{vec}(\Xi),
    \qquad
    \mathcal H_\star
    =
    a_\star^{-1/2}(I_s\otimes\Lambda_1-\Lambda_2\otimes I_r).
\]
Therefore
\[
    \operatorname{vec}\{-H_\star^{-1}(Z+\tau_\alpha g_G)\}
    =
    -\mathcal H_\star^{-1}\operatorname{vec}(Z)
    -
    \tau_\alpha\mathcal H_\star^{-1}\operatorname{vec}(g_G).
\]
Since \(\operatorname{vec}(Z)\sim N(0,\Omega_Z)\), the limit law is Gaussian with mean
\[
    \mu_{\star,\alpha}
    =
    -\tau_\alpha\mathcal H_\star^{-1}\operatorname{vec}(g_G)
\]
and covariance
\[
    V_\star
    =
    \mathcal H_\star^{-1}\Omega_Z(\mathcal H_\star^{-1})^\top .
\]
This completes the proof.
\end{proof}

\subsection{Adaptive projector CLT}
\label{supp:sec_projector_clt}

\begin{proof}[Proof of Theorem~\ref{thm:main_projector_clt}]
By Lemma~\ref{lem:supp_Q_expand},
\[
    Q(n^{-1/2}\Xi)
    =
    Q_\star+n^{-1/2}DQ(0)[\Xi]+n^{-1}R_2(\Xi)+o(n^{-1}),
\]
uniformly for bounded \(\Xi\). Since
\[
    \widehat\Xi_n^{\mathrm{ad}}
    :=
    \sqrt n\,\widehat K_n^{\mathrm{ad}}
    =
    O_p(1)
\]
by Theorem~\ref{thm:main_consistency_rootn}, we have
\[
\begin{aligned}
    \sqrt n(\widehat Q_n^{\mathrm{ad}}-Q_\star)
    &=
    \sqrt n
    \left[
    Q(n^{-1/2}\widehat\Xi_n^{\mathrm{ad}})-Q_\star
    \right]  \\
    &=
    DQ(0)[\widehat\Xi_n^{\mathrm{ad}}]
    +
    n^{-1/2}R_2(\widehat\Xi_n^{\mathrm{ad}})
    +
    o_p(1).
\end{aligned}
\]
Because \(R_2(\cdot)\) is quadratic and \(\widehat\Xi_n^{\mathrm{ad}}=O_p(1)\),
\[
    n^{-1/2}R_2(\widehat\Xi_n^{\mathrm{ad}})=o_p(1).
\]
Thus
\[
    \sqrt n(\widehat Q_n^{\mathrm{ad}}-Q_\star)
    =
    DQ(0)[\widehat\Xi_n^{\mathrm{ad}}]+o_p(1).
\]
By Theorem~\ref{thm:main_coordinate_clt},
\[
    \widehat\Xi_n^{\mathrm{ad}}
    \Rightarrow
    \Xi_{\infty,\alpha}
    :=
    -H_\star^{-1}(Z+\tau_\alpha g_G).
\]
Applying the continuous mapping theorem to the linear map
\(\Xi\mapsto DQ(0)[\Xi]\) yields
\[
    \sqrt n(\widehat Q_n^{\mathrm{ad}}-Q_\star)
    \Rightarrow
    DQ(0)[\Xi_{\infty,\alpha}].
\]
Finally,
\[
    DQ(0)[\Xi_{\infty,\alpha}]
    =
    \Gamma
    \begin{pmatrix}
    0 & \Xi_{\infty,\alpha}\\
    \Xi_{\infty,\alpha}^\top & 0
    \end{pmatrix}
    \Gamma^\top .
\]
This proves the residual-projector limit. Since
\[
    \widehat\Pi_n^{\mathrm{ad}}:=I_p-\widehat Q_n^{\mathrm{ad}},
    \qquad
    \Pi_\star:=I_p-Q_\star,
\]
we also have
\[
    \sqrt n(\widehat\Pi_n^{\mathrm{ad}}-\Pi_\star)
    =
    -\sqrt n(\widehat Q_n^{\mathrm{ad}}-Q_\star),
\]
which gives the stated principal-projector limit.
\end{proof}
\subsection{Consistency of the exact adaptive DRO minimiser}
\label{supp:sec_exact_consistency}

Recall that
\[
    \Phi_n^{\mathrm{ad}}(Q)
    =
    \Phi_{n,\widehat\delta_{n,\alpha}}^{\widehat G_n}(Q)
    =
    \sup_{\mathbb P:
    W_{2,\widehat G_n}(\mathbb P,\widehat{\mathbb P}_n)
    \le
    \widehat\delta_{n,\alpha}}
    \mathbb E_{\mathbb P}(Z^\top QZ),
    \qquad Q\in\mathcal G_s .
\]
Let
\[
    \widehat Q_n^{\mathrm{ex,ad}}
    \in
    \arg\min_{Q\in\mathcal G_s}\Phi_n^{\mathrm{ad}}(Q).
\]

\begin{lemma}
\label{lem:supp_exact_consistency}
Suppose Assumptions~\ref{ass:main_moment_cov}, \ref{ass:main_random_G} and
\ref{ass:main_rwp_nondeg} hold, and suppose
\[
    \widehat\delta_{n,\alpha}=O_p(n^{-1/2}).
\]
Then
\[
    \widehat Q_n^{\mathrm{ex,ad}}\xrightarrow{p}Q_\star .
\]
Consequently, under the assumptions of Theorem~\ref{thm:main_consistency_rootn},
\[
    \|\widehat Q_n^{\mathrm{ex,ad}}-\widehat Q_n^{\mathrm{ad}}\|_F
    \xrightarrow{p}0 .
\]
\end{lemma}

\begin{proof}
Define the population residual risk
\[
    L(Q):=\operatorname{tr}(\Sigma Q),
    \qquad Q\in\mathcal G_s.
\]
We first show that
\[
    \sup_{Q\in\mathcal G_s}
    |\Phi_n^{\mathrm{ad}}(Q)-L(Q)|
    \xrightarrow{p}0.
\]
Since the empirical law \(\widehat{\mathbb P}_n\) belongs to its own Wasserstein ball,
\[
    W_{2,\widehat G_n}(\widehat{\mathbb P}_n,\widehat{\mathbb P}_n)=0
    \le
    \widehat\delta_{n,\alpha},
\]
we have, for every \(Q\in\mathcal G_s\),
\[
    \Phi_n^{\mathrm{ad}}(Q)
    \ge
    \mathbb E_{\widehat{\mathbb P}_n}(Z^\top QZ)
    =
    \operatorname{tr}(\widehat\Sigma_nQ).
\]
On the other hand, applying the weighted DRO upper bound conditionally on the realised
matrix \(\widehat G_n\), we obtain
\[
    \Phi_n^{\mathrm{ad}}(Q)
    \le
    \left\{
    \sqrt{\operatorname{tr}(\widehat\Sigma_nQ)}
    +
    \widehat\delta_{n,\alpha}
    \sqrt{\rho_{\widehat G_n}(Q)}
    \right\}^2 .
\]
Therefore,
\[
\begin{aligned}
    0
    &\le
    \Phi_n^{\mathrm{ad}}(Q)-\operatorname{tr}(\widehat\Sigma_nQ) \\
    &\le
    2\widehat\delta_{n,\alpha}
    \sqrt{
    \operatorname{tr}(\widehat\Sigma_nQ)\rho_{\widehat G_n}(Q)
    }
    +
    \widehat\delta_{n,\alpha}^2\rho_{\widehat G_n}(Q).
\end{aligned}
\]
Uniformly over \(Q\in\mathcal G_s\),
\[
    \operatorname{tr}(\widehat\Sigma_nQ)
    \le
    s\|\widehat\Sigma_n\|_{\operatorname{op}},
\]
and
\[
    \rho_{\widehat G_n}(Q)
    =
    \lambda_{\max}(\widehat G_n^{-1/2}Q\widehat G_n^{-1/2})
    \le
    \|\widehat G_n^{-1/2}\|_{\operatorname{op}}^2 .
\]
By Assumption~\ref{ass:main_moment_cov},
\[
    \|\widehat\Sigma_n-\Sigma\|_{\operatorname{op}}\xrightarrow{p}0,
\]
and hence \(\|\widehat\Sigma_n\|_{\operatorname{op}}=O_p(1)\). By
Assumption~\ref{ass:main_random_G},
\[
    \|\widehat G_n^{-1/2}\|_{\operatorname{op}}=O_p(1).
\]
Since \(\widehat\delta_{n,\alpha}=O_p(n^{-1/2})\), it follows that
\[
    \sup_{Q\in\mathcal G_s}
    \left|
    \Phi_n^{\mathrm{ad}}(Q)-\operatorname{tr}(\widehat\Sigma_nQ)
    \right|
    =
    o_p(1).
\]
Moreover,
\[
    \sup_{Q\in\mathcal G_s}
    \left|
    \operatorname{tr}\{(\widehat\Sigma_n-\Sigma)Q\}
    \right|
    \le
    s\|\widehat\Sigma_n-\Sigma\|_{\operatorname{op}}
    \xrightarrow{p}0.
\]
Combining the last two displays yields
\[
    \sup_{Q\in\mathcal G_s}
    |\Phi_n^{\mathrm{ad}}(Q)-L(Q)|
    \xrightarrow{p}0.
\]
It remains to identify the unique minimiser of \(L\). Since
\[
    \Sigma
    =
    \Gamma
    \begin{pmatrix}
    \Lambda_1 & 0\\
    0 & \Lambda_2
    \end{pmatrix}
    \Gamma^\top
\]
with
\[
    \theta_1\ge\cdots\ge\theta_r>\theta_{r+1}\ge\cdots\ge\theta_p,
\]
the rank-\(s\) projector minimising \(Q\mapsto\operatorname{tr}(\Sigma Q)\) is uniquely
\[
    Q_\star=U_{\star,\perp}U_{\star,\perp}^\top.
\]
Indeed, this is the projector onto the eigenspace corresponding to the smallest \(s\)
eigenvalues of \(\Sigma\), and the eigengap
\(\theta_r>\theta_{r+1}\) separates it from the leading \(r\)-dimensional eigenspace.

The space \(\mathcal G_s\) is compact, and \(L\) is continuous with unique minimiser
\(Q_\star\). The uniform convergence just proved and the standard argmin theorem therefore
give
\[
    \widehat Q_n^{\mathrm{ex,ad}}\xrightarrow{p}Q_\star.
\]
Finally, Theorem~\ref{thm:main_consistency_rootn} gives
\[
    \widehat Q_n^{\mathrm{ad}}\xrightarrow{p}Q_\star.
\]
Hence, by the triangle inequality,
\[
    \|\widehat Q_n^{\mathrm{ex,ad}}-\widehat Q_n^{\mathrm{ad}}\|_F
    \le
    \|\widehat Q_n^{\mathrm{ex,ad}}-Q_\star\|_F
    +
    \|\widehat Q_n^{\mathrm{ad}}-Q_\star\|_F
    \xrightarrow{p}0.
\]
\end{proof}
\section{Stability of the diagonal adaptive geometries}
\label{supp:sec_adaptive_geometries}

The main asymptotic theory assumes that the adaptive transport matrix
\(\widehat G_n\) converges in operator norm to a deterministic positive definite limit and
that its eigenvalues remain bounded away from zero and infinity with probability tending to
one. We record here a simple fixed-dimensional verification for the two diagonal
geometries used in the numerical studies.

Let \(\widehat U_{\mathrm{PCA}}\in\mathbb V_{p,r}\) be the ordinary empirical PCA basis,
and write
\[
    \widehat P_{\mathrm{PCA}}
    =
    \widehat U_{\mathrm{PCA}}\widehat U_{\mathrm{PCA}}^\top,
    \qquad
    \widehat Q_{\mathrm{PCA}}
    =
    I_p-\widehat P_{\mathrm{PCA}} .
\]
Recall that the residual and PCA-block diagonal weights are
\[
    \widehat v_j
    =
    \left[
        \widehat Q_{\mathrm{PCA}}
        \widehat\Sigma_n
        \widehat Q_{\mathrm{PCA}}
    \right]_{jj},
    \qquad
    \widehat w_j
    =
    \left[
        \widehat P_{\mathrm{PCA}}
        \widehat\Sigma_n
        \widehat P_{\mathrm{PCA}}
    \right]_{jj},
    \qquad j=1,\ldots,p .
\]
For a ridge parameter \(\tau_n>0\) and positive normalising constants
\(c_n^{\mathrm{res}}\) and \(c_n^{\mathrm{pca}}\), the corresponding transport matrices are
\[
    \widehat G_n^{\mathrm{res}}
    =
    c_n^{\mathrm{res}}
    \operatorname{diag}\{(\widehat v_1+\tau_n)^{-1},\ldots,
    (\widehat v_p+\tau_n)^{-1}\},
\]
and
\[
    \widehat G_n^{\mathrm{pca}}
    =
    c_n^{\mathrm{pca}}
    \operatorname{diag}\{(\widehat w_1+\tau_n)^{-1},\ldots,
    (\widehat w_p+\tau_n)^{-1}\}.
\]

\begin{proposition}[Fixed-dimensional stability of the diagonal geometries]
\label{prop:supp_diagonal_geometry_stability}
Suppose \(p\) is fixed, \(\widehat\Sigma_n\to\Sigma\) in operator norm, and
\(\widehat P_{\mathrm{PCA}}\to\Pi_\star\) in operator norm, where
\(\Pi_\star\) is the population rank-\(r\) PCA projector and
\(Q_\star=I_p-\Pi_\star\). Suppose also that
\[
    \tau_n\to\tau>0,
    \qquad
    c_n^{\mathrm{res}}\to c^{\mathrm{res}}\in(0,\infty),
    \qquad
    c_n^{\mathrm{pca}}\to c^{\mathrm{pca}}\in(0,\infty).
\]
Then
\[
    \widehat G_n^{\mathrm{res}}\to G^{\mathrm{res}},
    \qquad
    \widehat G_n^{\mathrm{pca}}\to G^{\mathrm{pca}}
\]
in operator norm, where
\[
    G^{\mathrm{res}}
    =
    c^{\mathrm{res}}
    \operatorname{diag}
    \left\{
        ([Q_\star\Sigma Q_\star]_{11}+\tau)^{-1},
        \ldots,
        ([Q_\star\Sigma Q_\star]_{pp}+\tau)^{-1}
    \right\},
\]
and
\[
    G^{\mathrm{pca}}
    =
    c^{\mathrm{pca}}
    \operatorname{diag}
    \left\{
        ([\Pi_\star\Sigma\Pi_\star]_{11}+\tau)^{-1},
        \ldots,
        ([\Pi_\star\Sigma\Pi_\star]_{pp}+\tau)^{-1}
    \right\}.
\]
Moreover, \(G^{\mathrm{res}}\) and \(G^{\mathrm{pca}}\) are positive definite, and the
eigenvalues of \(\widehat G_n^{\mathrm{res}}\) and
\(\widehat G_n^{\mathrm{pca}}\) are bounded away from zero and infinity with probability
tending to one.
\end{proposition}

\begin{proof}
We prove the claim for \(\widehat G_n^{\mathrm{res}}\); the proof for
\(\widehat G_n^{\mathrm{pca}}\) is identical. Since
\(\widehat P_{\mathrm{PCA}}\to\Pi_\star\) in operator norm, we also have
\(\widehat Q_{\mathrm{PCA}}\to Q_\star\). Hence
\[
    \widehat Q_{\mathrm{PCA}}\widehat\Sigma_n\widehat Q_{\mathrm{PCA}}
    \to
    Q_\star\Sigma Q_\star
\]
in operator norm. In fixed dimension this implies convergence of every diagonal entry:
\[
    \widehat v_j
    =
    [\widehat Q_{\mathrm{PCA}}\widehat\Sigma_n
    \widehat Q_{\mathrm{PCA}}]_{jj}
    \to
    [Q_\star\Sigma Q_\star]_{jj},
    \qquad j=1,\ldots,p .
\]
Because \(\tau_n\to\tau>0\), the map \(x\mapsto (x+\tau_n)^{-1}\) is eventually uniformly
continuous on the relevant bounded set and its denominator is bounded away from zero.
Together with \(c_n^{\mathrm{res}}\to c^{\mathrm{res}}\in(0,\infty)\), this gives
entrywise convergence of the diagonal entries of \(\widehat G_n^{\mathrm{res}}\) to those
of \(G^{\mathrm{res}}\). Since the matrices are diagonal and \(p\) is fixed, entrywise
convergence is equivalent to operator-norm convergence.

The same argument gives
\[
    \widehat P_{\mathrm{PCA}}\widehat\Sigma_n\widehat P_{\mathrm{PCA}}
    \to
    \Pi_\star\Sigma\Pi_\star
\]
and hence \(\widehat G_n^{\mathrm{pca}}\to G^{\mathrm{pca}}\). The limiting diagonal
entries are strictly positive because \(\tau>0\) and the limiting normalising constants are
positive. Therefore the limiting matrices are positive definite. Since the diagonal entries
of the estimated matrices converge to positive finite limits, their smallest and largest
eigenvalues are bounded away from zero and infinity with probability tending to one.
\end{proof}

The matrices
\(\widehat P_{\mathrm{PCA}}\widehat\Sigma_n\widehat P_{\mathrm{PCA}}\) and
\(\widehat Q_{\mathrm{PCA}}\widehat\Sigma_n\widehat Q_{\mathrm{PCA}}\) may be rank
deficient. This does not affect the positive definiteness of the transport matrices,
because \(\widehat G_n^{\mathrm{res}}\) and \(\widehat G_n^{\mathrm{pca}}\) are not the
PCA-block or residual-block covariance matrices themselves. They are diagonal
inverse-ridge constructions based on the diagonal entries of those matrices. The ridge
\(\tau_n>0\) ensures that every diagonal entry of the resulting transport matrix is finite
and strictly positive.
\section{Numerical optimisation of the adaptive surrogate}
\label{sec:supp_numerical_optimisation}
\label{supp:sec_algorithm}
This section describes the numerical optimisation procedures used for the adaptive
DRO-PCA surrogate. For a fixed adaptive geometry \(\widehat G_n\) and RWPI-calibrated
radius \(\widehat\delta\), the objective can be written as a function of a rank-\(r\)
orthonormal basis \(U\in\mathbb V_{p,r}\):
\begin{equation}
\label{eq:supp_surrogate_U}
    \mathcal J_{n,\widehat\delta}^{\widehat G_n}(U)
    =
    \sqrt{
        \operatorname{tr}
        \left\{
            \widehat\Sigma_n(I_p-UU^\top)
        \right\}
    }
    +
    \widehat\delta
    \sqrt{
        \lambda_{\max}
        \left[
            \widehat G_n^{-1/2}
            (I_p-UU^\top)
            \widehat G_n^{-1/2}
        \right]
    } .
\end{equation}
The criterion depends on \(U\) only through the projector \(UU^\top\), and therefore is a
function on the Grassmann manifold of rank-\(r\) subspaces. The optimisation problem is
non-convex, so the numerical procedures below should be interpreted as reproducible
algorithms for computing stable local solutions rather than as global optimisation
guarantees.

\subsection{Path-based grid implementation}
\label{subsec:supp_grid_solver}

The main numerical experiments use a path-based implementation of
\eqref{eq:supp_surrogate_U}. For a chosen geometry \(\widehat G_n\), let
\[
    s_n
    =
    \frac1p\operatorname{tr}(\widehat\Sigma_n)
\]
be the average empirical coordinate variance. For the diagonal geometries used in the numerical studies, let
\[
    b_n
    :=
    \operatorname{diag}(\widehat G_n^{-1})\in\mathbb R^p
\]
denote the vector of inverse transport weights, and define the mean-normalised weights
\[
    d_n
    :=
    \frac{b_n}{p^{-1}\sum_{j=1}^p b_{n,j}} .
\]
Thus \(p^{-1}\sum_{j=1}^p d_{n,j}=1\), so the path scale is controlled by
\(\widehat\delta s_n\), while \(d_n\) determines only the relative coordinate weights. This normalisation matches the implementation used in the simulations and prevents the
path scale from changing when \(\widehat G_n\) is multiplied by a positive constant.

For a deterministic grid \(\mathcal T_\gamma\subset[0,\infty)\), we form candidate
subspaces
\[
    \widehat U_\gamma
    =
    \operatorname{Eig}_r
    \left\{
        \widehat\Sigma_n
        +
        \gamma\,\widehat\delta\,s_n
        \operatorname{Diag}(d_n)
    \right\},
    \qquad
    \gamma\in\mathcal T_\gamma,
\]
where \(\operatorname{Eig}_r(A)\) denotes the leading rank-\(r\) eigenspace of the
symmetric matrix \(A\), and \(\operatorname{Diag}(d_n)\) denotes the diagonal matrix with
diagonal vector \(d_n\). We also include the ordinary PCA subspace and, for diagonal
geometries, the coordinate subspace spanned by the \(r\) largest entries of \(d_n\). Among
these candidates we select
\[
    \widehat U_{\mathrm{path}}
    \in
    \arg\min_{U\in\mathcal C_{\mathcal T_\gamma}}
    \mathcal J_{n,\widehat\delta}^{\widehat G_n}(U),
\]
where \(\mathcal C_{\mathcal T_\gamma}\) is the finite candidate set generated by the path.

This path-based procedure is computationally stable because it reduces the search to a
sequence of ordinary eigenvalue problems. It is also reproducible, since the same grid is
used for all geometries and all Monte Carlo replications. However, the path is only a
restricted family of candidate subspaces. The robust term in
\eqref{eq:supp_surrogate_U} depends nonlinearly on \(UU^\top\), so the path-based solution
need not coincide with a local minimiser of the full Grassmann problem.

The tuning choices in the path-based implementation are not part of the statistical
definition of adaptive DRO-PCA. They are numerical choices used to compute the reported
tables. In all experiments, the same deterministic path construction, ridge stabilisation,
weight normalisation and candidate-selection rule are applied to ordinary PCA,
\(\widehat G_n^{\mathrm{res}}\) and \(\widehat G_n^{\mathrm{pca}}\) wherever relevant, so
that the comparisons isolate the effect of the transport geometry rather than
method-specific optimisation tuning. The exact numerical values used for the path grid,
ridge, clipping and Monte Carlo approximation of the RWPI quantile are given in the
accompanying reproducibility code.
\subsection{Alternative direct optimisation}
\label{subsec:supp_manifold_solver}

The path-based estimator used in the numerical studies is a deterministic finite-candidate
implementation of the surrogate criterion. The full surrogate objective
\eqref{eq:supp_surrogate_U} may also be optimised directly on the Grassmann manifold using
standard Riemannian optimisation methods. Such a direct approach treats
\(U\in\mathbb V_{p,r}\) as the optimisation variable, uses
\(Q_U=I_p-UU^\top\), and minimises
\[
    \mathcal J_{n,\widehat\delta}^{\widehat G_n}(U)
    =
    \sqrt{\operatorname{tr}(\widehat\Sigma_nQ_U)}
    +
    \widehat\delta
    \sqrt{
        \lambda_{\max}
        (\widehat G_n^{-1/2}Q_U\widehat G_n^{-1/2})
    } .
\]
When the leading eigenvalue in the exposure term is simple, gradients can be obtained by
differentiating the empirical reconstruction term and the leading-eigenvalue map; otherwise
one may use a subgradient direction. A Riemannian gradient or proximal-gradient method,
combined with a standard retraction onto the Stiefel or Grassmann manifold, then gives a
direct local solver.
\subsection{Rationale for the path-based implementation}
\label{subsec:supp_solver_rationale}

The numerical studies in the main text use the path-based estimator
\(\widehat U_{\mathrm{path}}\) defined in
Section~\ref{subsec:supp_grid_solver}. This choice is deliberate. Although direct
manifold optimisation of \eqref{eq:supp_surrogate_U} is possible, the objective is
non-convex and contains a spectral maximum term. Consequently, a manifold implementation
requires choices of initial subspaces, step-size rules, stopping tolerances and local
descent safeguards. These choices can affect the selected local solution, especially when
the empirical geometry is highly anisotropic.

By contrast, the path-based implementation reduces the computation to a deterministic
collection of ordinary eigenvalue problems. For each geometry \(\widehat G_n\), the same grid \(\mathcal T_\gamma\) is used across all
replications, and the final estimator is selected by evaluating the original surrogate
objective \eqref{eq:supp_surrogate_U} on the resulting finite candidate set. The procedure is therefore reproducible, stable across Monte Carlo
replications, and directly comparable across the residual and PCA-block geometries. These
properties are particularly important in the numerical experiments, where the goal is to
compare the statistical effect of different transport geometries rather than the behaviour
of competing non-convex optimisation routines.

The path should not be interpreted as a claim of global optimisation over the full
Grassmann manifold. Rather, it is a structured approximation to the adaptive surrogate
problem that moves the fitted subspace in directions favoured by
\(\operatorname{diag}(\widehat G_n^{-1})\), while still selecting the final candidate using
the original criterion. This is sufficient for the purposes of the numerical section:
the reported estimators correspond to a fixed, transparent and reproducible computational
rule. The same rule is applied to \(\widehat G_n^{\mathrm{res}}\) and
\(\widehat G_n^{\mathrm{pca}}\), so the comparisons isolate the effect of the adaptive
geometry rather than differences in optimisation tuning.

The manifold formulation in Section~\ref{subsec:supp_manifold_solver} remains useful as a
sensitivity check or as an alternative implementation when one wishes to study local
solutions of the full surrogate more directly. In the main experiments, however, we report
the path-based estimator because it gives a stable finite-sample procedure, avoids
run-to-run variation from random initialisations, and provides a clear algorithmic object
for comparing adaptive transport geometries.

\section{Additional numerical results}
\label{sec:supp_additional_numerical}

This section reports additional Monte Carlo summaries for the two simulation designs in
Section~\ref{sec:numerical}. The main text reports percentage target-risk gains and the
main visibility diagnostics in order to display the principal patterns compactly. To avoid
duplicating those tables, we report here complementary quantities: target risk, excess
target risk, excess-risk gain, win rates and fitted-subspace overlap diagnostics.

For a fitted rank-\(r\) basis \(\widehat U\), let
\(R_{\mathrm{tar}}(\widehat U)\) denote the population target reconstruction risk and let
\[
    \operatorname{Excess}(\widehat U)
    =
    R_{\mathrm{tar}}(\widehat U)
    -
    R_{\mathrm{tar}}(U_{\mathrm{tar}})
\]
be the excess target risk relative to the target oracle \(U_{\mathrm{tar}}\). In the
supplementary tables, the excess-risk gain is
\[
    100\,
    \frac{
        \operatorname{Excess}(\widehat U_{\mathrm{PCA}})
        -
        \operatorname{Excess}(\widehat U)
    }{
        \operatorname{Excess}(\widehat U_{\mathrm{PCA}})
    } ,
\]
where \(\widehat U_{\mathrm{PCA}}\) is ordinary PCA fitted from the same training sample.
The win rate is the Monte Carlo proportion of replications in which the fitted method has
lower target reconstruction risk than ordinary PCA. All DRO estimates use the adaptive
surrogate criterion with the same numerical implementation and RWPI calibration as
described in Section~\ref{sec:numerical} and
Supplementary Section~\ref{sec:supp_numerical_optimisation}..

\subsection{Additional results for the covariance-shift visibility design}
\label{subsec:supp_visibility_shift_additional}

We first give additional summaries for the covariance-shift design in
Section~\ref{subsec:visibility_shift_design}. Recall that
\[
    \mathcal V_0=\operatorname{span}(e_1,e_2,e_3)
\]
is the target-shift coordinate subspace. The parameter \(\ell\) controls how visible
\(\mathcal V_0\) is in the source PCA block, while \(\eta\) controls the strength of the
target covariance shift toward \(\mathcal V_0\). The main text reports the target-risk
gain table over the full grid of \((\ell,\eta)\). Here we report fitted-overlap movement
relative to PCA and full risk-scale quantities for the largest shift level \(\eta=1.5\).

For a fitted basis \(\widehat U\), define the target-shift overlap
\[
    O_{\mathcal V_0}(\widehat U)
    :=
    \operatorname{tr}(\widehat U^\top V_0V_0^\top \widehat U).
\]
This quantity is between \(0\) and \(r=3\), and equals the sum of squared cosines between
the fitted subspace and \(\mathcal V_0\). Table~\ref{tab:supp_visibility_overlap_change}
reports
\[
    O_{\mathcal V_0}(\widehat U)-O_{\mathcal V_0}(\widehat U_{\mathrm{PCA}}),
\]
so positive values indicate that the fitted DRO subspace is more aligned with the
target-shift subspace than ordinary PCA fitted from the same source sample.

\begin{table}[t]
\centering
\small
\setlength{\tabcolsep}{5pt}
\caption{Change in fitted target-shift overlap relative to ordinary PCA in the
covariance-shift visibility design. Entries are Monte Carlo averages over 1000
replications. Positive values indicate greater overlap with
\(\mathcal V_0=\operatorname{span}(e_1,e_2,e_3)\) than ordinary PCA.}
\label{tab:supp_visibility_overlap_change}
\begin{tabular}{ccc}
\hline
\(\ell\)
& \(\widehat G_n^{\mathrm{res}}\)
& \(\widehat G_n^{\mathrm{pca}}\) \\
\hline
0.00 &  0.792 & 0.231 \\
0.05 &  0.556 & 0.406 \\
0.15 &  0.177 & 0.612 \\
0.30 & -0.037 & 0.683 \\
0.50 & -0.039 & 0.508 \\
\hline
\end{tabular}
\end{table}

Table~\ref{tab:supp_visibility_overlap_change} supports the interpretation in the main
text. When \(\ell\) is small, the residual geometry produces the larger increase in
overlap with the target-shift subspace. When \(\ell\) is moderate or large, the PCA-block
geometry produces the larger increase. Thus the two adaptive geometries move the fitted
subspace in different directions depending on whether the target-relevant variation is
visible in the residual block or in the PCA block of the source sample.

\begin{table}[t]
\centering
\scriptsize
\setlength{\tabcolsep}{2.8pt}
\caption{Additional risk summaries for the covariance-shift visibility design at
\(\eta=1.5\). Entries are Monte Carlo averages over 1000 replications. Parentheses in the
target-risk column are Monte Carlo standard errors. Excess gain is the percentage reduction
in excess target risk relative to ordinary PCA.}
\label{tab:supp_visibility_risk_eta15}
\begin{tabular}{c l r c r r r r}
\hline
\(\ell\)
& Method
& Oracle risk
& Target risk
& Excess risk
& Excess gain
& Win rate
& Fitted overlap \\
\hline
0.00 & PCA
     & 30.500
     & 40.884 {\scriptsize (0.051)}
     & 10.384 &  0.00 & --    & 0.947 \\
     & \(\widehat G_n^{\mathrm{res}}\)
     &
     & 36.906 {\scriptsize (0.087)}
     &  6.406 & 36.23 & 0.977 & 1.739 \\
     & \(\widehat G_n^{\mathrm{pca}}\)
     &
     & 39.739 {\scriptsize (0.081)}
     &  9.239 & 12.34 & 0.863 & 1.178 \\
0.05 & PCA
     & 29.295
     & 38.166 {\scriptsize (0.051)}
     &  8.871 &  0.00 & --    & 1.184 \\
     & \(\widehat G_n^{\mathrm{res}}\)
     &
     & 35.118 {\scriptsize (0.060)}
     &  5.823 & 32.36 & 0.957 & 1.740 \\
     & \(\widehat G_n^{\mathrm{pca}}\)
     &
     & 36.214 {\scriptsize (0.084)}
     &  6.918 & 24.44 & 0.966 & 1.590 \\
0.15 & PCA
     & 27.298
     & 33.884 {\scriptsize (0.047)}
     &  6.586 &  0.00 & --    & 1.556 \\
     & \(\widehat G_n^{\mathrm{res}}\)
     &
     & 32.710 {\scriptsize (0.033)}
     &  5.413 & 14.77 & 0.835 & 1.733 \\
     & \(\widehat G_n^{\mathrm{pca}}\)
     &
     & 30.734 {\scriptsize (0.072)}
     &  3.437 & 51.48 & 0.998 & 2.168 \\
0.30 & PCA
     & 24.870
     & 29.397 {\scriptsize (0.032)}
     &  4.527 &  0.00 & --    & 1.957 \\
     & \(\widehat G_n^{\mathrm{res}}\)
     &
     & 29.607 {\scriptsize (0.030)}
     &  4.737 & -5.77 & 0.132 & 1.920 \\
     & \(\widehat G_n^{\mathrm{pca}}\)
     &
     & 25.920 {\scriptsize (0.026)}
     &  1.050 & 77.96 & 1.000 & 2.640 \\
0.50 & PCA
     & 22.211
     & 25.271 {\scriptsize (0.021)}
     &  3.060 &  0.00 & --    & 2.315 \\
     & \(\widehat G_n^{\mathrm{res}}\)
     &
     & 25.540 {\scriptsize (0.024)}
     &  3.329 & -9.10 & 0.004 & 2.276 \\
     & \(\widehat G_n^{\mathrm{pca}}\)
     &
     & 22.707 {\scriptsize (0.007)}
     &  0.496 & 83.61 & 1.000 & 2.823 \\
\hline
\end{tabular}
\end{table}

The risk-scale results in Table~\ref{tab:supp_visibility_risk_eta15} are consistent with
the gain patterns reported in the main text. In the low-visibility regimes
\(\ell=0\) and \(\ell=0.05\), the residual geometry gives the smaller target risk and the
smaller excess risk. In the moderate- and high-visibility regimes
\(\ell=0.15,0.30,0.50\), the PCA-block geometry gives the smaller risk and the larger
excess-risk reduction. 

\subsection{Additional results for the training-contamination design}
\label{subsec:supp_training_contam_additional}

We next report additional summaries for the training-contamination design in
Section~\ref{subsec:training_contamination_visibility}. Recall that the clean oracle
subspace is
\[
    \mathcal V_0=\operatorname{span}(e_1,e_2,e_3),
\]
while \(C_\kappa\) is the contamination subspace. The parameter \(\kappa\) controls the
alignment between the contamination subspace and the clean target-relevant subspace, and
\(\epsilon\) is the contamination proportion. The main text reports the clean-in-PCA
diagnostic and the target-risk gain table over the full grid of \((\kappa,\epsilon)\).
Here we report complementary diagnostics and full risk-scale summaries for
\(\epsilon=0.10\) and \(\epsilon=0.20\).

Table~\ref{tab:supp_contam_visibility_extra} reports two diagnostics for ordinary PCA
fitted from the contaminated training sample. The clean-in-residual column is
\[
    r-
    \operatorname{tr}
    \left(
        \widehat U_{\mathrm{PCA}}^\top V_0V_0^\top\widehat U_{\mathrm{PCA}}
    \right),
\]
which measures how much of the clean oracle subspace is absent from the contaminated PCA
block. The contamination-in-PCA column is
\[
    \operatorname{tr}
    \left(
        \widehat U_{\mathrm{PCA}}^\top C_\kappa C_\kappa^\top
        \widehat U_{\mathrm{PCA}}
    \right),
\]
which measures how strongly the empirical PCA block is aligned with the contamination
subspace.

\begin{table}[t]
\centering
\small
\setlength{\tabcolsep}{5pt}
\caption{Additional visibility diagnostics for the training-contamination design at
\(\epsilon=0.10\) and \(\epsilon=0.20\). Entries are Monte Carlo averages over 1000
replications.}
\label{tab:supp_contam_visibility_extra}
\begin{tabular}{cccc}
\hline
\(\kappa\)
& \(\epsilon\)
& Clean in residual
& Contamination in PCA \\
\hline
0.00 & 0.10 & 1.671 & 1.604 \\
0.00 & 0.20 & 2.590 & 2.549 \\
0.10 & 0.10 & 1.366 & 1.829 \\
0.10 & 0.20 & 2.123 & 2.635 \\
0.30 & 0.10 & 0.891 & 2.198 \\
0.30 & 0.20 & 1.437 & 2.726 \\
0.60 & 0.10 & 0.469 & 2.580 \\
0.60 & 0.20 & 0.721 & 2.828 \\
1.00 & 0.10 & 0.057 & 2.943 \\
1.00 & 0.20 & 0.035 & 2.965 \\
\hline
\end{tabular}
\end{table}

The diagnostics show that when \(\kappa\) is small and \(\epsilon\) is large, the clean
oracle subspace is substantially displaced from the contaminated PCA block. For example,
at \((\kappa,\epsilon)=(0,0.20)\), the clean-in-residual overlap is \(2.590\), close to
the maximum possible value \(3\). In contrast, when \(\kappa=1\), the contamination is
aligned with the clean signal and the clean subspace remains almost fully contained in the
empirical PCA block.

For the fitted methods, we report two overlap quantities:
\[
    O_{\mathrm{cl}}(\widehat U)
    :=
    \operatorname{tr}(\widehat U^\top V_0V_0^\top \widehat U),
    \qquad
    O_{\mathrm{cont}}(\widehat U)
    :=
    \operatorname{tr}(\widehat U^\top C_\kappa C_\kappa^\top \widehat U).
\]
The first measures alignment with the clean oracle subspace, and the second measures
alignment with the contamination subspace.

\begin{table}[t]
\centering
\scriptsize
\setlength{\tabcolsep}{3.2pt}
\caption{Additional risk summaries for the training-contamination design at
\(\epsilon=0.10\). Entries are Monte Carlo averages over 1000 replications. Parentheses in
the target-risk column are Monte Carlo standard errors. Excess gain is the percentage
reduction in excess target risk relative to ordinary PCA fitted from the same contaminated
sample.}
\label{tab:supp_contam_risk_eps10}
\begin{tabular}{c l r r r r r r}
\hline
\(\kappa\)
& Method
& Target risk
& Excess risk
& Excess gain
& Win rate
& \(O_{\mathrm{cl}}\)
& \(O_{\mathrm{cont}}\) \\
\hline
0.00 & PCA
     & 25.211 {\scriptsize (0.060)}
     & 8.211 &  0.00 & --    & 1.329 & 1.604 \\
     & \(\widehat G_n^{\mathrm{res}}\)
     & 23.033 {\scriptsize (0.052)}
     & 6.033 & 22.89 & 0.832 & 1.811 & 1.140 \\
     & \(\widehat G_n^{\mathrm{pca}}\)
     & 24.222 {\scriptsize (0.082)}
     & 7.222 & 14.06 & 0.882 & 1.519 & 1.407 \\
0.10 & PCA
     & 23.737 {\scriptsize (0.048)}
     & 6.737 &  0.00 & --    & 1.634 & 1.829 \\
     & \(\widehat G_n^{\mathrm{res}}\)
     & 22.767 {\scriptsize (0.036)}
     & 5.767 & 11.63 & 0.784 & 1.848 & 1.649 \\
     & \(\widehat G_n^{\mathrm{pca}}\)
     & 21.881 {\scriptsize (0.074)}
     & 4.881 & 30.94 & 0.991 & 2.001 & 1.453 \\
0.30 & PCA
     & 21.413 {\scriptsize (0.032)}
     & 4.413 &  0.00 & --    & 2.109 & 2.198 \\
     & \(\widehat G_n^{\mathrm{res}}\)
     & 21.545 {\scriptsize (0.030)}
     & 4.545 & -3.90 & 0.211 & 2.084 & 2.224 \\
     & \(\widehat G_n^{\mathrm{pca}}\)
     & 18.753 {\scriptsize (0.040)}
     & 1.753 & 63.00 & 1.000 & 2.645 & 1.662 \\
0.60 & PCA
     & 19.328 {\scriptsize (0.015)}
     & 2.328 &  0.00 & --    & 2.531 & 2.580 \\
     & \(\widehat G_n^{\mathrm{res}}\)
     & 19.459 {\scriptsize (0.016)}
     & 2.459 & -5.83 & 0.001 & 2.504 & 2.595 \\
     & \(\widehat G_n^{\mathrm{pca}}\)
     & 17.735 {\scriptsize (0.012)}
     & 0.735 & 68.60 & 1.000 & 2.852 & 2.275 \\
1.00 & PCA
     & 17.285 {\scriptsize (0.003)}
     & 0.285 &  0.00 & --    & 2.943 & 2.943 \\
     & \(\widehat G_n^{\mathrm{res}}\)
     & 17.298 {\scriptsize (0.003)}
     & 0.298 & -4.29 & 0.001 & 2.940 & 2.940 \\
     & \(\widehat G_n^{\mathrm{pca}}\)
     & 17.132 {\scriptsize (0.001)}
     & 0.132 & 52.01 & 1.000 & 2.974 & 2.974 \\
\hline
\end{tabular}
\end{table}

\begin{table}[t]
\centering
\scriptsize
\setlength{\tabcolsep}{3.2pt}
\caption{Additional risk summaries for the training-contamination design at
\(\epsilon=0.20\). Entries are Monte Carlo averages over 1000 replications. Parentheses in
the target-risk column are Monte Carlo standard errors. Excess gain is the percentage
reduction in excess target risk relative to ordinary PCA fitted from the same contaminated
sample.}
\label{tab:supp_contam_risk_eps20}
\begin{tabular}{c l r r r r r r}
\hline
\(\kappa\)
& Method
& Target risk
& Excess risk
& Excess gain
& Win rate
& \(O_{\mathrm{cl}}\)
& \(O_{\mathrm{cont}}\) \\
\hline
0.00 & PCA
     & 29.902 {\scriptsize (0.054)}
     & 12.902 &  0.00 & --    & 0.410 & 2.549 \\
     & \(\widehat G_n^{\mathrm{res}}\)
     & 28.459 {\scriptsize (0.083)}
     & 11.459 & 11.64 & 0.952 & 0.699 & 2.266 \\
     & \(\widehat G_n^{\mathrm{pca}}\)
     & 30.008 {\scriptsize (0.061)}
     & 13.008 & -0.63 & 0.287 & 0.387 & 2.545 \\
0.10 & PCA
     & 27.575 {\scriptsize (0.040)}
     & 10.575 &  0.00 & --    & 0.877 & 2.635 \\
     & \(\widehat G_n^{\mathrm{res}}\)
     & 25.985 {\scriptsize (0.052)}
     &  8.985 & 15.20 & 0.992 & 1.196 & 2.396 \\
     & \(\widehat G_n^{\mathrm{pca}}\)
     & 27.195 {\scriptsize (0.051)}
     & 10.195 &  4.00 & 0.844 & 0.949 & 2.544 \\
0.30 & PCA
     & 24.159 {\scriptsize (0.027)}
     &  7.159 &  0.00 & --    & 1.563 & 2.726 \\
     & \(\widehat G_n^{\mathrm{res}}\)
     & 23.680 {\scriptsize (0.022)}
     &  6.680 &  6.31 & 0.935 & 1.659 & 2.688 \\
     & \(\widehat G_n^{\mathrm{pca}}\)
     & 22.871 {\scriptsize (0.044)}
     &  5.871 & 18.90 & 0.996 & 1.820 & 2.559 \\
0.60 & PCA
     & 20.588 {\scriptsize (0.014)}
     &  3.588 &  0.00 & --    & 2.279 & 2.828 \\
     & \(\widehat G_n^{\mathrm{res}}\)
     & 20.699 {\scriptsize (0.014)}
     &  3.699 & -3.16 & 0.003 & 2.257 & 2.837 \\
     & \(\widehat G_n^{\mathrm{pca}}\)
     & 18.751 {\scriptsize (0.020)}
     &  1.751 & 52.01 & 1.000 & 2.648 & 2.588 \\
1.00 & PCA
     & 17.176 {\scriptsize (0.002)}
     &  0.176 &  0.00 & --    & 2.965 & 2.965 \\
     & \(\widehat G_n^{\mathrm{res}}\)
     & 17.181 {\scriptsize (0.002)}
     &  0.181 & -3.10 & 0.003 & 2.964 & 2.964 \\
     & \(\widehat G_n^{\mathrm{pca}}\)
     & 17.101 {\scriptsize (0.001)}
     &  0.101 & 41.30 & 1.000 & 2.980 & 2.980 \\
\hline
\end{tabular}
\end{table}

Tables~\ref{tab:supp_contam_risk_eps10} and \ref{tab:supp_contam_risk_eps20} provide the
risk-scale analogue of the gain table in the main text without repeating the target-risk
gain columns. In this contamination design, the target distribution is the fixed clean distribution
\(\Sigma_{\mathrm{cl}}\), so the target oracle subspace is
\(\operatorname{span}(V_0)\) for every value of \((\kappa,\epsilon)\). The corresponding
oracle target risk is \(17\), and hence the excess-risk column is equal to the target-risk
column minus \(17\). At \(\epsilon=0.10\), the residual geometry is preferable when
\(\kappa=0\), where the clean signal is substantially hidden from the contaminated PCA
block. For larger \(\kappa\), the clean subspace remains more visible in the PCA block and
the PCA-block geometry gives the smaller risk. At \(\epsilon=0.20\), the residual
geometry remains preferable for \(\kappa=0\) and \(\kappa=0.10\), while the PCA-block
geometry is preferable for moderate and large alignment. These patterns confirm that the
relative performance of \(\widehat G_n^{\mathrm{res}}\) and
\(\widehat G_n^{\mathrm{pca}}\) is governed by where the clean target-relevant directions
are located after contamination.

\subsection{Heavy-tailed \(t_5\) experiments}
\label{supp:sec_t5_additional_numerics}

We also repeated the main simulation designs under heavy-tailed sampling. In these
experiments, the Gaussian observations in the corresponding designs are replaced by
multivariate \(t_5\) observations, scaled so that the covariance matrix is the same as in
the Gaussian experiment. Thus the covariance structures, target subspaces, adaptive
geometries, RWPI calibration level \(\alpha=0.10\), and path-based implementation are kept
unchanged.

For the same-distribution experiment, we report the \(t_5\) analogue of the 
design used in the main text. For the covariance-shift and
training-contamination experiments, we also include two robust PCA baselines. The first is
MCD-based covariance PCA, obtained by computing a robust scatter estimate and taking its
leading eigenvectors. The second is ROBPCA, the robust PCA method of \citet{hubert2005new}, implemented in the
R package \texttt{rrcov} through the function \texttt{PcaHubert}. These
baselines are included as heavy-tail and contamination diagnostics. They should not be
interpreted as targeting exactly the same robustness notion as adaptive DRO-PCA: robust
PCA baselines seek robust scatter or robust subspace estimation, whereas adaptive DRO-PCA
regularises reconstruction through a calibrated Wasserstein neighbourhood and a
data-adaptive transport geometry.

Each reported entry is based on 1000 Monte Carlo replications. Gains are percentage
reductions in reconstruction risk relative to ordinary PCA, with Monte Carlo standard
errors in parentheses. Win rates are the proportions of replications in which the method
has smaller reconstruction risk than ordinary PCA.

\begin{table}[t]
\centering
\caption{Heavy-tailed \(t_5\) same-distribution out-of-sample reconstruction. Training and
test samples are independent and have the same scaled \(t_5\) distribution.}
\label{tab:supp_t5_same_distribution}
\small
\begin{tabular}{llcccc}
\toprule
 Method & Test risk & Gain over PCA (\%) & Win rate & Projector distance \\
\midrule
 PCA & 33.18 (0.03) & -- & -- & 1.31 \\
\(\widehat G_n^{\mathrm{res}}\) & 33.25 (0.03) & -0.24 (0.05) & 0.28 & 1.32 \\
 \(\widehat G_n^{\mathrm{pca}}\) & 31.80 (0.03) & 4.15 (0.06) & 0.99 & 0.64 \\

\bottomrule
\end{tabular}
\end{table}

Table~\ref{tab:supp_t5_same_distribution} shows that the PCA-block geometry remains the
more favourable choice in the same-distribution heavy-tailed setting. This is consistent
with the interpretation in the main text: when the evaluation distribution coincides with
the training distribution, the geometry that regularises directions visible in the
empirical PCA block can improve finite-sample reconstruction and can move the fitted
subspace closer to the population oracle.

\begin{table}[p]
\centering
\caption{Heavy-tailed \(t_5\) covariance-shift experiment with robust PCA baselines. The
source sample is drawn from a scaled \(t_5\) distribution with the same source covariance
as in the Gaussian visibility design. Entries report gain over PCA in percentage points,
with Monte Carlo standard errors in parentheses, and win rates.}
\label{tab:supp_t5_visibility_shift_robust}
\scriptsize
\resizebox{\textwidth}{!}{%
\begin{tabular}{ccrrrrrrrr}
\toprule
Visibility \(\ell\) & Shift \(\eta\)
& \multicolumn{2}{c}{MCD-PCA}
& \multicolumn{2}{c}{ROBPCA}
& \multicolumn{2}{c}{DRO, \(\widehat G_n^{\mathrm{res}}\)}
& \multicolumn{2}{c}{DRO, \(\widehat G_n^{\mathrm{pca}}\)} \\
\cmidrule(lr){3-4}\cmidrule(lr){5-6}\cmidrule(lr){7-8}\cmidrule(lr){9-10}
 & & Gain & Win & Gain & Win & Gain & Win & Gain & Win \\
\midrule
0.00 & 0.00 & 0.20 (0.06) & 0.53 & 0.96 (0.06) & 0.70 & -1.62 (0.06) & 0.12 & -0.83 (0.04) & 0.23 \\
0.00 & 0.50 & 0.38 (0.06) & 0.56 & 0.84 (0.05) & 0.67 & 2.26 (0.07) & 0.93 & 1.14 (0.05) & 0.86 \\
0.00 & 1.00 & 0.48 (0.11) & 0.56 & 0.71 (0.10) & 0.59 & 5.26 (0.15) & 0.94 & 2.71 (0.10) & 0.89 \\
0.00 & 1.50 & 0.54 (0.15) & 0.56 & 0.58 (0.14) & 0.56 & 7.65 (0.22) & 0.94 & 4.00 (0.14) & 0.91 \\
0.05 & 0.00 & -0.32 (0.10) & 0.46 & 1.04 (0.10) & 0.64 & -0.60 (0.06) & 0.40 & -1.25 (0.05) & 0.20 \\
0.05 & 0.50 & -0.26 (0.10) & 0.46 & 0.87 (0.09) & 0.58 & 2.45 (0.08) & 0.89 & 1.59 (0.06) & 0.87 \\
0.05 & 1.00 & -0.26 (0.14) & 0.48 & 0.70 (0.12) & 0.57 & 4.82 (0.14) & 0.92 & 3.87 (0.10) & 0.95 \\
0.05 & 1.50 & -0.28 (0.18) & 0.48 & 0.54 (0.16) & 0.54 & 6.71 (0.20) & 0.93 & 5.75 (0.14) & 0.96 \\
0.15 & 0.00 & -0.36 (0.15) & 0.47 & 1.82 (0.13) & 0.70 & 0.36 (0.06) & 0.72 & -1.18 (0.07) & 0.31 \\
0.15 & 0.50 & -0.32 (0.16) & 0.47 & 1.89 (0.14) & 0.68 & 1.66 (0.10) & 0.79 & 2.94 (0.08) & 0.92 \\
0.15 & 1.00 & -0.32 (0.19) & 0.47 & 1.92 (0.16) & 0.66 & 2.68 (0.14) & 0.79 & 6.32 (0.12) & 0.99 \\
0.15 & 1.50 & -0.35 (0.22) & 0.47 & 1.92 (0.19) & 0.64 & 3.51 (0.18) & 0.78 & 9.14 (0.16) & 0.99 \\
0.30 & 0.00 & -0.35 (0.14) & 0.42 & 1.91 (0.12) & 0.72 & 0.13 (0.06) & 0.56 & -0.82 (0.09) & 0.38 \\
0.30 & 0.50 & -0.39 (0.16) & 0.44 & 2.10 (0.13) & 0.70 & 0.05 (0.08) & 0.37 & 4.54 (0.09) & 0.97 \\
0.30 & 1.00 & -0.45 (0.19) & 0.45 & 2.24 (0.16) & 0.68 & -0.03 (0.11) & 0.34 & 9.06 (0.11) & 1.00 \\
0.30 & 1.50 & -0.53 (0.22) & 0.45 & 2.34 (0.18) & 0.66 & -0.10 (0.13) & 0.33 & 12.92 (0.14) & 1.00 \\
0.50 & 0.00 & -0.21 (0.12) & 0.44 & 1.90 (0.10) & 0.75 & -0.47 (0.05) & 0.05 & -0.15 (0.09) & 0.46 \\
0.50 & 0.50 & -0.29 (0.15) & 0.43 & 2.16 (0.12) & 0.73 & -0.92 (0.07) & 0.03 & 4.79 (0.09) & 0.98 \\
0.50 & 1.00 & -0.39 (0.18) & 0.43 & 2.37 (0.15) & 0.71 & -1.31 (0.09) & 0.03 & 9.11 (0.10) & 1.00 \\
0.50 & 1.50 & -0.48 (0.20) & 0.44 & 2.54 (0.17) & 0.70 & -1.65 (0.11) & 0.03 & 12.93 (0.11) & 1.00 \\
\bottomrule
\end{tabular}}
\end{table}

Table~\ref{tab:supp_t5_visibility_shift_robust} shows that the geometry-dependent pattern
observed in the Gaussian covariance-shift experiment persists under heavy-tailed sampling.
The residual geometry is favourable when the target-relevant protected directions are only
weakly visible in the source PCA block, while the PCA-block geometry becomes stronger as
those directions become more visible. The robust PCA baselines provide useful protection
against heavy tails, but they do not reproduce the same shift-adaptive pattern: their gains
are comparatively stable across shift levels, whereas the DRO gains change with the
alignment between the target shift and the adaptive geometry.

\begin{table}[p]
\centering
\caption{Heavy-tailed \(t_5\) training-contamination experiment with robust PCA baselines.
Clean observations are drawn from the scaled \(t_5\) distribution, and a proportion
\(\epsilon\) of the training sample is contaminated along a subspace with overlap
\(\kappa\) with the clean population subspace. Entries report gain over PCA in percentage
points, with Monte Carlo standard errors in parentheses, and win rates.}
\label{tab:supp_t5_contamination_robust}
\scriptsize
\resizebox{\textwidth}{!}{%
\begin{tabular}{ccrrrrrrrr}
\toprule
Overlap \(\kappa\) & Contamination \(\epsilon\)
& \multicolumn{2}{c}{MCD-PCA}
& \multicolumn{2}{c}{ROBPCA}
& \multicolumn{2}{c}{DRO, \(\widehat G_n^{\mathrm{res}}\)}
& \multicolumn{2}{c}{DRO, \(\widehat G_n^{\mathrm{pca}}\)} \\
\cmidrule(lr){3-4}\cmidrule(lr){5-6}\cmidrule(lr){7-8}\cmidrule(lr){9-10}
 & & Gain & Win & Gain & Win & Gain & Win & Gain & Win \\
\midrule
0.00 & 0.00 & 0.03 (0.11) & 0.45 & 2.12 (0.10) & 0.78 & -1.14 (0.07) & 0.01 & 5.44 (0.08) & 1.00 \\
0.00 & 0.05 & 17.70 (0.22) & 0.98 & 19.19 (0.21) & 1.00 & 1.82 (0.21) & 0.56 & 9.54 (0.24) & 0.99 \\
0.00 & 0.10 & 27.79 (0.17) & 1.00 & 28.97 (0.16) & 1.00 & 8.40 (0.26) & 0.87 & 4.97 (0.19) & 0.91 \\
0.00 & 0.15 & 33.54 (0.14) & 1.00 & 34.38 (0.13) & 1.00 & 8.42 (0.26) & 0.94 & 1.49 (0.11) & 0.74 \\
0.00 & 0.20 & 37.55 (0.13) & 1.00 & 37.99 (0.13) & 1.00 & 4.94 (0.19) & 0.95 & -0.07 (0.07) & 0.39 \\
0.10 & 0.00 & -0.21 (0.11) & 0.41 & 1.83 (0.10) & 0.76 & -1.13 (0.07) & 0.01 & 5.28 (0.08) & 1.00 \\
0.10 & 0.05 & 14.79 (0.21) & 0.98 & 16.41 (0.20) & 1.00 & 0.51 (0.17) & 0.45 & 12.00 (0.20) & 1.00 \\
0.10 & 0.10 & 23.73 (0.16) & 1.00 & 24.97 (0.16) & 1.00 & 4.39 (0.19) & 0.82 & 8.28 (0.20) & 0.98 \\
0.10 & 0.15 & 29.28 (0.14) & 1.00 & 30.27 (0.13) & 1.00 & 6.63 (0.17) & 0.96 & 4.25 (0.15) & 0.95 \\
0.10 & 0.20 & 33.23 (0.11) & 1.00 & 33.59 (0.12) & 1.00 & 5.93 (0.14) & 0.99 & 1.65 (0.08) & 0.82 \\
0.30 & 0.00 & -0.25 (0.11) & 0.42 & 1.83 (0.09) & 0.76 & -1.09 (0.07) & 0.01 & 5.31 (0.07) & 1.00 \\
0.30 & 0.05 & 8.91 (0.16) & 0.97 & 10.54 (0.15) & 1.00 & -1.54 (0.12) & 0.09 & 11.84 (0.13) & 1.00 \\
0.30 & 0.10 & 15.85 (0.14) & 1.00 & 17.18 (0.14) & 1.00 & -0.45 (0.10) & 0.30 & 12.40 (0.14) & 1.00 \\
0.30 & 0.15 & 20.91 (0.12) & 1.00 & 21.97 (0.11) & 1.00 & 0.83 (0.08) & 0.70 & 9.02 (0.16) & 1.00 \\
0.30 & 0.20 & 23.99 (0.10) & 1.00 & 24.18 (0.11) & 1.00 & 1.61 (0.07) & 0.90 & 6.10 (0.13) & 1.00 \\
0.60 & 0.00 & -0.14 (0.11) & 0.42 & 1.90 (0.10) & 0.76 & -1.01 (0.06) & 0.01 & 5.42 (0.08) & 1.00 \\
0.60 & 0.05 & 3.01 (0.11) & 0.84 & 4.85 (0.10) & 0.97 & -1.32 (0.08) & 0.02 & 7.51 (0.08) & 1.00 \\
0.60 & 0.10 & 6.48 (0.10) & 0.98 & 7.87 (0.09) & 1.00 & -0.99 (0.05) & 0.00 & 9.46 (0.08) & 1.00 \\
0.60 & 0.15 & 9.17 (0.09) & 1.00 & 10.18 (0.09) & 1.00 & -0.92 (0.05) & 0.01 & 10.21 (0.07) & 1.00 \\
0.60 & 0.20 & 10.85 (0.09) & 1.00 & 10.80 (0.09) & 1.00 & -0.71 (0.03) & 0.00 & 9.56 (0.07) & 1.00 \\
1.00 & 0.00 & -0.06 (0.12) & 0.45 & 2.01 (0.10) & 0.78 & -1.02 (0.06) & 0.01 & 5.45 (0.08) & 1.00 \\
1.00 & 0.05 & -2.84 (0.10) & 0.13 & 0.34 (0.08) & 0.51 & -0.64 (0.06) & 0.01 & 2.96 (0.06) & 1.00 \\
1.00 & 0.10 & -3.66 (0.09) & 0.06 & 0.01 (0.07) & 0.44 & -0.34 (0.04) & 0.01 & 1.80 (0.05) & 1.00 \\
1.00 & 0.15 & -4.10 (0.08) & 0.02 & -0.28 (0.05) & 0.35 & -0.19 (0.02) & 0.00 & 1.05 (0.03) & 1.00 \\
1.00 & 0.20 & -3.55 (0.08) & 0.03 & -0.06 (0.03) & 0.46 & -0.16 (0.02) & 0.00 & 0.73 (0.02) & 1.00 \\
\bottomrule
\end{tabular}}
\end{table}
\enlargethispage{2\baselineskip}
Table~\ref{tab:supp_t5_contamination_robust} shows that robust scatter methods are very
strong in the direct training-contamination setting, as expected. This comparison is useful
because it separates two robustness mechanisms. MCD-PCA and ROBPCA primarily target
outlier-resistant subspace estimation. Adaptive DRO-PCA is instead a geometry-calibrated
regularisation of reconstruction. The DRO procedures remain competitive in several
contamination configurations, especially for \(\widehat G_n^{\mathrm{pca}}\) when the
contaminating directions have non-negligible overlap with the clean signal, but the table
also shows that classical robust PCA can be preferable when the contamination is the
dominant difficulty. This supports the interpretation that adaptive DRO-PCA is not a
universal replacement for robust PCA; it is a complementary method whose benefit depends
on the alignment between the adaptive transport geometry and the perturbation structure.

\end{document}